\documentclass[12pt,reqno]{amsart}
\usepackage{graphicx,indentfirst, caption}
\usepackage{amsmath,amssymb,mathrsfs}
\usepackage{amsthm,amscd}
\usepackage{verbatim}
\usepackage{appendix}
\usepackage{enumerate}
\usepackage{enumitem,titletoc}

\usepackage[utf8]{inputenc}

\usepackage{fancyhdr}
\usepackage{amsfonts,color}
\usepackage[all]{xy}
\usepackage{tikz-cd}
\usepackage{syntonly}
\usepackage{float}
\usepackage{pgfplots}
\pgfplotsset{compat=1.18}
\usepackage{array}
\usepackage[normalem]{ulem}
\usepackage[left=2.5cm,right=2.5cm,bottom=2cm,top=2.1cm]{geometry}

\usepackage{tikz}
\usetikzlibrary{calc}
\usepackage{extarrows}
\usepackage{hyperref}
\usepackage{cancel}

\def\XXint#1#2#3{{\setbox0=\hbox{$#1{#2#3}{\int}$ }
		\vcenter{\hbox{$#2#3$ }}\kern-.6\wd0}}

\theoremstyle{theorem}

\newtheorem{theorem}{Theorem}[section]

\newtheorem{lem}[theorem]{Lemma}
\newtheorem{prop}[theorem]{Proposition}
\newtheorem{cor}[theorem]{Corollary}
\newtheorem{ques}[theorem]{Question}
\newtheorem{example}[theorem]{Example}

\newtheorem{conj}[theorem]{Conjecture}
\newtheorem{remark}[theorem]{Remark}

\theoremstyle{definition}
\newtheorem{defn}[theorem]{Definition}

\newcommand{\del}{\partial}

\newcommand{\R}{\mathbb R}

\newcommand{\detla}{\delta}

\makeindex

\definecolor{ForestGreen}{RGB}{34,139,34}
\hypersetup{
	colorlinks=true,
	citecolor=ForestGreen,
	filecolor=ForestGreen,
	linkcolor=blue,
	urlcolor=black
}

\numberwithin{equation}{section}

\usetikzlibrary{intersections}
\usepgfplotslibrary{fillbetween}
\usetikzlibrary{patterns,patterns.meta}

\usepackage{mathtools}
\mathtoolsset{showonlyrefs}

\author{Tristan C. Collins}
\author{Benjy Firester}
\email{\href{mailto:benjyfir@mit.edu}{benjyfir@mit.edu}}
\address{Department of Mathematics, Massachusetts Institute of Technology, 77 Massachusetts Ave., Cambridge, MA, USA}
\author{Freid Tong}
\email{\href{mailto:tristanc@math.toronto.edu}{tristanc@math.toronto.edu}}
\email{\href{mailto:freid.tong@utoronto.ca}{freid.tong@utoronto.ca}}
\address{Department of Mathematics, University of Toronto, 40 St. George Street, Toronto, ON, Canada}
\date{\today}

\newcommand{\p}{\partial}

\newcommand{\ve}{\varepsilon}

\newcommand{\la}{\langle}
\newcommand{\rg}{\rangle}

\newcommand{\mr}[1]{{\rm #1}}

\usepackage{wrapfig,caption}
\newcommand{\cA}{\mathcal{A}}

\newcommand{\cF}{\mathcal{F}}
\newcommand{\cG}{\mathcal{G}}\newcommand{\cH}{\mathcal{H}}

\newcommand{\cL}{\mathcal{L}}
\newcommand{\cM}{\mathcal{M}}
\newcommand{\cO}{\mathcal{O}}
\newcommand{\cQ}{\mathcal{Q}}

\newcommand{\cX}{\mathcal{X}}

\newcommand{\bB}{\mathbb{B}}

\newcommand{\bP}{\mathbb{P}}
\newcommand{\bR}{\mathbb{R}}
\newcommand{\bS}{\mathbb{S}}

\newcommand{\bZ}{\mathbb{Z}}

\newcommand{\sln}{\textup{SL}(n)}
\newcommand{\son}{\textup{SO}(n)}
\newcommand{\tC}{\mathtt{C}}

\newcommand{\ol}{\overline}
\newcommand{\ul}{\underline}
\usepackage{accents}

\title[Optimal regularity of optimal transport]{Optimal regularity of optimal transport and degenerations of convex cones}

\begin{document}

\begin{abstract}
 If $T$ is an optimal transport map between bounded convex domains $\Omega$ and $\Omega'$, we characterize the regularity/singularity dichotomy for $DT$ at $(x,T(x))\in\partial\Omega \times \partial \Omega'$. This is obtained through a new approach to the regularity theory based on affine degenerations of the tangent cones to $(\Omega, \Omega')$ at $(x,T(x))$. We formulate a general principle relating the optimal boundary regularity to a stability criterion and an $\mathrm{SL}(n)$ moduli space of pairs of convex cones. Among other results, we establish geometric criteria for existence and non-existence of homogeneous optimal transport maps between convex cones.
\end{abstract}
\maketitle
\vspace{-1cm}
\tableofcontents

\vspace{-1cm}
\section{Introduction}
We develop a new approach to the optimal boundary regularity of optimal transport maps between convex domains.
Central to our approach is a connection between the regularity/singularity dichotomy for optimal transport maps and the $\sln$-moduli space of pairs of convex cones $(\tC,\tC')\subset V\times V^*$ for $V$ a finite dimensional vector space.

We recall the basic setting for regularity theory of optimal transport with the quadratic cost.
Let $\Omega, \Omega'\subset \mathbb{R}^n$ be bounded, convex domains.
Equip $\Omega$ (resp.~$\Omega'$) with a probability measure $d\mu= g(x)dx$ (resp.~$ d\nu = g'(y)dy$).
The optimal transport problem for the quadratic cost seeks a map $T: \Omega \rightarrow \Omega'$ such that $T_{\#}d\mu = d\nu$ and $T$ minimizes the cost
\[
c(T) := -\int_{x\in\Omega} \la x, T(x)\rg \, d\mu.
\]
Optimal transport for the quadratic cost has applications spanning economics, engineering, and fluid mechanics, in addition to deep connections with geometry, functional and geometric inequalities, and probability~\cite{DPF, Evans, RR, Villani, Villani2, AmbCaffBren, BLMR, MaggiFigalli, Brendle, BrendleEichmair}.

Foundational results of Brenier~\cite{Brenier} and Gangbo-McCann~\cite{Gangbo-McCann} establish the existence and uniqueness of weak solutions of the classical optimal transport formulation, and show that, for the quadratic cost $T=\nabla u$, where $u$ is a convex potential function solving a Monge-Amp\`ere equation.
Landmark work of Caffarelli~\cite{Caffarelli, Caffarelli2, Caffarelli3} developed the regularity theory for optimal maps.
The convexity of $\Omega,\Omega'$ is necessary for the optimal map $T$ to be continuous \cite{Caffarelli}. Caffarelli proved that if $\Omega, \Omega'$ are bounded convex domains, and $C^{-1} < g,g'< C$, then $u$ is globally $C^{1,\delta}$-continuous for some $\delta\in (0, 1)$~\cite{Caffarelli, Caffarelli2}.
In the interior, $u\in W^{2,p}_{\mr{loc}}(\Omega)$ and is strictly convex \cite{Caffarelli4}.
Furthermore, $u$ is an Alexandrov solution of the Monge-Amp\`ere equation
\begin{equation}\label{eqn:OT}\tag{OT}
    \det D^2 u = \frac{g(x)}{g'(\nabla u)}\qquad \text{and}\qquad
    \nabla u(\Omega) = \Omega'.
\end{equation}

For more regular densities, higher order interior regularity follows from the interior regularity theory of the real Monge-Amp\`ere equation \cite{Caffarelli}.  While this theory provides a satisfactory resolution of the interior regularity theory, much less is known about the boundary regularity of solutions to~\eqref{eqn:OT}, beyond Caffarelli's $C^{1, \delta}$-estimate, for general convex domains.  In this generality, the optimal global regularity is $u\in C^{1,1-\epsilon}(\overline{\Omega}) \cap W^{2,p}(\overline{\Omega})$ for all $\epsilon <1$ and for all $p \geq 1$.
This was established:
\begin{itemize}
    \item for arbitrary convex $\Omega, \Omega'\subset \mathbb{R}^2$ and $g=g'=1$ by Savin-Yu \cite{Savin-Yu};
    \item for arbitrary convex $\Omega,\Omega'\subset \mathbb{R}^n$ and $g,g'$ Dini continuous and uniformly bounded away from zero by Collins-Tong \cite{TristanFreid};
    \item for arbitrary convex $\Omega,\Omega'\subset \mathbb{R}^n$ and continuous $g,g'$ uniformly bounded away from zero by Chen-Li-Wang~\cite{Chen-Li-Wang} making use of the monotonicity formula in \cite{TristanFreid}.
\end{itemize}

Currently, pushing beyond global $C^{1,1-\epsilon}$ and $W^{2,p}$-regularity requires imposing additional assumptions on the domains $\Omega, \Omega'$ and the densities $g,g'$. For example, global $C^{2,\alpha}$-regularity results under various assumptions have been established by Delanoe~\cite{Delanoe} (in dimension $2$), Urbas~\cite{Urbas}, Caffarelli \cite{Caffarelli3}, Chen-Liu-Wang \cite{Chen-Liu-Wang}, and Collins-Tong \cite{TristanFreid}.  The latter result \cite{TristanFreid} establishes global $C^{2,\alpha}$-regularity assuming that $g,g'$ are $C^{\alpha}$ and $\Omega,\Omega'$ are $C^{1,\alpha}$-bounded; to our knowledge these are the weakest assumptions under which higher regularity has been established.

In this paper we develop a new approach for establishing $C^{1,1}$-boundary regularity for general convex domains.  Our methods characterize the regularity/singularity dichotomy for $D^2u$ at the boundary.  Results of this type lead, for example, to improved convergence rate estimates for numerical approximations of optimal transport maps between general (e.g., polyhedral) convex domains \cite{Berman}.

Suppose $\Omega, \Omega'$ are convex domains and, for simplicity, assume $g=g'=1$.
Without loss of generality assume $0 \in \del \Omega \cap \del \Omega'$ and that the solution $u$ of~\eqref{eqn:OT} satisfies $u(0)=0$ and $\nabla u(0)=0$.
We examine the following question:

\begin{ques}\label{ques: toC11orNot} Can the existence or failure of pointwise $C^{1,1}$-regularity of $u$ be characterized purely in terms of the geometry of $\Omega$ and $\Omega'$?
\end{ques}

Let $\mathtt{C}$ denote the tangent cone of $\Omega$ at $0$, and let $\mathtt{C}'$ denote the tangent cone of $\Omega'$ at $0$.  The main observation is that, thanks to the monotonicity formula of \cite{TristanFreid}, the failure of pointwise $C^{1,1}$-regularity of $u$ at $0$ implies the existence of a sequence of matrices $M_j \in \sln$ with $\|M_j\| \rightarrow +\infty$ such that
\[
(M_j\cdot \tC, M_j^{-T}\cdot \tC')\leadsto (\tC_{\infty}, \tC_{\infty}'),
\]
where $(\tC_{\infty}, \tC_{\infty}')$ are convex cones admitting a $2$-homogeneous convex function $\varphi:\mathtt{C}_{\infty} \rightarrow \mathbb{R}_{>0}$ solving the homogeneous optimal transport (H.O.T.) problem
\begin{equation}\label{eq: hotIntro}
\det D^2\varphi(X)=1,  \qquad \nabla \varphi(\tC_{\infty})=\tC'_{\infty}.
\end{equation}
Suppose that $\Omega$ and $\Omega'$ are locally $C^{1,\alpha}$-epigraphs over their tangent cones $(\tC, \tC')$ and that the $\sln$ orbit of $(\tC, \tC')$ is locally closed in the Hausdorff topology.
If the tangent cones $(\tC,\tC')$ have compact automorphism group, then necessarily the pair $(\tC_{\infty}, \tC_{\infty}')$ is not $\sln$ equivalent to $(\tC,\tC')$. 
In particular, if such a degeneration {\em cannot occur}, then it follows that $u$ is $C^{1,1}$ at $0$. This motivates the following definition:

\begin{defn}\label{defn: OTSeparableIntro}
Let $(\mathtt{C},\mathtt{C}')$ be a pair of convex cones.
We say that $(\mathtt{C},\mathtt{C}')$ is \textbf{OT-separable} if the following holds: if there exists an $\sln$-degeneration $(\tC,\tC')\leadsto (\tC_\infty, \tC'_\infty)$ such that $(\tC_\infty, \tC'_\infty)$ admits a homogeneous optimal transport map, then $(\tC_\infty, \tC'_\infty) = M \cdot(\tC, \tC')$ for some $M \in  \sln$.
\end{defn}

With this motivation, it is clear that understanding when a particular pair of cones satisfies the property of OT-separability is essential to establishing optimal local regularity results for~\eqref{eqn:OT}.
In this paper, we establish OT-separability for several broad classes of cones.
In particular, this yields pointwise $C^{1,1}$-regularity.
What is perhaps surprising about these results is that the $C^{1,1}$-regularity is obtained without directly analyzing the Monge-Amp\`ere equation underlying the optimal transport equation.
As an example of our techniques, we obtain the following:

\begin{theorem}\label{thm: C11CrossingIntro}
    Let $(\Omega, \Omega')$ be bounded convex domains in $\bR^n$ equipped with probability measures $(g(x)dx, g'(y)dy)$ where $C^{-1}\leq g,g'\leq C$ are Dini continuous up to the boundary.  
    Suppose that $0\in \del\Omega  \cap \del\Omega'$, and let $(\tC,\tC')$.
    Denote the tangent cones of $(\Omega, \Omega')$ at $0$, and suppose $(\Omega, \Omega')$ are locally $C^{1,\alpha}$-epigraphs over their tangent cones. 
    Suppose that:
    \begin{itemize}
        \item[(i)] $(\tC,\tC')$ is a pair of pointed cones with compact automorphism group, and its $\sln$-orbit is locally closed in the Hausdorff topology,
        \item[(ii)] if $\tC^{\vee}$ denotes the dual cone of $\tC$, then for any non-zero $p\in \del\tC' \cap \del \tC^{\vee}$, the convex cones $\tC'$ and $\tC^{\vee}$ do not share a supporting hyperplane at $p$.
    \end{itemize}
    If $u$ solves the optimal transport problem~\eqref{eqn:OT}, and $\nabla u(0)=0$, then $u$ is $C^{1,1}$ at $0$.
\end{theorem}

Theorem~\ref{thm: C11CrossingIntro} suggests that $C^{1,1}$-regularity is  ``generic" at corners.  Conversely, we can establish the failure of $C^{1,1}$-regularity in stratified settings.

\begin{theorem}\label{thm: introSplitting}
    In the optimal transport setting $(\Omega, \Omega', g(x)dx, g'(y)dy)$ of Theorem~\ref{thm: C11CrossingIntro}, suppose that $\tC$ splits orthogonally as $ \tC= \bR^k\times \widehat{\tC}$ for a pointed cone $\widehat{\tC}\subset \bR^{n-k}$. If $u$ solves the optimal transport problem~\eqref{eqn:OT}, and $\nabla u(0)=0$, then $u$ is not $C^{1,1}$ at $0$ unless $\tC'$ splits as $\tC'= \bR^k\times \widehat{\tC}'$, where $\widehat{\tC}'\subset \bR^{n-k}$ is a pointed cone and the lower dimensional cone pair $(\widehat{\tC},\widehat{\tC}')$ admits a $2$-homogeneous optimal transport map and the linear subspaces of $\tC$ and $\tC'$ pair perfectly.
\end{theorem}

See Theorem~\ref{thm:splitting} and the subsequent discussion for the precise and affine invariant formulation of Theorem~\ref{thm: introSplitting}.  An immediate corollary of this result is the following:

\begin{cor}
 Let $(\Omega, \Omega')$ be bounded convex domains in $\bR^n$ equipped with probability measures $(g(x)dx, g'(y)dy)$, where $g,g'$ are strictly bounded below by a positive constant and Dini continuous up to the boundary.  Suppose that $u$ solves~\eqref{eqn:OT}.  If $u \in C^{1,1}(\ol{\Omega})$, then $\nabla u$ defines an isomorphism of skeletal filtrations $\mr{Sk}_*(\Omega) \cong \mr{Sk}_*(\Omega')$. 
\end{cor}
The skeletal filtration $\mr{Sk}_k(\Omega)$ consists of all the points $x \in \ol{\Omega}$ such that the tangent cone $\mr{Tan}_x \Omega$ does not split off $k + 1$ lines; see Section~\ref{sec:obstructions} for more details.
For example, this result implies that optimal transport maps between polyhedra $\Omega, \Omega'$ with a distinct number of vertices (or edges, faces, etc.) can never be globally $C^{1,1}$.  However, the conclusion that the skeletal isomorphism is induced by the gradient of a convex function implies stronger obstructions; see Figure~\ref{fig:ovidiu} for an example of polyhedra with isomorphic skeletal filtrations that cannot admit a globally $C^{1,1}$-optimal transport map.

Theorem~\ref{thm: C11CrossingIntro} and Theorem~\ref{thm: introSplitting} are obtained as part of a larger package of results aimed at establishing OT-separability.
We take two approaches to this. 
The first approach is to directly rule out the possibility of $\sln$-degenerations $(\tC,\tC')\leadsto(\tC_{\infty},\tC_{\infty}')$ where $(\tC_{\infty},\tC_{\infty}')$ admits a homogeneous optimal transport map.
This is done by establishing volume estimates for a Mahler type volume associated to the degeneration; see Section~\ref{sec:OTsep}.
The second approach is to find explicit obstructions ruling out optimal transport maps between the limiting cones $(\tC_{\infty},\tC_{\infty}')$; see Section~\ref{sec:obstructions}.
We expect that, using the methods of this paper, OT-separability is effectively checkable for broad classes of cone pairs, such as polyhedral convex cones.

The monotonicity formula of \cite{TristanFreid} implies that, when $u$ solves~\eqref{eqn:OT} and $C^{1,1}$-regularity holds at $0\in\del\Omega \cap \del\Omega'$, the tangent cones $(\tC,\tC')$ admit a $2$-homogeneous optimal transport map.  In particular, the existence of such a map is a necessary condition for $C^{1,1}$-regularity at $0$. We prove the following result:

\begin{theorem}\label{thm: HotExistIntro}
    Suppose $(\tC,\tC')$ are pointed convex cones satisfying
    \begin{itemize}
        \item[$(i)$] $\overline{\tC}\cap \overline{(-\tC')^{\vee}}=\{0\}$ and $\overline{-\tC^{\vee}}\cap \overline{\tC'}=\{0\}$;
        \item[$(ii)$] either $\tC' \subset \tC^\vee$, or $\widetilde{H}_{*}({\tC'}\setminus \tC^\vee) \neq 0$, where $\widetilde{H}_*$ denotes the reduced homology; and
        \item[$(iii)$] the cones $\tC^\vee$ and $\tC'$ do not share a supporting hyperplane that intersects $\p \tC^\vee \cap \p \tC'$.
    \end{itemize}
    Then, there is a convex function $\varphi:\tC\rightarrow \mathbb{R}_{>0}$ solving the homogeneous optimal transport (H.O.T.) equation
    \begin{equation}\label{eqn:HOT}\tag{HOT}
    \det D^2 \varphi = 1,\qquad \varphi(tX) = t^2\varphi(X) \quad \text{for all } t\in\bR_{>0}, \qquad \text{and}\qquad \nabla \varphi (\tC) = \tC'.
    \end{equation}
\end{theorem}

Theorem~\ref{thm: HotExistIntro} is a special case of a more general result (Theorem~\ref{prop:WidthAchieved}) in which condition $(ii)$ is replaced with a more general topological result and condition $(iii)$ is replaced by $OT$-separability. In Theorem~\ref{thm:polystableWidthFinite}, we prove an equivariant version that applies whenever the pairs $(\tC,\tC')$ admit the action of a reductive group.  
    
    Let us explain the assumptions appearing in Theorem~\ref{thm: HotExistIntro}. Condition $(i)$ is a necessary structural condition that there exist homogeneous convex functions on $\tC$ with gradient image $\tC'$.
This is automatically satisfied for any pair of tangent cones arising from optimal transport at $(x, T(x))$.
Condition $(ii)$ is a \textit{topological linking} condition for the cones.
Such a condition is needed since OT-separability can hold for cones without homogeneous optimal transport maps.
See Section~\ref{sec:Linking} and Figure~\ref{fig:acute/obtuse/maxmin} to see this linking property and Section~\ref{sec:polystable} for the equivariant version. 
Indeed, there are examples, (already occurring in dimension $2$), of cones satisfying $(i)$, but not satisfying $(ii)$, for which there are not even local solutions of the optimal transport problem.
That is, one cannot find open sets $N, N'$ containing the origin such that there is a convex function $u:N\cap \tC\rightarrow \bR$ solving the local optimal transport problem
\[
\det D^2u=1, \qquad \nabla u(N\cap \tC)=N'\cap \tC',\qquad \nabla u(0)=0;
\]
see, for example, Section~\ref{sec:modulin=2}.
The proof of Theorem~\ref{thm: HotExistIntro} involves a min-max argument.
Condition $(ii)$ is used to construct appropriate sweepouts, while the OT-separability condition is shown to imply a weak version of the Palais-Smale condition.
Condition $(iii)$ will imply the pair is OT-separable.

The assumptions of Theorem~\ref{thm: HotExistIntro} are sharp in the following sense.
In dimension $2$, Theorem~\ref{thm: HotExistIntro} recovers the classification of cones admitting homogeneous optimal transport maps \cite{TristanFreid}.
In higher dimensions, we show via example that without further structural assumptions, if any condition is omitted, then there exists a pair of cones satisfying the remaining conditions which cannot admit a homogeneous optimal transport map; see Section~\ref{sec:obstructions}. 

It may come as a surprise that the existence of homogeneous optimal transport maps requires min-max techniques.
Indeed, for optimal transport maps between bounded convex domains, the optimal transport problem can be formulated as a convex minimization problem~\cite{Kantorovich,Brenier, Gangbo-McCann}.
This is no longer true for homogeneous optimal transport maps between domains of infinite volume.
Indeed, in Section~\ref{sec:notConvex} we use perturbative methods to construct pairs of cones with trivial automorphism group and two distinct, isolated homogeneous optimal transport maps.

There is a connection between the approach to regularity developed here, particularly the notion of OT-separability, and a putative, Hausdorff $\sln$-moduli space of convex cones.
We make this connection explicit by formulating a general conjecture relating the existence of local and/or homogeneous optimal transport maps to a version of symplectic reduction.
We formulate notions of stability for pairs of cones and propose conjectures relating stability to regularity for optimal transport maps.
In dimension $2$, we construct the full moduli space and prove most of these conjectures.

We expect these results to have broad implications in understanding collapsing of Calabi-Yau metrics. In the context of the Strominger-Yau-Zaslow picture for mirror symmetry \cite{SYZ}, recent works \cite{Li-Fermat, HJMM, AHJMM, JMPS, Li-SYZ} have shown that optimal transport maps between convex polytopes with their respective Lebesgue densities arise naturally as adiabatic limits for large complex structure degeneration of Calabi-Yau metrics, and the boundary regularity of such optimal transport maps is intimately related to bubbling behaviour of the Calabi-Yau metrics under degeneration. More generally, optimal transport maps between polytopes with degenerate densities \cite{Li-Intermediate} and homogeneous optimal transport maps between cones \cite{TristanFreidYau} have also been shown to arise in various contexts as adiabatic limits of Calabi-Yau metrics. We expect the methods developed in this paper to be useful in studying the behaviour of collapsing Calabi-Yau metrics in these more general settings as well. 

There is a formal similarity between the results obtained here and the local regularity theory of singular K\"ahler-Einstein metrics, specifically the conjecture of Donaldson-Sun \cite{DS2}, its resolution by Li-Wang-Xu \cite{LWX}, and recent work of Zhang \cite{Zhang}. 
These results can be viewed as part of a broader theme arising throughout K\"ahler geometry that the existence of solutions to geometric PDEs (usually describing \textit{canonical metrics}) is intimately tied to the construction of moduli spaces (e.g., of germs of klt $K$-stable singularities).
At a conceptual level, this is due to the fact that many such PDEs arise as zeroes of a moment map for an infinite dimensional group admitting a (formal) complexification.
In particular, by the Kempf-Ness theorem~\cite{KN}, these PDEs can also be described as critical points of a convex functional describing \textit{stable points} in some infinite dimensional version of Geometric Invariant Theory~\cite{GIT}; see e.g.,~\cite{DoMoment, DoMoment2} for general discussion. 

We emphasize that the existence of \hyperref[eqn:HOT]{H.O.T.} maps does not fit this framework. Indeed, the examples of non-uniqueness constructed in Section~\ref{sec:notConvex} show that the existence of \hyperref[eqn:HOT]{H.O.T.} maps cannot be formulated as a convex problem.
From this point of view the apparent connection, described in Section~\ref{sec:ModuliMoment}, between optimal transport and a moduli space of convex cones obtained through a generalized form of symplectic reduction is rather unexpected.

\subsection{Outline}
The outline of this paper is as follows:
In Section~\ref{sec:prelim}, we establish the general principle relating $\sln$-degenerations to the $C^{1,1}$-boundary regularity of optimal transport using the monotonicity formula. 
This strategy reveals that the OT-separability property of the tangent cones determines the $C^{1,1}$-regularity at the boundary without directly analyzing the optimal transport equation.
A \hyperref[eqn:HOT]{H.O.T.} map $\varphi$ between a pair of cones $(\tC, \tC')$ is a critical point of the functional
\[
\cM_{(\tC,\tC')}(\varphi) := \log \int_\tC e^{-\varphi}\,dX +\log \int_{\tC'}e^{-\varphi^*}\,dY
\]
which can be written alternatively as a conical Mahler type volume,
\[
\cM_{(\tC,\tC')}(\varphi) = \log(|\{\varphi < 1\} \cap \tC|) + \log(|\{\varphi < 1\}^\circ \cap \tC'|) + C(n).
\]
Recall that the standard Mahler volume of a set $K$, $\log(|K|)+\log(|K^\circ|)$, is uniformly lower bounded, while the Blaschke-Santal\'o inequality provides an upper bound provided $K$ is centered at the origin.
For the conical M\"ahler volume $\cM_{(\tC, \tC')}$, both of these properties fail in general.
Central to our strategy is identifying conditions under which these properties are restored, see Definitions~\ref{def:saturated} and~\ref{eqn:widthDef}.

In Section~\ref{sec:acute}, we prove Theorem~\ref{thm: HotExistIntro} for the case $\tC' \subset \tC^\vee$, generalizing~\cite[Theorem 1.1]{TristanBenjyFreid}. 
The assumption that $\tC' \subset \tC^\vee$ is shown to imply a uniform upper bound for $\cM_{(\tC,\tC')}$. 
With this estimate, OT-separability implies the required compactness to find a critical point.

We prove our main theorem in Section~\ref{sec:maxmin}, which is extended to the equivariant case in Section~\ref{sec:polystable}.
The linking property (see Theorem~\ref{thm: HotExistIntro} condition $(ii)$, or Definition~\ref{def:Linked}) provides a non-trivial homology class over which we can sweepout.
A key step in the proof is to show that OT-separability of $(\tC,\tC')$ implies a weak version of the Palais-Smale condition necessary to deduce the existence of a critical point.

In Section~\ref{sec:notConvex}, we examine the linearization of the homogeneous optimal transport problem the associated perturbation theory.
We demonstrate that there exist cone pairs with multiple isolated \hyperref[eqn:HOT]{H.O.T.} maps, showing that, unlike in the compact optimal transport setting, there is no convex formulation of this problem.

Section~\ref{sec:obstructions} describes the conditions when the geometry of $(\tC, \tC')$ rules out $C^{1,1}$-regularity of the boundary via obstructions to solutions of~\eqref{eqn:HOT}.
We use the Brascamp-Lieb inequality to establish a splitting theorem.
We also introduce a symplectic structure on the space of cone pairs $(\tC, \tC')$ decorated by \textit{sections}. 
This symplectic structure has a moment map for the natural $\sln$ action, and solutions to~\eqref{eqn:HOT} lie in the zero locus of the moment map.
From this, we deduce obstructions to the existence of solutions to~\eqref{eqn:HOT}.

Finally, in Section~\ref{sec:ModuliMoment}, we elaborate on the moment map and symplectic structure to illustrate how one could construct an $\sln$-moduli space of pairs of convex cones admitting \hyperref[eqn:HOT]{H.O.T.} maps and we conjecture that the $C^{1,1}$-boundary regularity of optimal transport can equivalently be described by the structure of this moduli space.
We define a stability framework which, under very general analyticity assumptions, we conjecture to provide a complete affirmative answer to Question~\ref{ques: toC11orNot}.  For completely general convex cones, the $\sln$ orbit of a pair of cones can fail to be locally closed in the Hausdorff topology, suggesting that in full generality, the determination of $C^{1,1}$-boundary regularity may not be local.

\medskip
\noindent\textbf{Acknowledgments:} B.F.~recognizes support from a Simons Dissertation Fellowship, a MathWorks Fellowship, and the Citadel GQS PhD Fellowship.
T.C.C.~is supported in part by NSERC Discovery grant RGPIN-2024-03853.  
F.T.~is supported in part by NSERC Discovery grant RGPIN-2025-06760. 

\medskip
\noindent \textbf{AI disclosure}: ChatGPT 5.6 was used for literature review, drafting Tikz code for the figures per authors' specifications, and as a proof reader, which did not impact the mathematical content.

\section{Preliminaries}\label{sec:prelim}
\subsection{Notation}
We first recall and standardize some notions from convex analysis:

\begin{itemize}[leftmargin=*]
    \item A convex cone $\tC \subset \bR^n$ is a convex subset invariant under positive rescalings about its vertex. 
    For our purposes, every cone has its vertex at the origin and is open.

    \item For $x \in \ol{\Omega}$, we define the tangent cone at $x$ as $\mr{Tan}_x \Omega :=\{tX : X \in \mr{int}(\Omega - x), t > 0\}$.
    For a map $T : \Omega \to \Omega'$, the tangent cone pair at $(x, T(x))$ is given by the pair of cones $(\mr{Tan}_x \Omega, \mr{Tan}_{T(x)}\Omega')$.
    \item We say $\tC$ is pointed if the origin is the only point about which it is scale-invariant; equivalently $\overline{\tC} \cap \overline{(-\tC)} = \{0\}$.
    \item The Legendre dual of a convex function $u$ on $K$ is given by $u^*(y) = \sup_{x \in K}\la x, y \rg - u(x)$.
    \item For $K$ a convex body, the convex support function is defined by $\phi_K(Y) :=\sup_{X \in K}\la X, Y \rg$. We have
    \[
    \phi_{K}^*(X)=\mathbf{1}_K(X):= \begin{cases} 0 & \text{ if } X \in K, \\
    +\infty & \text{ else.}
    \end{cases}
    \]
   
    We denote $\chi_K := e^{-\mathbf{1}_K}$ the standard indicator function.
    \item The polar dual of a convex body $K$ is defined by $K^\circ := \{Y : \phi_K(Y) < 1\}$.
    \item For a convex body $K$, let $\rho_K(X) := \inf \{\lambda :  X \in \lambda K\} = \phi_{K^\circ}(X)$ be its radial support function.
    \item The John ellipsoid of a convex set $K$ is the unique maximal-volume inscribed ellipsoid with center at $c \in K$ satisfying $c + E \subset K \subset c + nE$.
    \item The dual of a cone $\tC$ is given by $\tC^\vee := \mr{relint}\{Y : \la X, Y \rg \geq 0 \text{ for all } X \in \tC\}$.
    The duals of non-pointed cones collapse in dimension which corresponds to $\tC$ splitting some lines. 

    \item We will call points $Y \in \tC'\setminus \tC^\vee$ \textbf{obtuse} and points $Y \in \tC' \cap \tC^\vee$ \textbf{acute}.
    A pair is acute (resp.~obtuse) if $\tC' \subset \tC^\vee$ (resp.~$\tC^\vee \subset \tC'$), and they are strongly acute (resp.~strongly obtuse) if $\ol{\tC'}\setminus\{0\} \subset \tC^\vee$ (resp.~$\ol{\tC^\vee}\setminus\{0\} \subset \tC')$.
    
    \item For $M \in \sln$, we define the action on $V \times V^*$ by $M \cdot (X,Y) := (MX, M^{-T}Y)$, which preserves the duality pairing. 

    \item The automorphism group of the pair $(\tC, \tC')$ is defined by $\mr{Aut}(\tC, \tC') := \{g \in \sln :(g\cdot \tC, g^{-T}\cdot \tC') = (\tC, \tC')\}$.

\end{itemize}

Throughout this paper we will consider pairs of convex cones $(\tC, \tC')$ satisfying the following condition:

\begin{defn}\label{def:posalign}
    The pair of cones $(\tC,\tC') \subset V \times V^*$ is \textbf{positively aligned} if $\tC \cap (\tC')^\vee \neq \emptyset$ and $\tC^\vee \cap \tC' \neq \emptyset$. 
\end{defn}

We denote points in $V, V^*$ with capital letters $X$ and $Y$, respectively.  Throughout, we will define coordinates $(X_1,\ldots, X_n)$ and $(Y_1,\ldots, Y_n)$ from dual bases $e_i$ and $e_i^*$, so $\la e_i, e_j^*\rg = \delta_{ij}$.
For positively aligned cones $(\tC,\tC')$ we may further assume that
\[
P := \tC \cap \{X_n = 1\}, \quad Q := \tC^\vee \cap \{Y_n = 1\}, \quad \text{ and } \quad \Sigma := \tC'\cap  \{Y_n = 1\}
\]
are non-empty convex sets. 

Let $X' := (X_1,\ldots, X_{n-1})$ and $Y' := (Y_1,\ldots, Y_{n-1})$. The lowercase variables $x := \frac{X'}{X_n}$ and $y := \frac{Y'}{Y_n}$ are the projectivized coordinates on the respective links. From these coordinates, we have a choice of affine unit balls $\bB :=\{X : \sum X_i^2 < 1\}$ and $\bB^* := \{Y : \sum Y_i^2 < 1\}$.

\subsection{Monotonicity, OT-separability, and \texorpdfstring{$C^{1,1}$}{C1,1}-regularity}\label{sec:OTSepRound}

We first outline our general strategy to answer Question~\ref{ques: toC11orNot}. Recall that $T = \nabla u$ is an optimal transport map from $\Omega \to \Omega'$ satisfies $u \in C^{1,1-\ve}(\ol{\Omega})$ from~\cite{TristanFreid}.
The pointwise $C^{1,1}$-boundary regularity of $u$ and $u^*$ at a point $(x_0,\nabla u(x_0)) \in \partial \Omega \times \del\Omega'$ is characterized by
\begin{equation}\label{eqn:C11atBdy}
    C^{-1}|x-x_0|^2 \leq u(x) - u(x_0) - \nabla u(x_0)(x - x_0)\leq C|x-x_0|^2
\end{equation}
for some $C$.
For simplicity, assume that $\Omega$ and $\Omega'$ locally agree with their tangent cones $\mr{Tan}_{x_0}(\Omega)$ and $\mr{Tan}_{\nabla u(x_0)}(\Omega')$, or are locally $C^{1,\alpha}$-epigraphs over them.
We may assume that $x_0$ is at the origin and $\nabla u(0) = 0$, so the $C^{1,1}$-regularity is simply $|\nabla u(x) | \leq C|x|$. 
Recall that the \textbf{centered sections} of $u$ at $0$, of height $h$ is
\[
S^c_h(u,0) := \{x : \bar{u}(x) - \bar{u}(0) -p_h\cdot x \leq h\} 
\]
where $\bar{u}$ is the minimal convex extension of $u$ and $p_h$ is chosen so that $0$ is the barycenter of $S^c_h(u,0)$.  The sections of $u$ at $0$ are \textbf{round} if
\[
B_{C^{-1}\sqrt{h}}(0) \subset S^c_h(u,0) \subset B_{C \sqrt{h}}(0).
\]
A direct consequence of this is
\[
C^{-1}|x|^2 \leq u(x) - u(0) - \nabla u(0)\cdot x \leq C|x|^2
\]
meaning that if the sections are round then $u$ is pointwise $C^{1,1}$ at the origin; cf.~\cite{TristanFreid}. 

The monotonicity formula for optimal transport shows that we may extract a blow-up limit $\varphi$ together with a pair of convex cones $(\mathtt{C}_{\infty},\mathtt{C}_{\infty}')$~\cite[Theorem 4.1]{TristanFreid}.  
Precisely, $\varphi:\mathtt{C}_{\infty} \rightarrow \mathbb{R}_{>0}$ is a $2$-homogeneous convex function on $\mathtt{C}_{\infty}$ such that
\begin{equation}\label{eq: introHOT} 
\det D^2\varphi(X)=1 \quad \text{ on  } \mathtt{C}_{\infty} \qquad \text{and} \qquad \nabla \varphi(\mathtt{C}_{\infty}) = \mathtt{C}_{\infty}'.
\end{equation}
The key point is that the failure of roundness yields a sequence of matrices $M_j \in \sln$ with $\|M_j\|\rightarrow \infty$ such that $(M_j\cdot \mathtt{C}, M_j^{-T}\cdot \mathtt{C}') \leadsto (\mathtt{C}_{\infty},\mathtt{C}_{\infty}')$.  We introduce the following definition

\begin{defn}
Let $(\tC,\tC')$ be a pair of positively aligned cones.
An $\sln$-degeneration $(\tC,\tC')\leadsto(\tC_{\infty},\tC_{\infty}')$ is given by:
\begin{itemize}
    \item[$(i)$] A sequence of matrices $M_t \in \sln$ with $\|M_t\| \to \infty$ as $t\rightarrow \infty$, and
    \item[$(ii)$] A pair of positively aligned cones $(\tC_\infty, \tC'_\infty)$ with non-empty interiors, such that
    \[
    (M_t \tC, M_t^{-T}\tC') \rightarrow (\tC_\infty, \tC'_\infty)
    \]
    in the Hausdorff sense on compact sets.
\end{itemize}
We say that the degeneration is non-trivial if $(\tC_\infty, \tC'_\infty)$ is not $\sln$ equivalent to $(\tC, \tC')$.
\end{defn}

For general convex cones it can happen that ${\rm Aut}(\tC,\tC')=\{\mr{Id}\}$ and yet there is a degeneration $(\tC,\tC')\leadsto (\tC,\tC')$.
In order to avoid this situation we make the following definition
\begin{defn}\label{defn:LC}
    We say a pair $(\tC, \tC')$ satisfies condition $(\mr{LC})$ if its $\sln$-orbit is locally closed, meaning  $\ol{\sln \cdot (\tC, \tC')} \setminus \sln \cdot (\tC, \tC')$ is closed in the Hausdorff topology.
\end{defn}

Condition (LC) is very general; see Section~\ref{sec: posAlginLC} for more discussion.  If $(\tC,\tC')$ satisfy condition (LC) and $\mr{Aut}(\tC, \tC')$ is compact, the limiting pairs $(\mathtt{C}_{\infty},\mathtt{C}_{\infty}')$ is not $\sln$ equivalent to $(\mathtt{C},\mathtt{C}')$.
If $\mr{Aut}(\tC, \tC')$ is non-compact then it is possible that the failure of $C^{1,1}$ may be realized by a diverging sequence $M_j \in \mr{Aut}(\tC, \tC')$, although we conjecture this cannot occur (see Section~\ref{sec:ModuliMoment}).
Thus, we are led to the following:
\begin{itemize}
    \item[(1)]  The {\em failure} of $C^{1,1}$-regularity is related to the existence of a non-trivial $\sln$-degeneration of $(\mathtt{C},\mathtt{C}')$ to a pair  of convex cones $(\mathtt{C}_{\infty},\mathtt{C}_{\infty}')$ admitting a homogeneous optimal transport map.
    \item[(2)] Conversely, if $u$ is $C^{1,1}$-regular at $0$, then we expect the asymptotics $u = \varphi(X) + o(|X|^2)$ where $\varphi$ is a $2$-homogeneous convex function solving the optimal transport problem~\eqref{eq: introHOT} for cones $(\mathtt{C}_{\infty},\mathtt{C}_{\infty}')=(\mathtt{C},\mathtt{C}')$. 
\end{itemize}

We can formalize these results as follows:

\begin{prop}\label{prop:OTsep->round}
    Let $\nabla u:\Omega \rightarrow \Omega'$ be an optimal transport map between bounded convex sets.
    Let $(\Omega,\Omega')$ be locally $C^{1,\alpha}$-epigraphs over $(\mathtt{C}, \tC')$ which is the tangent cone pair $(x, \nabla u(x)) \in \p\Omega \times \p \Omega'$. 
    If $(\tC, \tC')$ satisfies $(\textup{LC})$ and OT-separability, and $\mr{Aut}(\tC, \tC')$ is compact, then the sections of $u$ are round and $u$ is pointwise $C^{1,1}$ at $(x,\nabla u(x))$.
\end{prop}
\begin{proof}
    We argue by contradiction; suppose that roundness at $(x, \nabla u(x)) = (0,0)$ fails. 
    From~\cite[Theorem 4.1]{TristanFreid}, normalizing the sections $S^h_c(u,0)$ as $h \to 0$ produces a degeneration $(\tC,\tC')\leadsto (\tC_{\infty},\tC_{\infty}')$  with the limit  pair $(\tC_{\infty},\tC_{\infty}')$ admitting a~\hyperref[eqn:HOT]{H.O.T.} map. By (LC), the limiting pair $(\tC_\infty, \tC'_\infty)$ is not $\sln$ equivalent to $(\tC,\tC')$ violating OT-separability.
\end{proof}

The following result is, in some sense, a converse to the preceding result.

\begin{prop}\label{prop:NotHotNotC11}
   Let $\nabla u:\Omega \rightarrow \Omega'$ be an optimal transport map between bounded convex sets.
   Let $(\mathtt{C}, \tC')$ be the tangent cone of $(\Omega,\Omega')$ and $(x, \nabla u(x)) \in \p\Omega \times \p \Omega'$. If $(\tC, \tC')$ does not admit a solution to~\eqref{eqn:HOT}, then $u$ is not $C^{1,1}$ at $x$.
\end{prop}
\begin{proof}
    If $u$ is $C^{1,1}$ at $x \in \p \Omega$, then the sections of $u$ are round and so by~\cite[Theorem 4.1]{TristanFreid}, there exists a \hyperref[eqn:HOT]{H.O.T.} map on $(\tC, \tC')$ which is an $\sln$-degeneration of $(\tC, \tC')$, a contradiction.
\end{proof}

\subsection{Positive alignment and locally closed orbits}\label{sec: posAlginLC}
We consider the pair of open convex cones $(\tC, \tC')$ and show that positive alignment is equivalent to the existence of a homogeneous convex function whose gradient maps $\tC$ to $\tC'$.

\begin{lem}\label{lem:posAlign}
    The following are equivalent:
    \begin{enumerate}
        \item[$(i)$] The pair $(\tC,\tC')$ is positively aligned.
         \item[$(ii)$] Both $-\ol{\tC^\vee}\cap \ol{\tC'}$ and $\ol{\tC} \cap -\ol{(\tC')^\vee}$ consist of only the origin.
        \item[$(iii)$] There are positive, $2$-homogeneous, convex functions $\phi : \tC \to \bR_+$ and $\phi' : \tC' \to \bR_+$ such that $\p \varphi({\tC}) = \ol{\tC'}\setminus\{0\}$ and $\p \phi' ({\tC'}) = \ol{\tC}\setminus\{0\}$.
    \end{enumerate}
\end{lem}
\begin{proof}
    We first show that $(\tC')^{\vee} \cap \tC \neq \emptyset$ if and only if $ -(\tC)^{\vee} \cap \tC' = \emptyset$ which shows $(i) \Leftrightarrow (ii)$. 
    For the first direction, assume for the sake of contradiction that there is some non-zero vector $Y\in\tC'\cap -\tC^\vee$.  
    Then $\langle X,Y \rangle <0$ for all $X \in \tC$. 
    On the other hand, by assumption, there is some $X\in \tC$ such that $\langle X, Y \rangle >0$, a contradiction.
    For the other direction, suppose to the contrary that $(\tC)^{\vee} \cap \tC'= \emptyset$.
    Since $\tC^{\vee},\tC'$ are convex sets, there exists some $X$ such that
    \[
    \tC' \subset \{ Y:  \langle X, Y \rangle<0\} \qquad \text{and}\qquad  \tC^{\vee} \subset \{ Y: \la  X, Y \rangle>0\},
    \]
    This yields $X \in \tC$ and $-X \in (\tC')^{\vee}$, a contradiction.

    We now show $(iii) \Rightarrow (i)$.
    Suppose there is a non-negative homogeneous convex function $\phi(X)$ on $\tC$ such that $\p \phi(\tC)=\tC'$.  
    Suppose to the contrary that $\tC\cap -(\tC')^{\vee} \neq \emptyset$.  
    Then, there is some non-zero vector $X \in \tC$ such that $\langle X, Y \rangle <0$ for all $Y\in \tC'$.  
    On the other hand, since $\phi$ is homogeneous and convex, we have $\langle  X,\nabla \phi(X) \rangle \propto \phi(X) >0$ since $\nabla \phi$ maps $\tC$ to $\tC'$.
    Therefore, it follows that there is some $Y\in \tC'$ for which $\langle X, Y \rangle >0$, a contradiction. 
    Applying the same argument to $\phi'$ shows the dual property of positive alignment.

    To show $(i) \Rightarrow (iii)$, we consider the function $\phi(X) := v(X)^2$ where 
    \[
    v(X) := \sup_{Y \in \tC', \la e_n, Y\rg = 1}\langle X, Y \rangle\qquad \implies \qquad \p \phi(X) = 2 v(X) \p v(X) \subset \ol{\tC'} \setminus \{0\}
    \]
    recalling $e_n \in \tC \cap (\tC')^\vee$. 
    Therefore, $\p v(te_n) = \{Y_n = 1\} \cap \tC'$, which establishes the desired subgradient property.
    The same construction, but dualized using $e_n^*$, produces $\phi'$. 

\end{proof}
This lemma shows that positive alignment is a necessary condition to the existence of a homogeneous optimal transport map. 
The coordinates yielding $Q$ and $\Sigma$ as the affine links of $\tC^\vee$ and $\tC'$ are guaranteed by positive alignment. 
In fact, defining the cones $\tC^\vee := \{(ty,t) : y \in Q\}$ and $\tC' := \{ (ty, t) : y \in \Sigma\}$ directly implies positive alignment, provided that $Q \cap \Sigma \neq \emptyset$ and both are bounded.
Any point $(y, 1) \in V^*$ for $y \in Q \cap \Sigma$ shows the first half of positive alignment.
The hyperplane defined by $H := \{(y, 0) \}\subset V^*$ has its orthogonal complement generated by $e_n \in V$, which can be chosen such that $\la e_n, Y\rg > 0$ for all $Y \in \tC^\vee$ and $Y \in \tC'$ by letting $e_n$ be the inward normal direction of $H$.

\begin{example}
    If $(\tC, \tC')$ are half-spaces, $\tC = \{ X: \langle X,\nu \rangle >0\}$ and $\tC' = \{Y: \langle \nu', Y \rangle >0\}$, then they are positively aligned if and only $\langle \nu,\nu'\rangle >0$.

    On the other hand, if $\tC$ is such that $\tC^\vee \subset \p \tC'$, then the cones are not positively aligned. 
    This can occur if $\tC$ splits a line $L$, so $\tC^\vee \subset L^\perp$ which does not meet $\tC'$.
\end{example}

We now comment on condition (LC) introduced in Definition~\ref{defn:LC}. 
Natural classes of cones satisfying (LC) include polyhedral cones, piecewise algebraic cones, semianalytic and subanalytic cones, and, quite generally, cones definable in an $o$-minimal expansion of $\bR$; see~\cite{vandenDriesSurvey,ominimalbook} for definitions and properties of these objects as well as a more general survey on $o$-minimal expansions.

As a concrete example, recall that a semianalytic set is any set locally describable as a finite Boolean combination of sets of the form $\{f=0\}$ or $\{g>0\}$ for $f,g$ real analytic functions. The set of semianalytic cones already includes polyhedral and (semi)algebraic cones.

Even more broadly, any pair of cones definable in an $o$-minimal extension of $\bR$ satisfies (LC), as we show below.
We refer the reader to~\cite{ominimalbook} for background showing the above classes are $o$-minimal definable, although for the purposes of this paper, the following proposition is all that will be required.
\begin{prop}\label{prop: oMinimalLC}
    If $(\tC, \tC')$ are definable in an $o$-minimal expansion of $\bR$, they satisfy $(\textup{LC})$.
\end{prop}
\begin{proof}
    Let $D := \tC \cap \p\bB$ and $D' := \tC' \cap \p\bB^*$ be the restrictions of the cones to the unit spheres. 
    The $\sln$ action on points $(X,Y) \in \p\bB \times \p\bB^*$ is given by $g\cdot (X,Y) := (\frac{gX}{|gX|}, \frac{g^{-T}Y}{|g^{-T}Y|})$.
    The Hausdorff metric is definable on $o$-minimal definable sets, which is the metric we consider on the cones for $D$ and $D'$; cf.~\cite[\S 3.1]{Walsberg}.
    Furthermore,~\cite{OminimalLionSp} shows that Hausdorff limits of $o$-minimal definable sets are also definable.
    Remarkably,~\cite[Theorem 8.3]{Walsberg} shows that the $\sln$ orbit of $(D, D')$ is metrically locally Euclidean, for some large, but finite, dimension, after quotienting out by Hausdorff distance zero differences.
    In particular, the orbit is locally closed.
\end{proof}
Proposition~\ref{prop: oMinimalLC} applies, for example, to subanalytic cones since $\tC \cap \p\bB$ and $\tC' \cap \p\bB^*$ are bounded subanalytic sets, so their cones are globally subanalytic and hence definable in an $o$-minimal extension \cite{vandenDriesSurvey}. 
For our purposes, we only need the locally closedness property of $o$-minimal definable cones above as well as the general fact that any reasonably tame algebraic hypothesis is $o$-minimal definable.
For the remainder of this subsection, we assume all pairs satisfy (LC).

\subsection{Homogeneous variational framework}\label{subsec:variational}

Fix a pair of convex cones $\tC\times \tC' \subset V \times V^*$ equipped with the duality pairing. 
We consider the space of strictly positive, convex, homogeneous functions
\begin{equation}\label{HisforHonvex} 
\cH := \bigl\{\text{convex } \varphi: \tC \to \bR_+, \varphi(\lambda X) = \lambda^2 \varphi(X), 0 < \inf_P \varphi \leq \sup_P \varphi < \infty
\bigr\}.
\end{equation} 
We will assume $\varphi \in \cH$ takes on the value $+\infty$ outside $\tC$. 
We will use $v(X) := \sqrt{\varphi(X)}$ which is convex and $1$-homogeneous.
throughout, we will denote both $\varphi$ and $v$ as elements of $\cH$.
Furthermore, we will use $v(x)$ to denote the restriction of $v$ onto the affine link $\{X_n = 1\}$ with $x := \frac{1}{X_n}(X_1,\ldots, X_{n-1})$ as established above, so the identity $\varphi(X) = (X_n v(x))^2$ always holds by definition.
We may choose a representative of an equivalence class in $\bP(\cH)$ as follows: fix a choice $X_0 \in \tC$ and define $\cH_{X_0} := \{\varphi \in \cH : \varphi(X_0) = 1\}$.
We mainly use the normalization $\varphi \in \cH_{X_0}$ for $X_0 = (0,\ldots, 0,1)$ corresponding to $v(0) = 1$. 

We can now define the functionals
\begin{equation}\label{eqn:functionals}
I(\varphi) := |\{\varphi < 1\}|, \quad J(\varphi) = 2^{-n}|\{\varphi^* < 1\}\cap \tC'|, \quad \text{ and }\quad 
\cM(\varphi) := \log I(\varphi) + \log J(\varphi). 
\end{equation}
Throughout, we will denote the regions 
\[
K := \{\varphi < 1\} \qquad \text{and}\qquad K' := \{\varphi^* < 1/4\} \cap \tC' = K^\circ \cap \tC',
\]
and we define their \textbf{heights} by $\sup_K X_n$ and $\sup_{K^\circ} Y_n$, respectively. 
The constant $\frac{1}{4}$ is chosen so that the polar duality formulation holds. 
These are chosen so that $I(\varphi) = |K|$ and $J(\varphi) = |K'|$. 

\begin{remark}
The functionals $I,J$ are equivalent, up to a dimensional constant, to the weighted volume measures $\int_{\tC}e^{-\varphi}dX$ and $\int_{\tC'} e^{-\varphi^*}dY$. Indeed, using 2-homogeneity:
\[
\int_{\tC}e^{-\varphi}\,dX= \int_0^\infty e^{-t}|\{\varphi < t\}|\,dt = |\{\varphi < 1\}|\int_0^\infty e^{-t}t^{\frac{n}{2}}\, dt = \Gamma(\tfrac{n}{2}+1)I(\varphi)
\]
with a similar equation holding for $J$, with a different dimensional constant corresponding to $\{\varphi^* < \frac{1}{4}\} \cap \tC'$. 
\end{remark}

\begin{lem}\label{lem:EL-eqn}
    A $C^2$ strictly convex critical point of $\cM$ solves equation~\eqref{eqn:HOT}, after possibly rescaling.
\end{lem}
\begin{proof}
    It is sufficient to test variations $\psi$ given by $\psi(X) := X_n^2 \psi_0(x)$ for some $\psi_0 \in C^2_c(P)$.
    For $|t|\ll 1$, the function $\varphi_t = \varphi + t\psi$ is convex and $2$-homogeneous since $\varphi$ is strictly convex. Furthermore, since $\varphi = + \infty$ outside $\tC$, we have
    \[
    \frac{d}{dt}\bigg\vert_{t = 0}(\varphi + t\psi)^*(Y) = - \psi(\nabla \varphi^*(Y)),\qquad \text{for a.e. }Y \in \tC'.
    \]
    Therefore, we can compute
    \[
    \delta I_\varphi[\psi] = -\int_{\tC}\psi e^{-\varphi}\,dX \qquad \text{and}\qquad \delta J_\varphi[\psi]=\int_{\tC'}\psi(\nabla \varphi^*( Y)) e^{-\varphi^*}\,dY.
    \]
    Setting $Y = \nabla \varphi(X)$, we see $dY = \det D^2 \varphi(X)\,dX$. Euler's identity, $\la X ,\nabla \varphi(X)\rg= 2\varphi(X)$, implies $\varphi^*(\nabla \varphi(X)) = \varphi(X)$.
    Let 
    \[
    A:= \{X \in \tC : \nabla \varphi(X) \in \tC'\}=\nabla \varphi^* (\tC').
    \]
    We compute $\delta J_{\phi}[\psi] = \int_{A} \psi e^{-\varphi}\det D^2\varphi \,dX$, which means
    \begin{equation}\label{eqn:EL-eqn}
    \delta \cM_\varphi[\psi] = \int_{\tC} \left(\frac{(\det D^2 \varphi)\chi_A}{2^nJ} - \frac{1}{I}\right)\cdot \psi e^{-\varphi}\, dX = 0\qquad \text{ for all } \psi ,
    \end{equation}
    which holds if and only if $ \det D^2\varphi = \frac{2^nJ}{I}$  and $ A = \tC$. By rescaling $\varphi$, we can arrange $I(\varphi)= 2^nJ(\varphi)$ as desired.
\end{proof}

We now provide some estimates on the functionals $I$ and $J$ using the homogeneity to reduce them to integrals on an affine link.
The definitions of $I$ and $J$ in~\eqref{eqn:functionals} agree with the functionals defined in~\cite{TristanBenjyFreid} (in turn inspired by \cite{TristanFreidYau}). The functional $J$ includes a free boundary domain $\Omega := \{v^* + \rho_{\Sigma} < 0\}$. 
\begin{lem}\label{lem:IJonlinks}
    Let $\varphi(X) = (X_n v(x))^2$.
    The functionals $I,J$ are equivalent to the following functions on the link
    \[
    I(\varphi) = \frac{1}{n}\int_P \frac{d x}{v(x)^n}\qquad \text{and}\qquad  J(\varphi) =  \int_{\Omega} (-v^*(y) - \rho_{\Sigma}(y))\, dy
    \]
\end{lem}
\begin{proof}
    Consider any $x \in P$.
    We can rescale by $t=\frac{1}{v(x)}$ to find the point $X = (tx, t)$ on which $\varphi(X) = 1$. Changing variables, the measure $dx$ becomes $v(x)^{n-1}dx$, and therefore 
    \[
    \int_{\{\varphi < 1\}}dX = \int_{x\in P}\left(\int_0^{\frac{1}{v(x)}}\frac{X_n^{n-1}}{v(x)^{n-1}}\,dX_n\right)dx = \frac{1}{n}\int_P \frac{dx}{v(x)^n}.
    \]
Let $u = v^*$, so by definition $u(y) = \sup_{x \in P} \la x, y\rg - v(x)$.  In particular, we have 
    \[
    Y_{n} + u(y) < 0 \qquad \iff \qquad \la x, y\rg + Y_n \leq v(x),
    \]
    and multiplying both sides by $X_n$ shows
    \[
    \la X, Y \rg \leq X_n v(x) \qquad \iff \qquad \la X, Y \rg \leq \sqrt{\varphi}.
    \]
    Define the polar gauge function $v^\circ(y) := \sup_{X \in \tC(P)} \frac{\la X, Y \rg_+}{v(X)}$.
    We compute 
    \begin{equation}\label{eqn:polarVersionOfLegendre}
    \varphi^*(Y) = \sup_{x \in P} \sup_{t >0} \,\bigl(t(\la x,Y'\rg  + Y_n)- t^2v(x)^2\bigr) = \frac{1}{4}v^\circ(Y)^2 
    \end{equation}
    Therefore, we see
    \[
    \varphi^*(Y) < 1 \quad \iff \quad v^\circ(Y) < 2 \quad \iff \quad \la x, Y'\rg + Y_n < 2v(x)\; \text{for all}\;x\in P,
    \]
    meaning that $\frac{Y_n}{2} + \sup_{x \in P}\,(\la x, Y'/2\rg - v(x)) < 0$, showing that 
    \[
    \{\varphi^* < 1\} = \{(Y', Y_n) : Y_n < -2u(Y'/2)\}.
    \]
    Intersecting this set with $\tC'=\tC(\Sigma) := \{(Y', Y_n) : Y_n > \rho_{\Sigma}(Y')\}$, so 
    \[
    J(\varphi) = 2^{-n}\int_{\bR^{n-1}}\left(\int_{\rho_{\Sigma}(Y')}^{-2u(Y'/2)} dY_n\right)\, dY' = 2^{-n}\int_{\bR^{n-1}}(-2u(Y'/2) - \rho_{\Sigma}(Y'))_+\, dY'.
    \]
    Setting $Y' = 2y$ and using the homogeneity of $\rho_{\Sigma}$, we see that $J(\varphi) = \int_{\bR^{n-1}}(-u(y) - \rho_{\Sigma}(y))_+\, dy,$ proving the claim.
\end{proof}
A useful corollary of equation~\eqref{eqn:polarVersionOfLegendre} shows that $v(X) = \phi_{K^\circ}(X)$. 
We can therefore characterize when two functions $\varphi : \tC \to \bR_+$ and $\varphi' : \tC' \to \bR_+$ are Legendre dual based on their sublevel sets.
\begin{cor}\label{cor:duality=legendre}
    The $2$-homogeneous, convex functions $\varphi: \tC \to \bR_+$ and $\varphi' : \tC' \to \bR_+$ are Legendre dual if and only if
    \[
    \{\varphi' < 1/4\} = \{\varphi < 1\}^\circ \cap \tC' \qquad \text{and}\qquad \{\varphi < 1\} = \{\varphi' < 1/4\}^\circ \cap \tC.
    \]
\end{cor}
\begin{proof}
Using $\{\varphi^*<1/4\}=K^\circ$, the restriction to $\tC'$ shows $K'=K^\circ\cap\tC'$.
This identity applied to $\varphi'$ provides the reverse containment $K=(K')^\circ\cap\tC$. 
Conversely, these identities completely determine $\varphi'$ and $\varphi$ as determined by their sublevel sets and homogeneity.
\end{proof}

We note that the height of $K^\circ$ is given by $v(0)$ in our conventions.
The compact link representations of $I$ and $J$ from the previous two lemmas show we can equivalently compute the first variations as 
\begin{equation}\label{eqn:IJvariations}
    \delta I_v[\psi] = -\int_P \frac{\psi}{v(x)^{n+1}}\,dx\qquad \text{and}\qquad \delta J_v[\psi] = \int_{\{u + \rho_{\Sigma} < 0\}} \psi(\nabla u(y))\,dy.
\end{equation}

\begin{defn}\label{def:saturated}
We say that $\varphi\in \cH$ is {\bf saturated} if
\[
\varphi= (\phi_{K'})^2 \quad \text{ where }\quad K' := \{\varphi^* < 1/4\} \cap \tC'.
\]
We denote the class of saturated functions by $\ul{\cH}\subset \cH$.
\end{defn}
The geometric characterization of the class of saturated functions is that $\ul{\cH}$ is the class of 2-homogeneous convex functions for which $\nabla \phi(\tC) \subset \tC'$, a property clear from the definition of the convex support function. 
We record the useful geometric reframing of saturation.
\begin{cor}\label{lem:cor:geomSaturationDescription}
    Let $\varphi \in \cH$, and denote
    \[
    K:= \{ \varphi <1\} \quad \text{and} \quad  K'= \{\varphi^* <1/4\} \cap \tC'.
    \]
    Then $\varphi$ is saturated in the sense of Definition~\ref{def:saturated} if and only if $K^{\circ} = K'-\tC^{\vee}$.
\end{cor}

The following result shows that, given $\varphi \in \cH$, we can construct a canonically associated saturated function $\ul{\varphi} \in \ul{\cH}$ with the property that $\cM(\varphi) \leq \cM(\underline{\varphi})$ with equality if and only if $\varphi \in \underline{\cH}$.  In particular, we conclude that critical points of $\cM$ are always saturated.
\begin{lem}\label{lem:boundaryConditionAchievement}
    For any $\varphi \in \cH$, there exists a canonical saturated $\ul{\varphi} \in \ul{\cH}$ satisfying $I(\varphi) \leq I(\ul{\varphi})$ and $J(\varphi) = J(\ul{\varphi})$.
    Furthermore, any unsaturated function has a strictly $\cM$-increasing deformation, so it cannot be critical.
\end{lem}
\begin{proof}
    Define $\ul{\varphi} :=(\varphi^* + \mathbf{1}_{\tC'})^*+ \mathbf{1}_{\tC}$, which will satisfy $\ul{\varphi} \leq \varphi$, meaning $K \subset \ul{K}$, so $I(\varphi) \leq I(\ul{\varphi})$.
    However their Legendre transforms agree on $\tC'$ so this maintains $J(\ul{\varphi}) = J(\varphi)$, yielding the desired boundary condition.
    Furthermore, $K \subset \ul{K}$ implies that $v \geq \ul{v}$.

    Geometrically, this is alternatively realized by taking $\ul{K}^\circ := K' - \tC^\vee$ whose polar dual $\ul{K}$ defines the sublevel set of $1$ for $\ul{\varphi}$. 
 
    The saturation procedure is strictly energy increasing, which we quantify by 
\begin{equation}\label{eqn:saturationDefect}
    \delta \cM_v[\ul{v} - v] = \frac{\int_P (v - \ul{v})v^{-(n+1)}\,dx} {I(v)} \geq 0,
\end{equation}
which vanishes if and only if $v = \ul{v}$.
\end{proof}
Motivated by this result, we introduce the following definition.

\begin{defn}\label{defn:saturation}
Let $\varphi \in \cH$.  We define the {\bf saturation of $\varphi$} to be 
\[
\ul{\varphi} :=(\varphi^* + \mathbf{1}_{\tC'})^*+ \mathbf{1}_{\tC}\qquad \iff\qquad \ul{K}^\circ := (K^\circ \cap \tC') - \tC^\vee.
\]
\end{defn}

By Lemma~\ref{lem:boundaryConditionAchievement} we can always reduce to the case of saturated functions by replacing any $\varphi \in \cH$ with its saturation $\ul{\varphi} \in \ul{\cH}$, in the sense of Definition~\ref{defn:saturation}.

\subsection{Convex criticality and the Kantorovich formulation}
We identify the function space $\bP(\cH)$ with positive convex functions on $\ol{P}$, normalized to $v(0) = 1$.
This space is not a Banach manifold structure because convexity is not an open condition.
From a variational standpoint, this is problematic; the calculation in Lemma~\ref{lem:EL-eqn} cannot hold verbatim.
First, since the regularity of the critical point is not known a priori, and secondly since the space of variation fields is restricted.

These technical issues are circumvented by using the Kantorovich formulation~\cite{Kantorovich} of optimal transportation on the link $P$ and the work of Gangbo-McCann~\cite{Gangbo-McCann}.
In this section we describe how we can use a generalization of the min-max theory developed by Palais~\cite{Palais}, in combination with the work of Gangbo-McCann~\cite{Gangbo-McCann}, to show that critical points of the $\cM$-functional nevertheless solve the optimal transport problem. 

We first define a metric on $\bP(\cH)$ normalized to $v(0) = 1$ to remove scaling. 
For any two $v, w \in \cH$, define
\begin{equation}\label{eqn:distanceOnC}
d(v,w) := \| \log (v/w) \|_{L^\infty(P)}.
\end{equation}
By definition $d(v,w) \leq r$ if and only if $e^{-r} w \leq v \leq e^r w$ which translates to the uniform equivalence of the sets $K^\circ_v$ and $K^\circ_w$ with the same constants $e^{\pm r}$. 
An equivalent metric is given without normalization $v(0) = w(0) = 1$ by defining 
\[
d_{\bP}([v],[w]) := \inf_{r > 0} d(rv, w),
\]
which satisfies $d_{\bP}(v,w) \leq d(v,w) \leq 2d_{\bP}(v,w)$, so we may use either up to a uniform constant.
The lower bound follows immediately by taking the infimum.
For the upper bound, let $f := \log (v/w)$ and set $m:= \inf_P f$ and $M:= \sup_P f$. 
Therefore, we have $d_{\bP}([v],[w]) = \|f + \log r\|_\infty$ where the $L^\infty$ value is optimized at $\log r = -\frac{m+M}{2}$, so $d_{\bP}([v],[w]) =\frac{1}{2}(M-m)$.
If $v(0) = w(0) = 1$, then $m \leq 0 \leq M$.
Thus, $d(v,w) \leq \max \{-m,M\}$, which is the claimed upper bound. 

We define associated norms on the first-order variations $\psi$ arising from convex families $v_t = v + t\psi + o(t)$ projectively as
\begin{equation}\label{eqn:|psi|_v}
\|\psi\|_v := \sup_P\frac{|\psi|}{v} \qquad \text{and}\qquad \|\psi\|_{v, \bP} := \inf_{r}\left\|\frac{\psi}{v}-r\right\|_{L^\infty(P)}
\end{equation}
corresponding to $d$ and $d_{\bP}$.
These norms are uniformly equivalent and satisfy ${\lim_{t \to 0}\frac{d(v_t, v)}{t} = \|\psi\|_v}$.
Since $\cM(tv) = \cM(v)$, a variation is trivial projectively when $v \propto \psi$.
Since our domain is convex functions, we may only test curves $v_t \subset \cH$, so not every variation $\psi \in C^0(P)$ is admissible. 
\begin{defn}\label{defn:convexCritical}
    Let $v,w \in \bP(\cH)$. 
    We define the measures 
        \[
        d\mu := \frac{(v(x))^{-(n+1)}\chi_P\,dx}{\int_P v^{-(n+1)}\,dx}\qquad \text{and}\qquad d\nu := \frac{\chi_{\Omega_v}}{|\Omega_v|}\,dy.
        \]
    \begin{itemize}[leftmargin=1cm]
        \item[$(i)$] The \textbf{convex first variation} of $\cM$ at $v$ in the direction $w-v$ is 
    \[
    \delta^+ \cM(v; w-v) := \lim_{t \to 0}\frac{\cM\left((1-t)v + tw)-\cM(v\right)}{t},
    \]
    which, by Lemma~\ref{lem:EL-eqn} is given by
    \[
     \delta^+ \cM(v; w-v) =-\alpha_v\int_P(w(x)-v(x))\, d\mu + \beta_v \int_{\Omega_v}(w(\nabla v^*(y)) - v(\nabla v^*(y)))\,d\nu ,
    \]
    where
    \[
    \alpha_v := \frac{\int_{P}v(x)^{-(n+1)}\,dx}{I(v)}\qquad \text{and}\qquad \beta_v := \frac{|\Omega_v|}{J(v)}.
    \]
    \item[$(ii)$] We say that a function $v$ is \textbf{convex critical} if $\delta^+\cM(v;w-v) \leq 0$ for all $w \in \cH$. 
    \\
    \item[$(iii)$] Convex criticality is measured by the \textbf{convex defect}
    \begin{equation}\label{eqn:convexDefect}
        \mathfrak{a}(v) := \sup_{w \in \cH, w \neq v} \frac{(\delta^+\cM(v; w-v))_+}{\|w-v\|_v}.
    \end{equation}
    Clearly, $\mathfrak{a}(v)=0$ precisely when $\delta^+ \cM(v;w-v) \leq 0$ for all $w$. 
    \end{itemize}
\end{defn}

\begin{lem}\label{lem:(C,d)props}
    We record the following properties of the metric space $(\bP(\cH), d_{\bP})$:
    \begin{enumerate}
        \item[$(i)$]The functionals $\log I$ and $\log J$ are $n$-Lipschitz, so the energy $\cM$ bound is $2n$-Lipschitz: $|\cM(v) - \cM(w)| \leq 2 n d(v,w)$.
        \item[$(ii)$] The function space $(\bP(\cH), d_\bP)$ is complete.
        \item[$(iii)$] Saturation is distance-decreasing, $d(\ul{v},\ul{w}) \leq d(v, w)$.
        \item[$(iv)$] If $v_j \to v$ and $w_j \to w$ are Cauchy sequences, then $\delta^+ \cM(v_j; w_j - v_j) \to \delta^+ \cM(v;w-v)$.
    \end{enumerate}
\end{lem}
\begin{proof}
    The main characterization of the distance is 
    \begin{equation}\label{eqn:distanceGeometry}
        d(v,w) \leq r\qquad \iff \qquad e^{-r} K^\circ_v \subset K^\circ_w \subset e^{r}K^\circ_v \qquad \iff \qquad e^{-r} K_v \subset K_w \subset e^r K_v
    \end{equation}
    from which these properties follow. 
    Since $v = \phi_{K^\circ}$, we observe that $v_j \to v$ in $d$ is equivalent to the uniform convergence of functions $v_j \to v$ on $P$.

    \noindent$(i)$: From the geometric characterization~\eqref{eqn:distanceGeometry}, we have
    \[
    e^{-nr}|K_w| \leq |K_v| \leq e^{nr}|K_w|\qquad \text{and}\qquad e^{-nr}|K'_w| \leq |K'_v| \leq e^{nr}|K'_w|,
    \]
    which correspond to 
    \begin{equation}\label{eqn:LogILogJCont}
    |\log I(v) - \log I(w)| \leq n d(v,w)\qquad \text{and}\qquad |\log J(v) - \log J(w)| \leq n d(v,w),
    \end{equation}
    showing $\log I$ and $\log J$ are $n$-Lipschitz, so $\cM$ as their sum is $2n$-Lipschitz.
    In particular, $I$ and $J$ are continuous.

    \smallskip
    \noindent$(ii)$: Let $v_j$ be a Cauchy sequence, so $f_j := \log v_j$ is a Cauchy sequence in $C^0(P)$.
    Therefore, $v_j \to v$ uniformly in $C^0(P)$, so the limit $v := e^{\lim f_j}$ is convex. 
    The Cauchy property of $f_j$ provides the requisite uniform lower bound on $\inf_P v$. 
    
    \smallskip
    \noindent$(iii)$: Let $r = d(v,w)$, so from~\eqref{eqn:distanceGeometry}, we know that 
    \[
    e^{-r}K'_v \subset K'_w \subset e^r K'_v \qquad \implies \qquad e^{-r}(K'_v - \tC^\vee) \subset K'_w - \tC^\vee \subset e^r (K'_v - \tC^\vee)
    \]
    and 
    \[
    e^{-r}K'_v \subset K'_w \subset e^r K'_v \qquad \implies \qquad e^{-r}\phi_{K'_v} \leq \phi_{K'_w} \leq e^{r}\phi_{k'_v}
    \]
    which show that $d(\ul{v}, \ul{w}) \leq r$.
    
    \smallskip
    \noindent$(iv)$: The Cauchy property implies that both sequences $v_j$ and $w_j$, as well as their difference $\psi_j := w_j - v_j$, converge uniformly to $v$, $w$, and $\psi := w-v$.
    From this, we can deduce that the entire quantity in the convex first variation uniformly converges as well. 
    We have bounds $e^{-\varphi_j} \leq e^{-C|X|^2}$ and $e^{-\varphi^*_j} \leq e^{-C'|Y|^2}$ and homogenizing $|\psi_j| \leq C$ shows that by dominated convergence, we have $\delta^+\cM(v_j; w_j - v_j)$ converges to $\delta^+\cM(v;w-v)$.
\end{proof}

    The next key result uses the Kantorovich formulation of optimal transport, as well as the results of Gangbo-McCann~\cite{Gangbo-McCann} and regularity theory of Caffarelli~\cite{Caffarelli} to upgrade the convex criticality of a function to a genuine~\hyperref[eqn:HOT]{H.O.T.} map.

    \begin{prop}\label{prop:crit->HOT}
        If $v$ is convex critical, then it solves~\eqref{eqn:HOT}.
    \end{prop}
    \begin{proof}
        Consider any $\widehat{v} \in \cH$ and we consider the variation $\widehat{v}-v$ and let $u := v^*$ and $\widehat{u} := \widehat{v}^*$. 
        We consider the family $v_s := (1-s) v + s \widehat{v}$ and $u_s := v_s^*$ which satisfies
        \[
        u_s \leq (1-s) u + s \widehat{u} \qquad \implies \qquad \dot{u}_0 \leq \widehat{u} - u
        \]
        by Legendre duality.
        Furthermore, maximizing over all $y$, we have $\dot{u}_0 = (v - \widehat{v})(\nabla u(y))$ almost everywhere.
        Therefore, we have
        \begin{equation}\label{eqn:udiff}
            u(y) - \widehat{u}(y) \leq (\widehat{v} - v)(\nabla u(y)).
        \end{equation}
        By the convex criticality assumption (see Definition~\ref{defn:convexCritical}), we have that 
        \[
        \delta^+ \cM(v;\widehat{v}-v) := -\alpha_v\int_P (\widehat{v}-v)(x)\,d\mu + \beta_v\int_{\Omega} (\widehat{v}-v)(\nabla u(y))\,d\nu \leq 0.
        \]
        From the families $v_s :=\frac{v \pm s}{1 \pm s}$, keeping the normalization $v_s(0) = 1$, we have $\delta^+\cM_v[\pm 1] = \mp \alpha_v \pm \beta_v \leq 0$, so $\alpha_v = \beta_v$. 
        Therefore, the convex criticality assumption implies
        \[
        \int_{\Omega}(\widehat{v}-v)(\nabla u(y))\, d\nu \leq \int_P (\widehat{v}-v)\, d\mu,
        \]
        and applying~\eqref{eqn:udiff} and Legendre duality $u(x) + v(y) \geq \la x, y \rg$ yields
        \[
        \int_Pv\,d\mu + \int_\Omega u\, d\nu \leq \int_P \widehat{v}\,d\mu + \int_\Omega \widehat{u}\, d\nu.
        \]
        This inequality is the Kantorovich optimal transport inequality, so Gangbo-McCann~\cite{Gangbo-McCann} proves that $v$ is an Alexandrov solution of the Monge-Amp\`ere equation on $P$.
        Finally, the interior regularity theory of Caffarelli~\cite{Caffarelli} shows that $v \in C^{\infty}(P)$.
        Applying Lemma~\ref{lem:EL-eqn} completes the result.
        
    \end{proof}

\section{Acute cones}\label{sec:acute}
We first handle the case of generalized \textbf{acute cones}, meaning that $\tC' \subset \tC^\vee$.
This extends the result~\cite[Theorem~1.1]{TristanBenjyFreid} to include potential boundary agreement $y \in \p Q \cap \p \Sigma$.
The novel ingredient is the OT-separability hypothesis, which prevents degenerations which were previously ruled out by the compactness $\Sigma \Subset Q$ ~\cite{TristanBenjyFreid}.

\begin{lem}
    For an acute pair, there is a uniform energy bound $\sup_{\varphi \in \cH}\cM(\varphi) \leq C < +\infty$.
\end{lem}
\begin{proof}
    We consider $L := \mr{ConvexHull} (K, -K) \subset V$ for $K := \{\varphi < 1\}$.
    Since $L$ is centrally symmetric, it has its Santal\'o point at the origin, so the Blaschke-Santal\'o theorem says that $|L||L^\circ| \leq \omega_n^2$ for $\omega_n = |B_1|$. 
    Acuteness provides $\la X, Y\rg \geq 0$ for all $X \in K$ and every $Y \in \tC'$. 
    Since $0 \in \ol{K}$, for any $Y \in \tC'$, $\phi_K(-Y) = 0$, so $\phi_L(Y) = \max\{\phi_K(Y), \phi_K(-Y)\} = \phi_K(Y)$.
    Therefore, $L^\circ \cap \tC' = K^\circ \cap \tC' =: K'$.
    Since $|L \cap \tC| \geq |K|$, taking the intersection with the cones demonstrates
    \[
    |K||K'| \leq |L||L^\circ| \leq \omega_n^2,
    \]
    which provides the energy upper bound. 
\end{proof}

It remains to show that a sequence $\varphi_j$ such that $\cM(\varphi_j) \to \sup_{\cH} \cM$ converges to a critical point of $\cM$.
The following proves Theorem~\ref{thm: HotExistIntro} for the first alternative in condition $(ii)$.
\begin{prop}\label{prop: Sep+Acute=HOT}
    If $(\tC, \tC')$ is OT-separable and acute, then there exists a solution $\varphi$ of~\eqref{eqn:HOT} which maximizes $\cM$. 
\end{prop}
\begin{proof}
    Consider an energy maximizing sequence $\varphi_j$. 
    The body $L_j := \mr{ConvexHull}\{K_j, -K_j\}$ is symmetric about the origin as is $L_j^\circ$, which by acuteness satisfies $L_j^\circ \cap \tC' = K_j'$. 
    Let $E_j$ be the John ellipsoid of $L_j$, which is centered at the origin. 
    Since $\cM$ is scale-invariant, we may assume each $|E_j|$ has volume 1.
    If $E_j$, up to a subsequence, converges to some non-degenerate ellipsoid $E$, then $\varphi_j \to \varphi \in \cH$ has a convergent sequence. 
    Applying~\cite[Proposition 3.1]{TristanBenjyFreid} shows $\varphi$ solves~\eqref{eqn:HOT}.

    Otherwise, suppose $E_j$ are degenerating.
    We can normalize by $g_j \in \sln$ such that $\widehat{L}_j := g_jL_j$ and $\widehat{L}_j^\circ = g_j^{-T} L_j^\circ$. 
    The rescaled sections $\widehat{K}_j := g_j K_j$ and $\widehat{K}'_j := g_j^{-T} K'_j$ arise on the rescaled cones $(g_j \cdot \tC, g_j^{-T}\cdot \tC')$. 
    Taking the limits $\widehat{K}_\infty := \lim_{j\to \infty}\widehat{K}_j$ and $\widehat{K}'_\infty := \lim_{j \to \infty}\widehat{K}'_j$ are likewise uniformly round, so the limiting cones $(\tC_\infty, \tC'_\infty)$ generated by them are of full measure and realized by the $\sln$-degeneration $(g_j\cdot \tC, g_j^{-T} \tC') \rightsquigarrow (\tC_\infty, \tC'_\infty)$.
    By the volume convergence of the sections, we have $\cM_\infty(\varphi_\infty) = \lim_{j \to \infty}\cM(\varphi_j) = \sup_{\cH}\cM$. 
    It remains to show that $\varphi_\infty$ is an energy maximizer of the limiting pair. 
    Suppose not; then for $\widetilde{\varphi}_\infty$ a function with strictly higher energy, we can consider $\widetilde{L}_\infty := \mr{ConvexHull}\{\widetilde{K}_\infty, -\widetilde{K}_\infty\}$ and define $\widetilde{K}_j := (\widetilde{L}^\circ_\infty \cap \tC_j')^\circ \cap \tC_j$, which satisfies $\widetilde{L}_\infty^\circ \subset \widetilde{K}_j^\circ$ since polarity reverses inclusion.
    Therefore,
    \[
    \sup_{\cH}(IJ) \geq \limsup_{j} |\widetilde{K}_j||\widetilde{K}'_j| \geq |\widetilde{K}_\infty||\widetilde{K}'_\infty| > |K_\infty||K'_\infty|  
    \]
    which violates the assumption that $\varphi_j$ was energy maximizing.
    Once again, applying~\cite[Proposition 3.1]{TristanBenjyFreid} shows $\varphi_\infty$ solves~\eqref{eqn:HOT} on the limit cones $(\tC_\infty, \tC'_\infty)$.
    Since we assumed $(\tC,\tC')$ was OT-separable, this implies $(\tC_\infty, \tC'_\infty)$ is $\sln$-equivalent to $(\tC,\tC')$ and the Proposition follows.
\end{proof}

\begin{remark}
    Proposition~\ref{prop: Sep+Acute=HOT} yields a new proof of \cite[Theorem~1.1]{TristanBenjyFreid}.
    Indeed, as proved in Corollary~\ref{cor:Crossing=OTsep}, strongly acute cones are OT-separable.
\end{remark}

\section{Max-min critical points}\label{sec:maxmin}
In this section, we first describe the main ingredients necessary to find a critical point via max-min techniques: a non-trivial topology over which we can sweepout, the finiteness of a max-min width, and a Palais-Smale type condition.
We then utilize a generalized min-max theory for semicontinuous functionals to find a critical point of $\cM$, which by Proposition~\ref{prop:crit->HOT} produces a \hyperref[eqn:HOT]{H.O.T.} map. 
From now on we assume the pair is not acute, meaning $\tC' \not\subset \tC^\vee$. 
The energy $\cM$ now takes on all real values, requiring the max-min formulation to find critical points.
We will also show that in many cases OT-separability is effectively verifiable.

\subsection{Linking}\label{sec:Linking}
We now discuss the linking topology of the pair $(\tC, \tC')$.
The link $\Sigma$ decomposes where $O := \Sigma \setminus \ol{Q}$ are the obtuse points and $A := Q \cap \Sigma$ are the acute points. 
We define a quantifiable acute region as $A_{\ve} := \{y \in A : d(y, \p Q) > \ve\}$.
Because $O \subset \Sigma \setminus A_\ve$, we have an induced map on the relative homology $\iota_* H_*(\Sigma, O) \to H_*(\Sigma,\Sigma \setminus A_\ve)$ from the inclusion of pairs $\iota: (\Sigma, O) \to (\Sigma, \Sigma \setminus A_\ve)$.
\begin{defn}\label{def:Linked}
    The pair $(\tC, \tC')$ is \textbf{linked} if, for some $\ve > 0$, 
    \[
    \mr{im}\bigl(H_*(\Sigma, O) \xrightarrow{\iota_*} H_*(\Sigma,\Sigma \setminus A_\ve)\bigr) \neq 0.
    \]
\end{defn}
The definition generalizes the second alternative in Theorem~\ref{thm: HotExistIntro} $(ii)$.

Recall $\bB := \{X :\sum X_i^2 < 1\} \subset V$ and similarly for $\bB^* \subset V^*$ are the choice of unit balls induced by our coordinates.
To any function $v = \phi_{K'}$, we associate the data of $E := c + L\bB^*$, the John ellipsoid of $K'$ for $L$ a positive-definite symmetric matrix.
We say that $v$ is in \textbf{vertical John position} if $L = \mr{Id}$.
The \textbf{Steiner point}, which measures the average gradient image, is given by
\begin{equation}\label{eqn:steiner}
S_U(K') := \frac{1}{|U|}\int_U \nabla \phi_{K'}(X)\,dX,\qquad  \mathfrak{q} :=\frac{
S_U(K')}{\la e_n, (
S_U(K'))\rg} \in \Sigma,
\end{equation}
where $U = e_n + r\bB$ is some fixed open set in $\tC$ with $r$ sufficiently small. 

\begin{remark}
    The Steiner point $\mathfrak{q}$ depends on the choice of $U$, which we fix once and for all based on Lemma~\ref{lem: SteinerPointLowBound} and Lemma~\ref{lem:affineSection}. 
\end{remark}
  Since $\mathfrak{q}(v) = \mathfrak{q}(\lambda v)$ for all $\lambda > 0$, we can work in the full space $\cH$ in place of $\bP(\cH)$. 
The map $\mathfrak{q}$ defines a map $\mathfrak{q} : \bP(\cH) \to \Sigma$ which allows us to trace quantitatively where a definite piece of the mass of $K'$ lies when projected to $\{Y_n = 1\}$. 

\begin{lem}\label{lem: SteinerPointLowBound}
    For any $0 < c_\Sigma < 1$, if $r$ is sufficiently small, we have $c_{\Sigma}\leq \la e_n, S_U\rg \leq 1$. 
\end{lem}
\begin{proof}
    Let $M := \sup_{\tC' \cap \{Y_n < 1\}}|Y|$, which is finite since $\tC'$ is a pointed cone. Consider any $Y \in \tC' \cap \{Y_n < 1\}$. 
    For any $X \in U$, let $Y \in K'$ satisfy $\la X, Y\rg = \phi_{K'}(X)$. 
    For any $Z \in K'$ with $Z_n = 1$, we know that $\la X, Y\rg \geq \la X, Z\rg$. 
    We can bound
    \begin{align*}
    Y_n = \la X, Y\rg - \la X - e_n, Y\rg &\geq \la X, Z\rg -|X-e_n||Y|  \\
    &\geq \la e_n ,Z\rg + \la X - e_n, z\rg - rM \geq 1 - 2rM
    \end{align*}
    and choosing $r \leq \frac{1-c_\Sigma}{2M}$ completes the proof.
\end{proof}

The following lemma establishes a uniform energy bound based on the position $\mathfrak{q}$.

\begin{lem}\label{lem:EnergyInfBound}
    For any $v$, if $\mathfrak{q} \in Q$, then $\cM(v) \leq C + n\log \frac{1}{\mr{dist}(\mathfrak{q}, \partial Q)} =: E_{\mr{dist}(\mathfrak{q}, \partial Q)}$.
\end{lem}
\begin{proof}
    For any $q \in Q$, let $\alpha_q := \inf_{x\in P} (1 + \la x, q\rg) > 0$.  We have that $\alpha_q \to 0$ as $q \to \p Q$ and $\alpha_q \sim \mr{dist}(q, \p Q)$ with uniform constants depending on $P$. 
    Since $v(X) = \sup_{Y \in K^\circ} \la X, Y\rg$ and the Steiner point $S_U \in K'$, we have for all $x$
    \[
    v(x) \geq \la (x, 1), S_U\rg \geq \la e_n, S_U\rg(1 + \la x, \mathfrak{q}\rg) \geq c_\Sigma (1 + \la x, \mathfrak{q}\rg) \geq c_\Sigma{\alpha_q},
    \]
    which implies that $I(v) \leq \frac{|P|}{n} c_\Sigma^{-n}\alpha_q^{-n}$.
    Under the height $1$ normalization, we have\\ \mbox{$J \leq |\tC' \cap \{ y_n < 1\} | \leq C$}.
    Combining these estimates yields $\cM(v) \leq C + n\log \frac{1}{\mr{dist}(q,\p Q)}$ as desired.

\end{proof}

We can now use the fibers of the center map $\mathfrak{q} : \cH \to \Sigma$ to decompose the function space.
We define
\[
\cO := \mathfrak{q}^{-1}(O) \qquad \text{and}\qquad \cA_\ve := \mathfrak{q}^{-1}(A_\ve)
\]
and, as above, since $\cO \subset \cH \setminus \cA_{\ve}$, we have a map $\iota_* :H_*(\cH, \cO) \to H_*(\cH, \cH\setminus \cA_\ve)$.
We will show that the linking condition~\ref{def:Linked} shows that the image of this map is non-trivial, for $\ve$ sufficiently small.

We now show that $\mathfrak{q}$ provides a homotopy equivalence between any $U \subset \Sigma$ and $\mathfrak{q}^{-1}(U)$. 
\begin{lem}\label{lem:affineSection}
    The fibers of $\mathfrak{q}$ are non-empty and contractible.
    Furthermore, the piecewise-affine function $\gamma(y) := \max\{\ve, 1 + \la \cdot, y \rg\} \in \ul{\cH}_{e_n}$ for $y \in \Sigma$ is a right inverse to $\mathfrak{q}$, meaning $\mathfrak{q} \circ \gamma = \mr{Id}$.
\end{lem}
\begin{proof}
    Let $K'_0$ and $K'_1$ arising from $v_0, v_1 \in \cH$ both have the same rescaled center $\mathfrak{q}$.
    Consider $v_t := (1-t)v_0 + t v_1$ which produces a convex interpolation of admissible $K'_t$ between any $v_0$ and $v_1$ with the same Steiner point $\mathfrak{q}$.
    Indeed, $K'_t = (1-t)K'_0 + tK'_1$ has the same rescaled center $\mathfrak{q}$ for all $t$. 
    By a general property of Steiner points, clear from ~\eqref{eqn:steiner}, we have $S_U((1-t)K'_0 + tK'_1) = (1-t)S_U(K'_0) + tS_U(K'_1)$.
    Therefore, the saturated fibers of $\mathfrak{q}$ are convex.
    Saturation preserves $\mathfrak{q}$ and is a retract, and thus contractible, exhibiting the homotopy equivalence.

    We note that the gradient image of $\varphi$ is given by $\{(ty, t) : t > 0\}$. 
    Since $y \in \Sigma$, the function $\gamma(y)$ is automatically saturated. 
    From the definition of the Steiner point~\eqref{eqn:steiner}, $S_U(K')$ for the function $\varphi_Y(X) = (\la X, Y\rg)^2$, or $v_y(x) = 1 + \la x, y \rg$ is $y$.
    We note that for $y \in \Sigma \setminus Q$, the function $1 + \la \cdot, y\rg$ is not strictly positive on $P$.
    The homogeneous extension is given by 
    \[
    v_{\ve, y}(X) := \max \{\ve X_n, X_n + \la X', y\rg\},
    \]
    so for $R_\Sigma := \max_{y \in \Sigma} |y|$ and $U := e_n + r\bB$, we may choose $\ve$ such that $R_\Sigma \frac{r}{1-r} \leq 1 - \ve$.
    We may therefore fix $0 < \ve \ll 1$ sufficiently so that $\gamma_{y}(X) = v_{\ve, y}(X)$ which has its gradient image from $U$ is contained in $\{(ty, t), t > 0\}$, so $\mathfrak{q}(\gamma(y)) = y$ as claimed.
\end{proof}
The triviality of the fibers and the section providing a homotopy inverse imply the following homology equivalence.
\begin{cor}\label{cor:homology=}
For any $D \subset \Sigma$, the map $\mathfrak{q}_*$ induces an isomorphism of homology groups $\mathfrak{q}_{*}:H_*(\cH, \mathfrak{q}^{-1}(D)) \to H_*({\Sigma}, D)$.
\end{cor}
\begin{proof}
    We consider the short exact sequence of chains given by the pairs $(\Sigma, D)$ and $(\cH, \mathfrak{q}^{-1}(D))$. 
    From Lemma~\ref{lem:affineSection}, we know that the induced maps $\mathfrak{q}_* : H_*(\cH) \to H_*(\Sigma)$ and $\mathfrak{q}_* : H_*(q^{-1}(D)) \to H_*(D)$ are isomorphisms.
    Therefore, we can apply the five-lemma to the long exact sequences in homology given by the pairs using the map $\mathfrak{q}_*$ which fills in the final vertical isomorphism $\mathfrak{q}_* : H_*(\mathfrak{q}^{-1}(D)) \to H_*(D)$.
\end{proof}

\begin{lem}\label{lem:tilt=king}
    Let $v \in \cH$ and consider the path $v_t(x) := v(x) + t \cdot\max\{0, 1 + \la x, \mathfrak{q}(v)\rg\}$.
    This path satisfies:
    \begin{enumerate}
        \item[$(i)$] $v_t \in \cH$ is continuous,
        \item[$(ii)$] $\mathfrak{q}(v_t) = \mathfrak{q}(v)$ for all $t$, and 
        \item[$(iii)$] if $\mathfrak{q}(v) \in O$, then $\cM(v_t) \to \infty$ as $t \to \infty$.
    \end{enumerate}
\end{lem}
\begin{proof}
    The first two properties are immediate based on the definitions.

    For $(iii)$, the normalized functions $\frac{v_t}{1 + t}$ converge uniformly on $P$ to $\phi_H,$ where $H := \mr{ConvexHull}\{0, \mathfrak{q}(v)\}-\tC^\vee$.
    Since $\mathfrak{q} \in \tC'$ is an interior point, we have that $|H| > 0$, so $J(\frac{v_t}{1+t}) \to |H| > 0$. 
    By strict obtuseness, we know that $\phi_H$ is 0 on an open set of $P$, so Lemma~\ref{lem:IJonlinks} shows that $I(\frac{v_t}{1+t}) \geq C t^{n}$, noting that $C$ depends on $\inf_P v$ and $\mathfrak{q}$.
\end{proof}

For any chain $\Gamma = \sum a_i \sigma_i$ for $\sigma_i : \Delta^k \to \cH$, we define its support as $|\Gamma| := \bigcup \sigma_i(\Delta^k)$, the image of all the functions in this chain.
From Corollary~\ref{cor:homology=} and the linking hypothesis, we can take some $\alpha \in H_*(\cH, \cO)$ whose image in $H_*(\cH, \cH\setminus \cA_\ve)$ is non-trivial, again from the five-lemma applied to the long exact sequence of a triple. 
Recall $E_\ve$ which is the constant from Lemma~\ref{lem:EnergyInfBound} arising from $A_\ve$. 
We define the space of \textbf{admissible families} as
\begin{equation}
    \mathscr{S}_k := \{\Gamma \in C_k(\cH) : \p \Gamma \in C_{k-1}(\cO), \inf_{v \in |\p \Gamma|} \cM(v) \geq E_\ve + 10\}.
\end{equation}
From this, we may define the \textbf{max-min width} of $\alpha \in H_k(\cH, \cO)$ as
\begin{equation}\label{eqn:widthDef}
    W_\alpha  := \sup_{\substack{\Gamma  \in \mathscr{S}_k \\ [\Gamma] = \alpha}}\inf_{v \in |\Gamma|}\cM(v)
\end{equation}
where we use the normalization $\inf_\emptyset = -\infty$.

\begin{prop}\label{prop:widthFinite}
    Let $\alpha \in H_*({\cH},\cO) $ satisfy $\iota_*\alpha \in H_*(\cH,\cH\setminus \cA_\ve) \neq 0$ for some $\ve >0$.
    Then, the width $W_\alpha$ exists and is finite. 
\end{prop}
\begin{proof}

    Let $\Gamma = \sum a_i \sigma_i$ be any chain realizing $\alpha \in H_*({\Sigma},O)$.
    Let $\gamma(y)$ be the section as defined in Lemma~\ref{lem:affineSection}.
    Thus, we can define $\tilde{\Gamma} = \sum a_i (\gamma \circ \sigma_i)$ which is a chain in $H_*(\cH, \cO)$ realizing $\alpha \in H_*(\cH, \cO)$

For any $\tilde{\sigma} : \Delta^k \to \cH$, define $\tilde{\sigma}_{t} : \Delta^k \to \cH$ given by 
\[
\tilde{\sigma}_t(p)(x) := \tilde{\sigma}(p)(x) + t\cdot\max\{0, 1 + \la x, \mathfrak{q}(\tilde{\sigma}(p))\rg \}.
\]
Properties $(i)$ and $(ii)$ of Lemma~\ref{lem:tilt=king} shows that $\tilde{\sigma}_i$ and $\tilde{\sigma}_{t}$ are homotopic as a chains. 
These are standard properties; see Hatcher~\cite[Proposition 2.19]{Hatcher}.
We apply this homotopy to $\tilde{\Gamma}$ yielding $\tilde{\Gamma}_t := \sum a_i (\gamma \circ \sigma_i)_t$.
Property $(iii)$ of Lemma~\ref{lem:tilt=king} shows that $\tilde{\Gamma}_t$ is admissible for $t \gg 1$.

    Now consider any non-trivial chain $\Gamma \in Z_K(\cH, \cO)$ such that $[\Gamma] = \alpha$ whose image in $H_*(\cH, \cH\setminus \cA_\ve)$ is non trivial.
    Therefore, we must have some $v \in |\Gamma|$ such that $\mathfrak{q}(v) \in A_\ve$. Lemma~\ref{lem:EnergyInfBound} therefore shows $W \leq E_\ve$.
\end{proof}

\subsection{OT-separability and Palais-Smale}\label{sec:Transverse=PS}
We now examine a Palais-Smale type condition for our max-min constructions. 
The central observation is that sequences that exhibit the failure of the Palais-Smale condition provide $\sln$-degenerations to a pair of cones with a solution of ~\eqref{eqn:HOT}.

\begin{lem}\label{lem:saturatedConvexDefect}
    The space of saturated function $(\ul{\cH}, d)$ is complete and for every $v \in \ul{\cH}$, 
    \[
    \mathfrak{a}(v) = \sup_{w \in \ul{\cH}, w \neq v}\frac{(\delta^+\cM(v;w-v))_+}{\|w-v\|_v},
    \]
    so the convex defect $\mathfrak{a}$ can be computed using only saturated variations.
\end{lem}
\begin{proof}
    Completeness follows from the fact that saturation is retraction on a complete distance non-increasing on a metric space, cf.~Lemma~\hyperref[lem:(C,d)props]{\ref{lem:(C,d)props}$(iii)$-$(iv)$}.

    For $v_t := (1-t)v_1 + tv_0$, the Brunn-Minkowski inequality applied to $K'_{v_s}$ 
    \[
    \delta^+ \log J(v; w-v) \geq \log J(w) - \log J(v).
    \]
    Let $A := \sup_{w \in \ul{\cH}, w \neq v}\frac{(\delta^+\cM(v;w-v))_+}{\|w-v\|_v}$.
    For any $z \in \cH$ with $z \neq v$, let $z_t := v + t(z-v)$ and $\ul{z}_t$ its saturation for $t \leq \min\{1, \|z-v\|_v^{-1}\}$ sufficiently small.
    Therefore,
    \[
    (1-t\|z-v\|_v)v \leq \ul{z}_t \leq (1 + t\|z-v\|_v)v, \qquad \ul{z}_t \leq z_t, \qquad J(z_t) = J(\ul{z}_t).
    \]
    From equation~\eqref{eqn:IJvariations}, the first variation of $\log I$ is linear and decreasing, so 
    \[
    t \delta^+ \log I(v;z-v) + \log J({z}_t) - \log J(v) \leq \delta^+ \cM(v, \ul{z}_t-v) \leq A\|\ul{z}_t - v\|_v \leq At\|z - v\|_v
    \]
    whereby dividing by $t$ and letting $t \to 0$ proves the claim.
   
\end{proof}

The following result is one of the key ingredients for our work.
It can be viewed as a version of Lemma~\ref{lem:(C,d)props} $(iv)$ in the situation where the underlying cones are also allowed to vary.

\begin{prop}\label{prop:PS=SLNdegToPolystable}
    Let $\varphi_j \in \mathbb{P}(\ul{\cH})$ be a sequence such that $|\cM(\varphi_j)| \leq C$ and $\mathfrak{a}(\varphi_j) \to 0$.  Then, either:
    \begin{itemize}
        \item[(i)] $\varphi_j$ sub-sequentially converges to $\varphi\in \mathbb{P}(\ul{\cH})$ a critical point of $\cM$, or,
        \smallskip
        \item[(ii)] there exists an $\sln$-degeneration of $(\tC, \tC')$ to a pair $(\tC_\infty, \tC'_\infty)$ which has a solution of~\eqref{eqn:HOT}. In particular, if $(\tC_\infty, \tC'_\infty)$ is not $\sln$-equivalent to $(\tC, \tC')$, then $(\tC, \tC')$ is not OT-separable.
        \end{itemize}
    Conversely, if $(\tC, \tC')$ is not OT-separable then there exists a sequence $\varphi_j\in \mathbb{P}(\ul{\cH})$ with $|\cM(\varphi_j)| \leq C$, $\mathfrak{a}(\varphi_j) \to 0$ and which does not admit any convergent subsequence. 
\end{prop}
\begin{proof}
Let us first explain the idea of the proof. We first show that if the sequence $\varphi_j\in\bP(\ul{\cH})$ does not converge, then the convex sets $K_j'= \{\varphi_j^* <\tfrac{1}{4}\}$ are not projectively pre-compact in the Hausdorff topology.
By considering the John ellipsoid of $K_j'$, we generate a sequence of positive-definite symmetric matrices $L_j$ with $\|L_j\|_{\mr{HS}} +\|L_j^{-1}\|_{\mr{HS}} \rightarrow +\infty$ (by symmetry, we can drop the transpose from the dual action).  
After rescaling $L_j$ to have determinant $1$, this yields an $\sln$-degeneration $(L_j\tC_j, L_j^{-1}\tC_j')\rightsquigarrow(\tC_{\infty}, \tC_{\infty}')$. 
Since $\mathfrak{a}(\varphi_j)\rightarrow 0$, one expects that the cone pair $(\tC_{\infty}, \tC_{\infty}')$ admits a solution of~\eqref{eqn:HOT}. 
The main difficulty is to show: $(i)$ that $(\tC_{\infty}, \tC_{\infty}')$ is an open cone pair, and $(ii)$ the renormalized functions $\widehat{\varphi}_j(X):= \varphi(L_{j}^{-1}X)$ converge in sufficiently strong topologies such that $0=\lim_{j\rightarrow \infty}\mathfrak{a}(\widehat{\varphi}_j)=\mathfrak{a}(\lim_{j\rightarrow \infty}\widehat{\varphi}_j)$.

    Let $Z_j + L_j\bB$ be the John ellipsoid of $K'_j= \{\varphi_j^* <\tfrac{1}{4}\}$.  
Consider the sequence $(\tC_j, \tC'_j) := (L_j\tC, L_j^{-1} \tC')$ with the potentials $\widehat{\varphi}_j(X) := \varphi(L_j^{-1}X)$ so that $\widehat{\varphi}_j^*(Y) = \varphi^*(L_jY)$. 
We have $\widehat{K}_j' := L_j^{-1} K'_j$ which has John ellipsoid $L_j^{-1}Z_j + \bB^*$.
From Lemma~\ref{lem:saturatedConvexDefect}, it is sufficient to examine saturated degenerations, so we can describe $\varphi$ from $K'$. 
By the $\textup{GL}(n)$-invariance of $\cM$, we know $\cM_{(\tC,\tC')}(\varphi_j) = \cM_{(\tC_j,\tC_j')}(\widehat{\varphi}_j)$ and $\mathfrak{a}_{(\tC,\tC')}(\varphi_j) = \mathfrak{a}_{(\tC_j,\tC_j')}(\widehat{\varphi}_j) \leq \frac{1}{2j}$ by assumption.
By standard compactness results, we have $(\tC_j, \tC_j') \rightarrow (\tC_{\infty},\tC_{\infty}')$ in the sense of Hausdorff convergence on compact sets for some convex cones $(\tC_{\infty},\tC_{\infty}')$.  Note that $\tC_{\infty}$ may in principle have empty interior.
The key property we will prove is the containment
    \begin{equation}\label{eqn:uniformRoundPS}
        \frac{1}{2n}\bB \cap \tC_j \subset \widehat{K}_j \subset R \bB \cap \tC_j
    \end{equation}
    for a uniform constant $R$.

We first analyze some consequences of being in vertical John position. Since $0\in \del K_j$, and $K_j\subset Z_j + nL_j \bB^*$, there is some $Y_j\in \bB^*$ such that $0= Z_j+nL_jY_j$, from which it follows that $|L_j^{-1}Z_j|\leq n$.  Therefore 
    \begin{equation}\label{eq: normSectionControl}
        L_j^{-1}Z_j+\bB^* \subset \widehat{K}_j'\subset L_j^{-1}Z_j + n\bB^* \quad \text{ with }  |L_j^{-1}Z_j|\leq n ,
    \end{equation} 
    and so $|\bB^*| \leq |\widehat{K}'_j| \leq n^n |\bB^*|$.
    It follows that $J_{(\tC_j, \tC'_j)}(\widehat{\varphi}_j)$ is uniformly bounded.
    Therefore, thanks to the assumption that $\cM(\varphi_j)$ is uniformly bounded, we deduce that $I_{(\tC_j, \tC'_j)}(\widehat{\varphi}_j)$ is uniformly bounded as well.

    For $X\in \tC_j$ we have
    \[
        \phi_{\widehat{K}'_j-\tC_j^\vee}(X) :=  \sup_{Y_1 \in \widehat{K}'_j}\,\sup_{Y_2  \in - \tC_j^\vee} (\la X, Y_1\rg   + \la X, Y_2\rg ) =\sup_{Y \in \widehat{K}_j'} \la X, Y\rg \leq 2n|X|
    \]
    using $\la X, Y_2\rg \leq 0$ in the first equality and~\eqref{eq: normSectionControl} in the second inequality. 
    Therefore, we deduce that $\frac{1}{2n}\bB \cap \tC_j \subset \widehat{K}_j$.
    By the convexity of $\widehat{K}_j$ and the uniform bound $I_{(\tC_j, \tC'_j)}(\widehat{\varphi}_j) \sim 1$ we obtain $|\widehat{K}_j| \leq C$.
    Therefore, if $\mr{int}(\tC_\infty) \neq \emptyset$, we deduce the upper bound in equation~\eqref{eqn:uniformRoundPS} for some $R$ based on $\sup_j \cM(\varphi_j)$.
    
    Suppose that $\mr{int}(\tC_\infty) = \emptyset$.
    We will show that in this case $\lim_j \mathfrak{a}(v_j) > 0$, contradicting the assumptions of the proposition.
    Let $R_j := |\tC_j \cap \bB|^{-\frac{1}{2n}}$.
    Since we are assuming that $\mr{int}(\tC_\infty) = \emptyset$, we have $R_j\to \infty$. 
    By homogeneity we have $|\tC_j \cap R_j \bB| = |\tC_j \cap \bB|^{\frac{1}{2}} \to 0$.
    Note that since $\widehat{K}_j \subset \tC_j$, we have $|\widehat{K}_j \cap R_j \bB| \to 0$.
    On the other hand, thanks to the bound $I_{(\tC_j, \tC'_j)}(\widehat{\varphi}_j) \sim 1$ we have $|\widehat{K}_j| \geq c > 0$, so
    \begin{equation}\label{eq:MassConcentratesAtInfinity}
    \lim_{j \to \infty}\frac{|\widehat{K}_j \setminus {R_j}\bB|}{|\widehat{K}_j|} \to 1.
    \end{equation}
    In other words, the mass of $\widehat{K}_j$ is concentrating at infinity.
    Let $\widehat{Z}_{j} = L_j^{-1}Z_j$ be the center of the John ellipsoid of $\widehat{K}_j'$ and define $\ell_j(X) := \la X, \widehat{Z}_j\rg$.
    Using~\eqref{eq: normSectionControl} we can bound
    \begin{equation}\label{eqn:shrinkingVariationineq}
    \phi_{\widehat{K}'_j}(X) = \sup_{Y \in \widehat{K}'_j} \la X, Y\rg \geq \sup_{Y \in \widehat{Z}_j + \bB^*}\la X, Y\rg =\ell_j(X) + |X|.
    \end{equation}
    
    We now demonstrate how $\mathfrak{a}(\varphi_j) \to 0$ ensures that $\tC_\infty$ has non-empty interior.
    We argue by contradiction, exhibiting an explicit deformation with definite derivative were this not to hold.
    Note that since $v_j$ is saturated, we have the formula $\widehat{v}_j(X) =  \phi_{\widehat{K}'_j}(X)$.
    Define
    \[
    \widetilde{v}_j(X)  := \max\{0, (1-\tfrac{1}{2R_j})\phi_{\widehat{K}'_j}(X) + \tfrac{1}{2R_j}\ell_j(X)\},
    \]
    which is the support function for 
    \begin{equation}\label{eq:modifiedSectionPerturbation}
    \widetilde{K}'_j := \mr{ConvexHull}\{0, (1-\tfrac{1}{2R_j})\widehat{K}'_j + \tfrac{1}{2R_j}\widehat{Z}_j\} \subset \ol{\tC'_j}.
    \end{equation}
    Since the family $(1-t)\widehat{v}_j + t\widetilde{v}_j$ lies in $\cH_{(\tC_j, \tC'_j)}$ for $t < 1$, we can use this curve to bound $\mathfrak{a}(\widehat{v}_j)$.
    By~\eqref{eqn:shrinkingVariationineq}, we have 
    \begin{equation}\label{eqn:shrinkingVariationineqPartII}
    (1-\tfrac{1}{2R_j})\phi_{\widehat{K}'_j}(X) + \tfrac{1}{2R_j}\ell_j(X) \leq \phi_{\widehat{K}_{j}'}(X)-\frac{|X|}{2R_j} = \widehat{v}_j-\frac{|X|}{2R_j}.
    \end{equation}
    Since $\widehat{v}_j>0$ we conclude that $ 0\leq \widetilde{v}_j <\widehat{v}_j$, from which we obtain $\| \widetilde{v}_j -\widehat{v}_j\|_{\widehat{v}_j}= \sup_{P_j}\left(1-\frac{\widetilde{v}_j}{\widehat{v}_j}\right) \leq 1$. 
    Let $E_j := \{x \in P_j : \frac{|(x,1)|}{\widehat{v}_j(x)} > R_j\}$.  By~\eqref{eq:MassConcentratesAtInfinity} we have $\lim_{j\rightarrow \infty} \frac{\int_{E_j}\widehat{v}_j(x)^{-n}\,dx}{\int_{P_j}\widehat{v}_j(x)^{-n}\,dx} = 1$.
    Using~\eqref{eqn:shrinkingVariationineqPartII} we see that on $E_j$, 
    \[
    (1 - \tfrac{1}{2R_j})\widehat{v}_j(x) + \tfrac{1}{2R_j}\ell_j(x) \leq  \widehat{v}_j(x)- \frac{|(x,1)|}{2R_j}\leq  \frac{1}{2}\widehat{v}_j(x).
    \] 
    We can now compute
    \[
    \delta^+ \log I_j(\widehat{v}_j;\widetilde{v}_j -\widehat{v}_j) =n \frac{\int_{P_j}(\widehat{v}_j(x)-\widetilde{v}_j(x) )\widehat{v}_j(x)^{-(n+1)}\,dx}{\int_{P_j}\widehat{v}_j(x)^{-n}\,dx}.
    \]
    Now, since $\widetilde{v}_j - \widehat{v}_j\leq 0$ on $P_j$ and $\widetilde{v}_j - \widehat{v}_j\leq-\frac{1}{2}\widehat{v}_{j}$ on $E_j$, we obtain
    \[
    \delta^+ \log I_j(\widehat{v}_j;\widetilde{v}_j - \widehat{v}_j)\geq \frac{n}{2}\frac{\int_{E_j}\widehat{v}_j(x)^{-n}\,dx}{\int_{P_j}\widehat{v}_j(x)^{-n}\,dx} = \frac{n}{2} + {o(1)}.
    \]
    We now bound the variation of $J_{(\tC_j,\tC_j')}$ along the family $\widetilde{v}_{j,t} :=(1 - t)\widehat{v}_j + t\widetilde{v}_j$.
    Recall that $\widetilde{v}_j$ is the convex support function of $\widetilde{K}_j'$ as defined in~\eqref{eq:modifiedSectionPerturbation}. Standard properties of convex support functions yield that $\widetilde{v}_{j,t}$ is the convex support function of
    \[
    \widehat{K}_t={\rm ConvexHull}\left\{(1-t)\widehat{K}_j',\left((1-\frac{t}{2R_j})\widehat{K}_j'+\frac{t}{2R_j}\widehat{Z}_j\right)\right\}.
    \]
    It follows that $|\widehat{K}_t| \geq \left(1-\frac{t}{2R_j}\right)^n|\widehat{K}_j'|$. 
   Thus, we have
    \[
    J_{(\tC_j,\tC_j')}(\widetilde{v}_{j,t} ) \geq (1 - \tfrac{t}{2R_j})^n J_{(\tC_j,\tC_j')}(\widehat{v}_j) \qquad \implies \qquad 
    \delta^+ \log J_{(\tC_j,\tC_j')}(\widehat{v}_j;\widetilde{v}_j - \widehat{v}_j) \geq -\frac{n}{2R_j}= o(1).
    \]
    Therefore, we see $\delta^+ \cM_j(\widehat{v}_j; \widetilde{v}_j - \widehat{v}_j) > \frac{n}{4} > 0$, a contradiction.
    Thus, we conclude the outer containment in equation~\eqref{eqn:uniformRoundPS}.

    We can now finish the proof.
    Since $\widehat{K}_j, \widehat{K}_j'$ have diameters bounded from above, we can extract a convergent subsequence in the Hausdorff topology $(\widehat{K}_j,\widehat{K}_j') \rightarrow (\widehat{K}_\infty, \widehat{K}'_\infty)$.
    Since $\widehat{K}'_j - \tC^\vee_j = \widehat{K}_j^\circ$ for all $j$ thanks to saturation, it follows that $\widehat{K}_{\infty}'-\tC^{\vee}_{\infty} = (\widehat{K}_{\infty})^{\circ}$.
    By the geometric characterization of Legendre duality in Corollary~\ref{cor:duality=legendre}, we obtain that $\varphi_{\infty}:= \lim_j \widehat{\varphi}_j$ satisfies $\varphi_{\infty}\in \ul{\cH}_{\infty}$ and $\varphi_{\infty}^{*}= \lim_{j\rightarrow \infty}\varphi_j^*$.
    The containment $\widehat{K}_j \subset R\bB$ from~\eqref{eqn:uniformRoundPS} together with the equality $\widehat{K}_j' -\tC_j=\widehat{K}_j^{\circ}$ yields $R^{-1}\bB^*\cap \tC_j\subset \widehat{K}_j'$.
    Combining this bound with \eqref{eqn:uniformRoundPS} and ~\eqref{eq: normSectionControl} yields the bounds
    \begin{equation}\label{eq:quadraticGrowthPS}
    C^{-1}|X|^2 \leq \widehat{\varphi}_j(X) \leq C|X|^2 \qquad \text{and}\qquad 
    C^{-1}|Y|^2 \leq \widehat{\varphi}^*_j(Y) \leq C|Y|^2
    \end{equation}
    for a uniform constant $C>0$.  We claim that we also have $|\nabla \widehat{\varphi}^*_j(Y)|^2 \leq C\,\widehat{\varphi}^*_j(Y)$.
    Indeed, since both sides of the estimate are homogeneous of degree $2$, it suffices to check the estimate on $\widehat{K}_j'= \{\widehat{\varphi}^*_j\leq \tfrac{1}{4}\}$.
    By $2$-homogeneity, we have $ \nabla \widehat{\varphi}_j^* (2\widehat{K}_j')\subset  \ol{\widehat{K}_j}$, and so the desired estimate follows from the upper bound in~\eqref{eqn:uniformRoundPS}.

    Given $\psi \in \cH_{\infty}$ with $\|\psi\|_{\varphi_{\infty}}=1$, we have $\psi(X) \leq \varphi_{\infty}(X) \leq C|X|^2 $. 
    Let $\Psi_\ve(X) := \inf_{Z \in \tC_\infty}\{\psi(Z) + \ve^{-1}|X-Z|^2\}$.
    Then, $\Psi_\ve : \bR^n \to \bR$ is a finite, positive $2$-homogeneous convex function such that $0 \leq \Psi_\ve \leq \ve^{-1}|X|^2$ and $\Psi_\ve \to \psi$ locally on $\tC_\infty$ as $\ve \to 0$.
    Therefore, for fixed $\ve$, $\Psi_\ve\vert_{\tC_j}$ provides an admissible variation with a uniform quadratic bound.
    We now let $\Psi = \Psi_\ve$ and take $j \to \infty$ and then $\ve \to 0$. 
    By the dominated convergence theorem we have
    \[
    \int_{\tC_j}\Psi(X)e^{-\widehat{\varphi}_j(X)}dX \rightarrow \int_{\tC_\infty}\Psi(X)e^{-\widehat{\varphi}_{\infty}(X)}dX.
    \]
    Similarly, we have
    \[
    \Psi(\nabla \widehat{\varphi}^*_j(Y)) \leq C|\nabla \widehat{\varphi}^*_j(Y)|^2 \leq C^2\widehat{\varphi}^*_j(Y).
    \]
    Thus, another application of the dominated convergence theorem, together with a.e.~convergence of gradients for converging convex functions (cf.~\cite[\S~1.1.1]{Gutierrez}), implies
    \[
      \int_{\tC_j'}\Psi(\nabla \widehat{\varphi}^*_j(Y))e^{-\widehat{\varphi}^*_j(Y)}dY \rightarrow  \int_{\tC_\infty'}\Psi(\nabla \widehat{\varphi}^*_{\infty}(Y))e^{-\widehat{\varphi}^*_\infty(Y)}dY.
    \]
    Finally, we obtain
    \[
   0=\lim_{j\rightarrow \infty}\mathfrak{a}(\varphi_j) \geq C^{-1}\lim_{j\rightarrow \infty} \delta^{+}\cM(\widehat{\varphi}_j;\Psi|_{\tC_j}-\widehat{\varphi}_j)= C^{-1}\delta^+\cM(\varphi_{\infty};\psi-\varphi_{\infty}).
    \]
    Applying Proposition~\ref{prop:crit->HOT} completes the proof of the first part of the proposition.

    For the converse, assume that $\varphi_\infty$ solves~\eqref{eqn:HOT} on the pair $(\tC_\infty, \tC'_\infty)$ which is a non-trivial degeneration of $(\tC, \tC')$ by a sequence $M_j \in \sln$. 
    Then, consider $\varphi_t$ defined on $\tC_j := M_j\cdot \tC$ by $\varphi_j(X) := (\phi_{K^\circ_j}(X))^2$ where $K^\circ_j := (K_\infty^\circ \cap \tC'_j) - \tC^\vee_j$.
    This section represents the saturation of $K^\circ$ restricted to $\tC'_j$.
    By construction, $K_j \to K_\infty$ and $K^\circ_j \to K^\circ_\infty$ in the Hausdorff topology and each of $K_j,K'_j$ are uniformly bounded. 
    By construction, the limit is critical, so the argument of Proposition~\ref{prop:PS=SLNdegToPolystable} shows that along this sequence, a posteriori, the associated measures $e^{-\varphi_j}\chi_{\tC_j}$ and $e^{-\varphi^*_j}\chi_{\tC'_j}$ converge to the measures on the limiting pair.
    We claim that $\mathfrak{a}(v_j) \to 0$ along this sequence.
    Otherwise, there is some $\ve_0 > 0$ and competitors $\Psi_j$ such that 
    \[
    \delta_+ \cM(\varphi_j;\Psi_j - \varphi_j) \geq \ve_0 \|\Psi_j - \varphi_j\|_{\varphi_j}
    \]
    and we may replace $\Psi_j$ with some $\varphi_j + t(\Psi_j-\varphi_j)$ for some $t > 0$ such that $\|\Psi_j - \varphi_j\|_{\varphi_j} = \frac{1}{2}$, so $\frac{1}{2}\varphi_j \leq \Psi_j \leq \frac{3}{2} \varphi_j$ on the normalized links. 
    The links of $\tC_j^\vee$ and $\tC'_j$ are compact convex sets in $\{Y_n = 1\}$ converging to the compact links of $\tC_\infty^\vee$ and $\tC'_\infty$, from the previous argument, we can extract some $\Psi_\infty \in \cH_{(\tC_\infty, \tC'_\infty)}$ such that 
    \[
    \delta_+ \cM_{(\tC_j, \tC'_j)}(\varphi_j ; \Psi_j - \varphi_j) \to \delta_+ \cM_{(\tC_\infty, \tC'_\infty)}(\varphi_\infty ; \Psi_\infty - \varphi_\infty)
    \]
    with $\|\Psi_\infty - \varphi_\infty\|_{\varphi_\infty} = \frac{1}{2}$.
    The lefthand side is bounded below by $\frac{\ve_0}{2} > 0$, contradicting the criticality of $\varphi_\infty$,
    Therefore, pulling back $\varphi_j$ by $\varphi(M_j x)$ produces a family $\varphi_j \in \ul{\cH}$ on the fixed pair $(\tC, \tC')$ which, when John normalized, degenerates to $(\tC_\infty, \tC'_\infty)$ with $\cM(\varphi_j) \to \cM(\varphi_\infty)$ and satisfies $\mathfrak{a}(\varphi_j) \to 0$.
\end{proof}

\subsection{Existence of homogeneous optimal transport maps}
We now prove Theorem~\ref{thm: HotExistIntro} by showing that $\cH$ satisfies that the appropriate analytic properties to conclude that the max-min width $W$ is achieved by a critical point of $\cM$.
The main technical point is that $\bP(\cH)$ is not a Banach manifold, so the classical pseudo-gradient deformation of Palais~\cite{Palais}. 
We require the following pulltight operation.
\begin{lem}\label{lem:pulltight}
    There exists a sequence $v_j \in \ul{\cH}$ such that $\cM(v_j) \to W_\alpha$ and $\mathfrak{a}(v_j) \to 0$.  
\end{lem}
\begin{proof}
    This argument follows the standard pulltight argument by Palais~\cite{Palais}, using a weaker metric deformation instead of the pseudo-gradient supplied by~\cite[Definition 2.1]{TwoItalians} and~\cite{ThreeItalians}.
    Since we use a max-min instead of min-max conventions, the inequalities will be reversed accordingly.
    Since saturation is an energy increasing homotopy which preserves the Steiner point, we may assume all sweepouts are saturated.

    We apply the following construction of the weak slope from Degiovanni-Marzocchi~\cite{TwoItalians} to $\cM$ on the complete metric space $(\ul{\cH}, d)$.
    We recall their definition of $|d^+\cM(\varphi)|$ as the supremum of all $\sigma \geq 0$ for which there exists some $r > 0$ and a continuous map $H : B_r(\varphi)\times [0,r] \to \ul{\cH}$ such that
    \[
     d(H(u,t) , u) \leq t \qquad  \text{and}\qquad \cM(H(u,t))\geq \cM(u) + \sigma t. 
    \]
    Corvellec~\cite[Theorem 2.4]{OneItalian?} quantifies an energy interval without critical points. 
    Precisely, this result states that if $A < B$ and $\sigma, \delta > 0$ and 
    \[
    A - \delta \leq \cM(\varphi) \leq B + \delta \qquad \implies \qquad |d^+\cM(\varphi)| > \sigma,
    \]
    then there is a continuous deformation $\eta : \ul{\cH} \times [0,1] \to \ul{\cH}$ satisfying:
    \begin{itemize}
        \item $d(\eta(v,t), v) \leq \frac{B-A}{\sigma }t $,
        \item $\cM(\eta(\varphi, t)) \geq \cM(\varphi)$ with equality if and only if $\eta(\varphi,t)= \varphi$, and 
        \item if $A \leq \cM(\varphi) \leq B$, then $\cM(\eta(\varphi,t)) \geq \cM(\varphi) + (B - \cM(\varphi))t$.
    \end{itemize}
    
    To apply this result, we must show that $|d^+\cM(v)|$ is controlled by $\mathfrak{a}(\varphi)$.
    Choose some $0 < a < \mathfrak{a}(v)$ and $w \in \ul{\cH}$ such that $\delta^+ \cM(v; w-v) > a \|w - v\|_{v}$.
    We can replace $w$ with $(1-s)v + sw$ for $0 < s \ll 1$ without changing the inequality, so we may assume that $r := \|w - v\|_v$ is sufficiently small.
    We consider $U$ to be a neighborhood of $v$ sufficiently small.
    From the completeness of $(\cH, d)$ and the convergence of the first variation from Lemma~\hyperref[lem:(C,d)props]{\ref{lem:(C,d)props}(iv)}, we can assume $U$ satisfies the following: for every $u \in U$, $\|w-u\|_u \leq 2r$, the curve $u_\theta := (1 - \theta)u + \theta w$ lies in $U$ for $\theta$ small, $\|w - u\|_{u_\theta} < 4r$, and $\delta^+ \cM(u_\theta; w- u_\theta) \geq\frac{ar}{2}$.  
    For small enough $t > 0$, set $s = \frac{t}{4r}$ and define $H(u,t) := (1-s)u + sw$.
    We can bound
    \begin{align*}
    d(H(u,t),u) &\leq \left\|\log \frac{(1-s)u + sw}{u}\right\|_{L^\infty(P)} \\
    &= \left\|\int_0^s\frac{w-u}{u_\theta}\,d\theta\right\|_{L^\infty(P)} \leq \int_0^s\|w-u\|_{u_\theta}\,d\theta \leq 4rs = t.
    \end{align*}
    On the other hand, $\frac{d}{d\theta}\cM(u_\theta) = \frac{1}{1-\theta} \delta^+\cM(u_\theta; w - u_\theta)$ shows that
    \[
    \cM(H(u,t)) - \cM(u) = \int_0^s\frac{1}{1-\theta}\delta^+ \cM(u_\theta;w-u_\theta)\,d\theta  \geq\frac{ar}{2}s = \frac{at}{8}.
    \]
    Therefore, $H$ as defined satisfies the condition for the weak slope $|d^+\cM(v)| >\frac{a}{8}$.
    Since $a < \mathfrak{a}(v)$ was arbitrary, we conclude the uniform constant $|d^+\cM(v)| \geq \frac{\mathfrak{a}(v)}{8}$.

    We now argue by contradiction, using an argument like that of Palais~\cite{Palais}.
    For a sequence of admissible sweepouts $\Gamma_j \in \mathscr{S}_k$ such that $[\Gamma_j] = \alpha$, and $v_j := \mr{argmin}_{v \in |\Gamma_j|}\cM(v)$ satisfies $\cM(v_j) \geq W_\alpha- \tfrac{1}{j}$.
    By admissibility, $\inf_{w \in |\p \Gamma|}\cM(w) \geq W_\alpha + 10$.
    If any sequence $\mathfrak{a}(v_j) \to 0$, then we are done.
    Suppose otherwise, then we may assume that there exists some $a > 0$ such that for all $v$  with $|\cM(v) - W_\alpha| \leq 3\delta$, we have $\mathfrak{a}(v) >a$ for some $0 < \delta < 1$. 
    Thus, from above $|d^+\cM(v)| > ca$ for some uniform $c > 0$. 
    Applying~\cite[Theorem 2.4]{OneItalian?} with $A = W_\alpha - \delta$, $B = W + \delta$, and $\sigma = \frac{a}{16}$ produces some $\eta_t$ such that $\eta_0$ is the identity and $\cM(\eta_1(v)) \geq W + \delta$ for any $v$ with $\cM(v) \in [W_\alpha - \delta, W_\alpha + \delta]$.
    Let $\tilde{\eta}(v,t)$ be $\eta_{\chi(\cM(v))t}$ compose with a cutoff $\chi$ which is $1$ on $[W_\alpha - \delta, W_\alpha + \detla]$ and $0$ outside $(W_\alpha - 2\delta, W_\alpha + 2\delta)$.
    Therefore, applying this to $\Gamma_j$ for $j \gg 1$ fixes its boundary by admissibility, the image of $\Gamma_j$ under this pulltight is homologous in $H_*(\cH, \cO)$.
    However, $\min_{v \in |(\tilde{\eta}_1)_*{\Gamma}_j|}\cM(v) > W + \delta$, which violates Proposition~\ref{prop:widthFinite}. 

\end{proof}
We may now prove the main result of this section.

\begin{theorem}\label{prop:WidthAchieved}
    Let $(\tC, \tC')$ be a pair of positively aligned, pointed cones.
    If the pair is linked and OT-separable, then there exists $\varphi \in \cH_{(\tC, \tC')}$ solving~\eqref{eqn:HOT}.
\end{theorem}
\begin{proof}
    From Lemma~\ref{lem:pulltight}, we have a sequence $\varphi_j$ such that $\cM(\varphi_j) \to W_\alpha$ and $\mathfrak{a}(\varphi_j) \to 0$.
    Applying Proposition~\ref{prop:PS=SLNdegToPolystable} shows that either $\varphi_j \to \varphi$ converges to a \hyperref[eqn:HOT]{H.O.T.} solution, or produces a \hyperref[eqn:HOT]{H.O.T.} solution $\varphi_\infty$ on $(\tC_\infty, \tC'_\infty)$ from an $\sln$-degeneration.
    In this latter case, OT-separability implies that $(\tC, \tC') = A \cdot (\tC_\infty, \tC'_\infty)$, so $\varphi_\infty \circ A$ solves~\eqref{eqn:HOT}.

\end{proof}

\begin{remark}\label{rmk:C->CDegen}
The OT-separability property provides a suitable substitute for the Palais-Smale condition for $\cH_{(\tC, \tC')}$, but in general OT-separability is a weaker notion. 
\end{remark}

\subsection{Verifying OT-separability}\label{sec:OTsep}
We now show how one can practically verify OT-separability for a given pair and show that this property is generic.
We show that the only degenerations that can violate OT-separability are those where $K'$ collapses to a height 1 object contained in the boundary along the shared supporting hyperplane of $\tC'$ and $\tC^\vee$.
In particular, if $\tC'$ and $\tC^\vee$ do not share any supporting hyperplanes, then the pair is OT-separable.
\begin{defn}\label{def:residue}
Let $H \subset \p {\tC'} \cap\{0\leq Y_n \leq 1\}$ 
be a convex set with $0\in \ol{H}$ and $\ol{H}\cap \{Y_n =1\} \ne \emptyset$. 
A non-negative convex function $h$ on $P$ is a \textbf{residue} if $h(x) = \phi_H(x)$. 
We may further consider its saturation, given by $\ul{H} = (H - \tC^\vee) \cap \tC'$. 
\end{defn}

 Geometrically, the notion of a residue can be realized as $K^\circ_\infty := H - \tC^\vee$ and $K_\infty := (K_\infty^\circ)^\circ$.  This section is degenerate since $K_\infty$ doesn't contain an open neighborhood of the origin within $\tC$. 
 The following estimate, arising from the variation of $I$ by a constant shows that degenerations of $K'$ must accumulate at the boundary as $J \to 0$ if the sequence is becoming convex critical. 
\begin{lem}\label{lem:IvariationEstimate}
    For any positive convex function $v$ on $P$, we have
    \[
    \inf_P v \leq \frac{\int_P v(x)^{-n}\,dx }{\int_P v(x)^{-(n+1)}\,dx}\leq n (\inf_P v) .
    \]
\end{lem}
\begin{proof}
The lower bound inequality $(\inf_P v)\int_P v(x)^{-(n+1)}\,dx \leq \int_P v(x)^{-n}\,dx$
    follows pointwise and is sharp as seen by constant functions.

    Let $a \in P$ and define $T_s(x) = (1-s)a + sx$ for $0 \leq s \leq 1$, so $T_s(P) \subset P$ by convexity.
    For any $q$, we can bound
    \begin{align*}
    \int_P v(x)^{-q}\,dx &\geq \int_{T_s(P)} v(x)^{-q}\,dx = s^{n-1}\int_P v(T_s(x))^{-q}\,dx \\
    &\geq s^{n-1}\int_P((1-s)v(a) + sv(x))^{-q}\,dx =: G_q(s)
    \end{align*}
    for any $s \in [0,1]$.
    When $s = 1$, all the above are equalities, so $G_q(1) = \int_P v(x)^{-q}\,dx$ and $G_q'(1) \geq 0$.
    Differentiating $G$ computes 
    \begin{align*}
        G'_q(1) &= (n-1)\int_P v(x)^{-q}\,dx - q\int_P (v(x) - v(a))v(x)^{-(q+1)}\,dx\\
        &=  (n-1 - q)\int_P v(x)^{-q}\,dx + qv(a)\int_Pv(x)^{-(q+1)}\,dx,\\
    \end{align*}
    meaning that 
    \[
    (q -n + 1)\frac{\int_P v(x)^{-q}\,dx}{\int_P v(x)^{-(q+1)}\,dx} \leq qv(a),
    \]
    whereupon setting $q = n$ and $a \in P$ achieving the infimum of $v$ proves the claim.
\end{proof}

We now show that degenerations of normalized functions yield residues.
\begin{prop}\label{prop:residueLimit}
Consider a sequence of saturated $v_j$ such that $v_j(0) = 1$, $m_j :=\inf_P v_j \leq \frac{1}{j}$, and $m_j(\delta \cM_{v_j}[1]) \to 0$. 
Then, there exists some residue $H$ such that, up to a subsequence, $v_j \to \phi_{\ul{H}}$ uniformly. 

\end{prop}
\begin{proof}
    From equation~\eqref{eqn:IJvariations}, we have $|\delta \cM_{v_j}[\pm 1]| = \left| \alpha_v - \beta_v \right| =:\ve_j$ which we assume is not growing too quickly. 
    Lemma~\ref{lem:IvariationEstimate} shows that 
    \[
    \frac{c}{m_j} \leq \alpha_{v_j} \leq \frac{C}{m_j} \qquad \implies \qquad \frac{\ve_j}{\alpha_{v_j}}  \lesssim m_j \ve_j \to 0,
    \]
    using the assumption $m_j \ve_j \to 0$.
    Therefore, for large enough $j$, we have $\frac{1}{2}\alpha_{v_j} \leq  \beta_{v_j} \leq 2 \alpha_{v_j}$, so $\beta_{v_j}^{-1} \leq C m_j$. 
    From the definition of $\beta_v$, this shows that 
    \[
    \frac{1}{\beta_{v_j}} = \frac{\int_{\Omega_j}(-u_j(y)-\rho_{\Sigma}(y))\,dy }{|\Omega_j|} \leq C m_j \to 0.
    \]
    However, $K'$ always contains some point $(y', 1) \in \tC'$ and no points $(y', y_n)$ with $y_n > 1$, so there exists some $K'_\infty = \lim K'_j$ with $|K'_\infty| = 0$, $K'_\infty \subset \ol{\tC'}$, and containing some point $(y',1)$, meaning $K'_\infty$ must be a compact subset of an exposed face of $\tC'$ and $K^\circ = K' - \tC^\vee$, which is therefore a residue.

\end{proof}
The above estimate is quite strong, and in particular shows that any Palais-Smale sequence $v_j$ such that $\delta \cM_{v_j} \to 0$ for all variations must converge uniformly to a residue.
The vast majority of residues imply the energy blows up to $\pm \infty$.
In particular, the only case where $v_j$ can degenerate to a residue $\phi_H$ with finite energy is when $H$ lies on a shared exposed face of both $\partial \tC'$ and $\partial \tC^\vee$.

\begin{defn}
We say that a residue $H$ is \textbf{obtuse} if for every saturated convergence $K'_j \to \ul{H}$ in the Hausdorff distance, $\cM(\varphi_j) \to \infty$. 
\end{defn}
By Proposition~\ref{prop:residueLimit}, if every residue is obtuse, then the pair is OT-separable.

\begin{cor}\label{cor:Crossing=OTsep}
    If $|(\ul{H} - \tC^\vee) \cap \tC'| > 0$, then $H$ is obtuse.
    In particular, if $\tC^\vee$ and $\tC'$ do not have a shared supporting hyperplane that meets $\p \tC' \cap \p \tC^\vee$, then they are OT-separable.
\end{cor}

\begin{proof}
Assume $\ul{H}$ is saturated and has positive volume, so any sequence $v_j$ such that $K'_j \to \ul{H}$ satisfies $J(v_j) \to |\ul{H}| > 0$. 
We now show that $I(v_j) \to \infty$. 
By saturation, we know that $\p v(P) \subset \Omega$, which is uniformly bounded since $\Omega \subset \Sigma$. 
Therefore, we have uniform gradient bounds on $v_j$, so applying Lemma~\ref{lem:IJonlinks}, $I(\ul{v_j}) \to \infty$ if $\inf_P v_j \to 0$ (cf.~\cite[Lemma 1.1]{TristanBenjy}).
If a saturated residue $\ul{H}$ contains any interior point of $\tC^\vee$, then $|\ul{H}| > 0$ and the above implies it is obtuse.
Any saturated $\ul{H}$ meeting $p \in \p \tC^\vee \cap \p \tC'$ which does not have a shared supporting hyperplane must meet $\tC^\vee$. 
\end{proof}

We can now prove Theorem~\ref{thm: HotExistIntro}.
\begin{proof}[Proof of Theorem~\ref{thm: HotExistIntro}]
    Proposition~\ref{prop: Sep+Acute=HOT} covers the acute first alternative in $(ii)$.
    Corollary~\ref{cor:Crossing=OTsep} shows that $(\tC, \tC')$ are OT-separable. 
    When $(\tC, \tC')$ are crossing, then $\widetilde{H}_*(\tC'\setminus \tC^\vee) \neq 0$ implies that the cones are linked.
    Therefore, Theorem~\ref{prop:WidthAchieved} proves the second alternative in $(ii)$.
\end{proof}

\subsubsection{Examples}\label{sec:examples}
We now illustrate some examples showing how to effectively identify OT-separability. 
Proposition~\ref{prop:residueLimit} shows that to test OT-separability, we must only test degenerations that arise from residues contained in $\p \tC'$ that meet $\p \tC^\vee$.
This drastically reduces the number of degenerations to be checked, often there are finitely many or none.
    For example, strictly acute and obtuse cones, i.e.~$\Sigma \Subset Q$ and $Q \Subset \Sigma$ are OT-separable. 
    In fact, such cones satisfy a stronger, more \textit{obviously}, $C^{1,1}$-regularity, as they have no possible degenerations.
\begin{lem}\label{lem:AcuteObtuseNodegen}
    Strictly acute or obtuse cones have no positively aligned degenerations.
\end{lem}
\begin{proof}
    Let $M_j$ be an $\sln$-degeneration and $S_j := \frac{M_j^{-T}}{\|M_j^{-T}\|_{\mr{HS}}}$.
    Because $\|M_j\|_{\mr{HS}} \to \infty$, we have that $S := \lim_{j\to\infty} S_j$ must have rank between $1$ and $n-1$. 
    We may assume $\tC_j := M_j \cdot \tC \to \tC_\infty$ and $\tC'_j := M_j^{-T} \cdot \tC' \to \tC'_\infty$ and label $Q_\infty := \tC^\vee_\infty \cap \{Y_n = 1\}$ and $\Sigma_\infty := \tC'_\infty \cap \{Y_n = 1\}$. 

    We first handle the strictly acute case.
    Suppose that $\ker S \cap \tC' = \{0\}$. 
    Therefore, on $\ol{\tC'} \cap \partial \bB^*$, $S$ is bounded away from 0, so $\tC'_\infty \subset \mr{im}\,S$ and thus $\tC'_\infty$ has no interior. 
    Otherwise, let $Y \in \ker S \cap \ol{\tC'}$, so by acuteness, $Y +\ve \bB^*\in \tC'$ for some small, but uniform, $\ve$.
    By linearity, we have that $S_j(Y + \ve Y_0) \to \ve S(Y_0)$ for any $Y_0 \in V^*$, so $\mr{im}\,S \subset \mr{Lin}(\tC'_\infty)$, the subspace of all translational symmetries of $\tC'_\infty$.
    However, each $\tC_j' \subset \tC_j^\vee$ by acuteness, so $L \subset \mr{Lin}(\tC^\vee_\infty)$.
    If the translational space of $\tC^\vee_\infty$ is non-empty, then $\tC_\infty$ has no interior. 

    Suppose now that the pair is strictly obtuse.
    If $\ker S \cap \ol{\tC^\vee} \neq \{0\}$, choose some $Y \in \ker S\cap \ol{\tC^\vee}$, so by strict obtuseness, there is some $\ve$ such that $Y + \ve \bB^* \subset \tC'$. 
    Therefore, $S_j(Y \pm\ve Y_0) \to \pm S(\ve Y_0)$, so $\mr{im}\, S \subset \mr{Lin}(\tC'_\infty)$, so $\tC^\vee_\infty$ has a non-trivial translational space, implying $\tC_\infty$ is not open. 
    Finally, if $\ker S \cap \tC^\vee = \{0\}$, then $\tC^\vee_\infty \subset \mr{im}\,S$.
    For some $X \in (\tC')^\vee \setminus \mr{im}\, S^T$, let $X_j := \frac{S_j^{-T}X}{|S_j^{-T}X|}$, where the denominator goes to infinity by the assumption $X \not\in\mr{im}\, S^T$. 
    Let $X_\infty = \lim X_j$ up to a subsequence and $S^T X_\infty = 0$ with $X_\infty \in (\tC'_\infty)^\vee$.
    Therefore, $X_\infty$ annihilates all of $\mr{im}\, S$, and in particular all of $\tC^\vee_\infty$, which implies $\tC^\vee_\infty \subset \p \tC'_\infty$, so the pair is not positively aligned.
\end{proof}

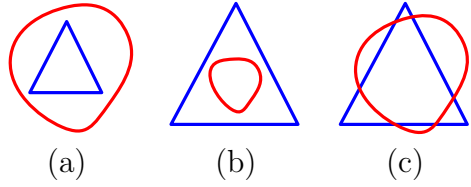
\begin{figure}
    \centering
    \setlength{\tabcolsep}{0pt}
    \begin{tabular}{ccc}
\begin{tikzpicture}[scale=0.4,line cap=round,line join=round]
  \path[use as bounding box] (-2.8,-2.35) rectangle (2.8,2.35);

  % Red rounded non-regular pentagon
  \begin{scope}[scale=1.58]
    \draw[red, line width=1.2pt]
      plot[smooth cycle, tension=0.9] coordinates {
        (0.05,1.20)
        (1.28,0.60)
        (0.90,-0.82)
        (-0.10,-1.32)
        (-1.18,0.08)
      };
  \end{scope}

  % Blue triangle, strictly inside red
  \draw[blue, line width=1.2pt]
    (0,1.40) -- (-1.22,-0.95) -- (1.16,-0.95) -- cycle;
\end{tikzpicture} & 
    
       % Figure 1: red strictly inside blue
\begin{tikzpicture}[scale=0.4,line cap=round,line join=round]
  \path[use as bounding box] (-2.8,-2.35) rectangle (2.8,2.35);

  % Blue triangle
  \draw[blue, line width=1.2pt]
    (0,2) -- (-2.15,-2) -- (2.05,-2) -- cycle;

  % Red rounded non-regular pentagon, strictly inside
  \begin{scope}[xshift=-0.08cm,yshift=-0.5cm,scale=0.66]
    \draw[red, line width=1.2pt]
      plot[smooth cycle, tension=0.9] coordinates {
        (0.05,1.00)
        (1.28,0.40)
        (0.90,-1.02)
        (-0.10,-1.52)
        (-1.18,0.28)
      };
  \end{scope}
\end{tikzpicture}  &

\begin{tikzpicture}[scale = 0.4,line cap=round,line join=round]
  \path[use as bounding box] (-2.8,-2.35) rectangle (2.8,2.35);

  % Blue triangle
  \draw[blue, line width=1.2pt]
    (0,2) -- (-2.15,-2) -- (2.05,-2) -- cycle;

  % Red rounded non-regular pentagon in mixed position
  \begin{scope}[xshift=0.08cm,yshift=-0.23cm,scale=1.46,rotate=2]
    \draw[red, line width=1.2pt]
      plot[smooth cycle, tension=0.9] coordinates {
        (0.05,1.20)
        (1.28,0.60)
        (0.90,-0.82)
        (-0.10,-1.32)
        (-1.18,0.08)
      };
  \end{scope}
\end{tikzpicture}\\
       (a) & (b) & (c)
    \end{tabular}
    \caption{We illustrate the boundaries of $Q$ in red and $\Sigma$ in blue to show how to identify the various OT-separable cases and the linking property.
    Figures (a) and (b) are strictly acute and obtuse, respectively.
    Case (c) is OT-separable and linked with $H_*({\Sigma},{\Sigma}\setminus Q) = \bZ^2$, which is entirely in degree $1$.}
    \label{fig:acute/obtuse/maxmin}
\end{figure}

\begin{remark}
    If both $Q$ and $\Sigma$ are polytopes, it seems likely there are only finitely open, positively aligned pairs of limit cones that one would need to examine to verify OT-separability.
    
\end{remark}

\begin{example}
    Let $\tC$ be a strictly convex cone and let $h : \bS^{n-2} \to \bR$ be any continuous function.
    Defining $\Sigma := \mr{ConvexHull}\{\mr{graph}_Q(\ve h)\}$, we can examine the pair $(\tC, \tC(\Sigma))$. 
    For $\ve$ sufficiently small, $\tC(\Sigma)$ is convex for any $h$.
    Furthermore, if we suppose that $h$ is differentiable and $\nabla h \neq 0$ on $h^{-1}(0)$, the pair is OT-separable by Corollary~\ref{cor:Crossing=OTsep}.
    The linking property is therefore given by $\widetilde{H}_*(h^{-1}(\bR_{\geq 0}))$, so in particular, we can prescribe any arbitrarily complicated homology in all degrees $0$ through $n-2$ simultaneously by taking a partition of unity associated with any $\bS^{n-2} = A \sqcup B$ for $A$ open.
\end{example}

\section{Equivariant homogeneous optimal transport}\label{sec:polystable}
We now consider the equivariant case in which the cones $(\tC, \tC')$ have a non-compact shared symmetry group.
Recall, we have $\mr{Aut}(\tC, \tC') := \{g \in \sln :(g\cdot \tC, g^{-T}\cdot \tC') = (\tC, \tC')\}$. 
We will assume that $G \subset \mr{Aut}(\tC, \tC')$ is a reductive subgroup; it is an interesting question if there exist \hyperref[eqn:HOT]{H.O.T.} maps between $\tC$ and $\tC'$ whose automorphism group is not reductive as this presents an obstruction in most moduli theories.
From Mostow's self-adjointness theorem~\cite{Mostow}, we may, up to a change of coordinates, assume that $G$ is self-adjoint. 
Let $K_G := G \cap \son$ be a maximal compact subgroup and for $\mathfrak{g} = \mr{Lie}(G)$, we can decompose
\[
K_G := G \cap \son,\qquad \mathfrak{k} := \mathfrak{g} \cap \mathfrak{so}(n),\qquad     \mathfrak{p} := \mathfrak{g} \cap \mr{Sym}_0(n).
\]
Every $g \in G$ is expressible as $k e^{A}$ for $k \in K_G$ and $A \in \mathfrak{p}$.

We now separate two different types of degenerations of functions, those arising from the symmetric degenerations $G \ni g \to \infty$ and the non-trivial degenerations $(\tC, \tC') \rightsquigarrow (\tC_\infty, \tC'_\infty)$. The symmetry group $G$ acts naturally on $\cH_{(\tC,\tC')}$, and this action satisfies $K_{g\cdot \varphi} = g K_\varphi$ and $K_{g \cdot \varphi}' = g^{-T} K'_{\varphi}$.  In particular, both $I$ and $J$ are $G$-invariant. 

\begin{defn}
We say a point $Y \in V^*$ is \textbf{balanced} if $\la A Y, Y\rg = 0$ for all $A \in \mathfrak{g}$.
If $Y$ is balanced, then $\lambda Y$ is also balanced for any $\lambda > 0$, so we may rescale to $y := (Y',1) \in \{Y_n = 1\}$ on the affine link. We define $\Sigma_{\mr{Bal}} := \{q \in \Sigma : q \text{ is balanced}\}$.
It is sufficient to have $\la AY, Y\rg = 0$ for $A \in \mathfrak{p}$ since for $A \in \mathfrak{k}$, $\la AY, Y\rg = \la Y, -AY\rg$, so this vanishes.
\end{defn}

By rescaling, we have an action $G$ on $\{Y_n = 1\}$ which we denote $g \star y := \frac{g\cdot (y,1)}{\la e_n, g\cdot (y,1)\rg}$.

\begin{lem}\label{lem:balancing}
    For any $y \in \Sigma$, there exists $A_y \in \mathfrak{p}$ such that $e^{A_y}\star y \in \Sigma_{\mr{Bal}}$, which is unique up to $K_G$. 
\end{lem}
\begin{proof}
Consider a basis $z_i$ such that $A z_i = w_i z_i$ for $\sum w_i = 0$ the weights of $A$.
So for a fixed $y$, we may decompose $gy = \sum \lambda_i z_i$.
Let $p_i := \frac{ |\lambda_i|^2e^{2w_jt}}{\sum |\lambda_j|^2e^{2w_jt}}$ and for $g = K_G e^{tA}$.
Consider the function
\[
B_{y}(g) := \frac{1}{2}\log \left(\sum |\lambda_i|^2e^{2 w_it}\right)\qquad \implies \qquad B''_{y}(g) = 2 \left( \sum p_i w_i^2 - \left(\sum p_i w_i\right)^2\right) \geq 0,
\]
whose second derivative vanishes if and only if all eigenvalues $w_i$ of $A$ are equal, which since $A$ is traceless implies $A = 0$.
Furthermore, $B_y(g) = B_y(kg)$ for any $k \in K_G$ since $K_G \subset \son$. 
Therefore, $B_y$ is strictly convex and one can see $B_y \to \infty$ as $g \to \infty$.
Thus, $y$ is balanced if it attains the unique minimum value of $B_y$, which is achieved for some $e^{A_y} \in \mathfrak{p}$, and its balanced representative $e^{A_y}\star y$ is unique up to $K_G$. 
At a minimum, differentiating in every direction $A \in \mathfrak{p}$ shows
\[
0 = \frac{d}{dt}\bigg\vert_{t=0}\frac{1}{2}\log |e^{tA}Y|^2 = \frac{\la AY, Y\rg}{|Y|^2}
\]
showing the balanced equation is equivalent to this minimization.
The family $e^{s A_y}\star y$ for $s \in [0,1]$ shows that balancing is a strong deformation retract from $\{Y_n = 1\}$ to the balanced locus.
\end{proof}
We define $b : \Sigma \times [0,1] \to \Sigma$ to be the balancing homotopy such that $b_0 = \mr{Id}$ and $b_1(\Sigma) = \Sigma_{\mr{Bal}}$.
Since $G$ preserves both $\Sigma$ and $Q$, this map provides a homotopy of pairs $(\Sigma, \Sigma \setminus Q) \to (\Sigma_{\mr{Bal}}, \Sigma_{\mr{Bal}} \setminus Q)$. 
From this property, we can define the balanced class of functions.
We define $O_{\mr{Bal}} := \Sigma_{\mr{Bal}} \setminus \ol{Q}$ and $A_{\mr{Bal}, \ve} := \{y \in \Sigma_{\mr{Bal}} \cap Q : d(y,\p Q) > \ve\}$.

\begin{defn}
    The pair $(\tC, \tC')$ is \textbf{$G$-linked} if
    \[
    \mr{im}\bigl[H_*(\Sigma_{\mr{Bal}}, O_{\mr{Bal}})\xrightarrow{\iota_*}H_*(\Sigma_{\mr{Bal}}, \Sigma_{\mr{Bal}} \setminus A_{\mr{Bal}, \ve})\bigr] \neq 0
    \]
    for some $\ve > 0$ sufficiently small.
\end{defn}

The following properties allow us to use the preimages of the Steiner points to deduce the corresponding linking property for a balanced function space and complete the existence argument.
\begin{lem}\label{lem:balancingSteiner}
    For any $v\in \cH$, there exists $g \in G$ such that $\mathfrak{q}(g \cdot v) \in \Sigma_{\mr{Bal}}$. 
\end{lem}
\begin{proof}
    Since the Steiner point is not equivariant, we use a degree argument. 

    Lemma~\ref{lem: SteinerPointLowBound} implies that for any $g \in G$, we have $|S_U(g^{-T}K')| \geq c_v \|g^{-T}\|_{\mr{HS}}$ where $c_v$ depends on $K'$. 
    Let $Z_a := S_U(e^AK')$ for $A \in \mathfrak{p}$. 
    For $\|A\|_{\mr{HS}} = 1$, we have $|Z_{tA}| \geq c_v e^{t\lambda_+(A)}$. 
    Therefore, we have
    \[
    \frac{\la A Z_{tA}, Z_{tA}\rg}{|Z_{tA}|^2} \geq \lambda_+(A) - \frac{c_v}{t} \to \lambda_+(A)
    \]
    which the constant independent of $A$. 
    Therefore, for $\la A Z_{tA}, Z_{tA}\rg > 0$ for all every unit $A$ and some $t$ sufficiently large.
    Consider map $\Phi(A) := \Pi_{\mathfrak{p}}\left(\frac{Z_{A}\otimes Z_{A}}{|Z_{A}|^2}\right)$.
    For $\{A : \|A\|_{\mr{HS}} = R\}$, this provides a degree $1$ map to the sphere, so that means some $\tilde{A}$ with $\|\tilde{A}\|_{\mr{HS}} < R$, we have $\Phi(\tilde{A}) = 0$, which means that $\la A Z_{\tilde{A}}, Z_{\tilde{A}}\rg = 0$ for all $A \in \mathfrak{p}$, and thus $\mathfrak{q}(e^{\tilde{A}}K')$ is balanced.
    
\end{proof}

\begin{lem}\label{lem:GProper}
    The $G$-action on $\Sigma$ is proper.
\end{lem}
\begin{proof}
    We must show that if $y_j \to y$ in $\Sigma$ and $g_j \star y_j \to z$ in $\Sigma$, then $g_j \to g$, up to a subsequence.
    Let $A_j := \frac{g_j^{-T}}{\la e_n, (y_j, 1)\rg}$, which is a linear map, inducing a homeomorphism from $\tC'$ to itself.
    By construction, it satisfies $A_j (y_j,1) \to (z, 1)$.
    Since $\tC'$ is open, there is some $\ve$ such that $(y_j,1) + \ve \bB^* \subset \tC'$ for all $j$ sufficiently large.
    Applying $A_j$ to this set shows $(g_j \star y_j, 1) + \ve A_j \bB^* \subset \tC'$.
    Since $(g_j \star y_j, 1) \to (z,1)$ is bounded, $A_j$ must be bounded in operator norm.
    Suppose otherwise, along some sequence $\pm W_j \in A_j \bB^*$ such that $\|W_j\| \to \infty$, we have $\frac{(g_j\star y_j ,1) \pm W_j}{|W_j|} \in \tC'$.
    However, that implies that $\pm W_j \in \tC'$, which contradicts that $\tC'$ is a pointed cone. 
    Therefore $\|A_j\| \leq C$.
    The same argument applied to $A_j^{-1}$ shows that $\|A_j^{-1}\| \leq C'$.
    Since $g_j \in \sln$, we we have $\det (A_j) = \la e_n, g_j^{-T}(y_j,1)\rg^{-n}$, which is therefore uniformly bounded. 
    Therefore, we have $\|g_j^{-1}\| + \|g_j\| \leq C$, so $g_j \to g$. 
\end{proof}

\begin{cor}\label{lem:H/Gcomplete}
    The $G$-action on $\bP(\cH)$ is proper.
\end{cor}
\begin{proof}
    Let $b(K')$ be the barycenter of $K'$ and $\beta_\varphi := \frac{b(K')}{\la e_n, b\rg}$ be its rescaling to $\Sigma$. 
    Consider $\varphi_j \to \varphi$ and $\varphi_j \circ g_j \to \psi$ in $(\bP(\cH), d_\bP)$, so we must show that $g_j \to g$ converges, up to a subsequence. 
    We know that $\beta_{\varphi_j} \to \beta_{\varphi}$ since $K'_{\varphi_j} \to K'_\varphi$ in the Hausdorff distance.
    Since $g \star \beta_{\varphi_j} = \beta_{\varphi_j \circ g_j} \to \beta_\psi \in \Sigma$.
    If $g_j$ did not converge, then, by Lemma~\ref{lem:GProper}, $\beta_{\varphi_j} \to \p \Sigma$.
    However, since $\psi \in \bP(\cH)$, $\beta_\psi \in \Sigma$, a contradiction. 
\end{proof}

We can now define $\cH_{\mr{Bal}} := \mathfrak{q}^{-1}(\Sigma_{\mr{Bal}})$, $\cO_{\mr{Bal}} := \mathfrak{q}^{-1}(O_{\mr{Bal}})$, and $\cA_{\mr{Bal}, \ve} := \mathfrak{q}^{-1}(A_{\mr{Bal}, \ve})$.
Lemma~\ref{lem:balancingSteiner} shows that we may restrict to the balanced locus.
Because the $G$ acts properly on $\cH$, so the balanced functions are complete. 
If $\beta$ is a non-trivial class exhibiting the $G$-linking property, then $\alpha = \gamma_* \beta \in H_*(\cH_{\mr{Bal}}, \cO_{\mr{Bal}})$ has non-zero image in $H_*(\cH_{\mr{Bal}},\cH_{\mr{Bal}}\setminus \cA_{\mr{Bal},\ve})$ where $\gamma$ is the section from Lemma~\ref{lem:affineSection}. 
Lemma~\ref{lem:EnergyInfBound} shows that we have $\cM(\cA_{\mr{Bal}, \ve}) \leq E_{\ve} < \infty$. 
We can now define the analogous notions of admissible families and widths in the equivariant section:
\[
\mathscr{S}_{G,k} := \{\Gamma \in C_k(\cH_{\mr{Bal}}) : \p \Gamma \in C_{k-1}(\cO_{\mr{Bal}}), \inf_{v \in |\p \Gamma|} \cM(v) \geq E_{ \ve} + 10\}
\]
and for some $\alpha \in H_*(\cH_{\mr{Bal}}, \cO_{\mr{Bal}})$ whose image in $H_*(\cH_{\mr{Bal}}, \cH_{\mr{Bal}} \setminus \cA_{\mr{Bal}, \ve})$ is non-trivial, define
\[
W_\alpha^G := \sup_{\substack{\Gamma  \in \mathscr{S}_{G,k} \\ [\Gamma] = \alpha}}\inf_{v \in |\Gamma|}\cM(v).
\]

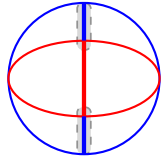
\begin{figure}
    \centering
   \begin{tikzpicture}[scale=1]

    \def\w{0.09}
    \begin{scope}
        \clip (0,0) circle (1);
        \filldraw[fill=gray!70, fill opacity=0.42, draw=gray!85,
                  dashed, line width=0.65pt, rounded corners=0.07cm]
            (-\w,0.38) rectangle (\w,0.995);
        \filldraw[fill=gray!70, fill opacity=0.42, draw=gray!85,
                  dashed, line width=0.65pt, rounded corners=0.07cm]
            (-\w,-0.995) rectangle (\w,-0.38);
    \end{scope}

    % unit circle
    \draw[blue, line width=0.8pt] (0,0) circle (1);

    % ellipse: width 2, height 1
    \draw[red, line width=0.8pt] (0,0) ellipse [x radius=1, y radius=0.5];

    % thick vertical line: outside ellipse blue, inside ellipse red
    \draw[blue, line width=1.6pt] (0,1) -- (0,0.5);
    \draw[red,  line width=1.6pt] (0,0.5) -- (0,-0.5);
    \draw[blue, line width=1.6pt] (0,-0.5) -- (0,-1);

\end{tikzpicture}
    \caption{We illustrate an equivariant linking $G$-linked pair that is not linked; $\Sigma$ is in blue and $Q$ is in red.
    The two cones share a boost symmetry with eigenvectors given by $(\pm 1, 0,1)$ with weights $\pm 1$ and $(0,1,0)$ with weight 0.
    The balanced locus $\Sigma_{\mr{Bal}}$ consists of the vertical line, which has an obtuse region in blue and an acute region in red.
    An example neighborhood of the obtuse balanced region $O_{\mr{Bal}}$ is shown by the shaded region.
    The region $A_{\mr{Bal}, \ve}$ is the portion of the vertical red line outside the shaded region, which is compactly contained in $Q \cap \Sigma$.
    The homology $H_*({\Sigma_{\mr{Bal}}} ,O_{\mr{Bal}}  ) = \bZ$, which is in degree $1$.}
    \label{fig:equivariantLinking}
\end{figure}

\begin{theorem}\label{thm:polystableWidthFinite}
    Let $(\tC, \tC')$ be positively aligned pointed cones and $G \subset \mr{Aut}(\tC, \tC')$ is reductive.
    Suppose that $(\tC, \tC')$ are $G$-linked and OT-separable. 
    Then, there exists $\varphi \in \cH_{\mr{Bal}}$ which solves~\eqref{eqn:HOT}.
\end{theorem}
\begin{proof}
    By construction, we have some class $\alpha \in H_k(\cH_{\mr{Bal}}, \cO_{\mr{Bal}})$ whose image in $H_k(\cH_{\mr{Bal}},\cH_{\mr{Bal}} \setminus \cA_{\mr{Bal}, \ve})$ is non-trivial.
    Since $\cA_{\mr{Bal},\ve} \subset \cA_{\ve}$, Proposition~\ref{prop:widthFinite} shows that $W_\alpha^G$ is finite.
    Similarly, we may apply Lemma~\ref{lem:pulltight} to find a sequence $v_j$ such that $\cM(v_j) \to W_\alpha^G$ and $\mathfrak{a}(v_j) \to 0$.
    Proposition~\ref{prop:PS=SLNdegToPolystable}, combined with the OT-separability hypothesis, therefore shows that we have some $\varphi \in \cH_{\mr{Bal}}$ solving~\eqref{eqn:HOT}.
\end{proof}

\section{Perturbations and non-uniqueness}\label{sec:notConvex}

This section investigates two aspects of the preceding work from a perturbative point of view. 
Is the max-min framework used in Section~\ref{sec:maxmin} and Section~\ref{sec:polystable} really necessary?
More specifically, since the optimal transport problem between compact convex domains is convex, does the necessity of the max-min argument arise from a poor choice of variational framework? Secondly, is the topological linking condition sharp?  To address these questions we exhibit a pair of cones with two distinct and isolated~\hyperref[eqn:HOT]{H.O.T.} maps.
This immediately implies that the homogeneous optimal transport problem does not have a convex formulation.
Furthermore, we show that these maps arise on cones which do not satisfy the linking condition Definition~\ref{def:Linked}, suggesting that this topological condition can be relaxed in some settings.

Let $\varphi$ be a~\hyperref[eqn:HOT]{H.O.T.} map on a pair $(\tC, \tC')$. We can study the existence of ~\hyperref[eqn:HOT]{H.O.T.} on nearby cones $(\widetilde{\tC}, \widetilde{\tC}')$ as a linearized version of the second boundary value problem.
Consider the metric $g=D^2\varphi$ on $\tC$.  In radial variables $g := dr^2 + r^2 h$ where $h$ is the induced metric on $S := \{\varphi = \frac{1}{2}\}$.
The linearized second boundary value problem for the Monge-Amp\`ere equation is
\begin{equation}\label{eqn:linearization}
    \Delta_h \psi + 2n\psi = a,\qquad \kappa \p_\nu \psi = b,  \qquad \kappa :=\sqrt{(D \rho')^T D^2\varphi (D\rho')}
\end{equation}
for $y := \nabla \varphi(X)$, $\rho'$ a homogeneous convex function defining $\tC' = \{\rho' < 0\}$, and $a,b$ functions determining the desired perturbation data.
Therefore, we say that $\psi$ is a \textbf{Jacobi field} if $\psi$ solves equation~\eqref{eqn:linearization} with $(a,b) = (0,0)$. 
These are the directions in which~\eqref{eqn:linearization} is not locally invertible. 
By the Fredholm alternative, a perturbative direction $(a,b)$ gives rise to a first-order deformation exactly when it is orthogonal to every Jacobi field, meaning
\[
\int_S a \psi\, dV_h - \int_{\p S}\frac{b}{\kappa}\psi\,dA_h = 0.
\]
We say $\varphi$ is \textbf{non-degenerate} if there are no non-trivial Jacobi fields, and then the implicit function theorem implies that all nearby perturbations admit~\hyperref[eqn:HOT]{H.O.T.} maps. 
In particular, if $\tC' = \tC^\vee$ and $\varphi$ is non-degenerate, there exist perturbations that unlink the cones, but will still have~\hyperref[eqn:HOT]{H.O.T.} maps.

In the case where $G := \mr{Aut}(\tC, \tC')$ is non-compact, a map $\varphi$ is non-degenerate modulo $G$ if the only Jacobi fields arise from the $G$-action. In this case, one can apply the perturbation theory using the Lyapunov–Schmidt reduction.

\begin{prop}\label{prop: notUnique}
    There exist pairs $(\widetilde{\tC}, \widetilde{\tC}')$ which are OT-separable, unlinked, and exhibit two distinct, isolated solutions of ~\eqref{eqn:HOT}.
\end{prop}
\begin{proof}
    Let $(\cL, \cL')$ be dual Lorentz cones in $\bR^4$: $\cL := \{(tx,t) : x\in B_1^3, t > 0\}$ and $\cL^\vee = \cL'$. 
    We define the perturbed target cone in direction $H : \bS^2 \to \bR$ as
    \begin{equation}\label{eqn:PerturbedTarget}
    \cL'_{H} := \{\rho \xi : \rho > 0, \xi = (\sin(\theta)\omega, \cos(\theta)) \in \bS^3, \omega \in \bS^2, 0 \leq \theta < \tfrac{\pi}{4} + H(\omega)\}
    \end{equation} 
    for which we will find solutions for $\cL'_{\ve H}$ for $0 < \ve \ll 1$ sufficiently small.
    For $a \in \bR^3$ and $s_a := \sqrt{1 + |a|^2}$, we consider the function
    \[
    \Phi_H(a) := \int_{\bS^2}  \frac{H(\omega)}{(s_a - a \cdot \omega)^4}\,d\sigma(\omega).
    \]
    The following lemma produces~\hyperref[eqn:HOT]{H.O.T.} maps on the perturbed cone pair $(\cL, \cL_{\ve H}')$ based on the critical points of $\Phi_H$. 
    Afterwards, we show how to produce an example of an $H$ such that $\Phi_H$ has multiple non-degenerate critical points, which then completes the proof.
    \begin{lem}\label{lem:LorentzPerturb}
         For every non-degenerate critical point $a$ of $\Phi_H$, there exists some $\ve_a > 0$ such that for all $0 < \ve < \ve_a$, there exists some $\varphi_{\ve, a}$ solving~\eqref{eqn:HOT} on the perturbed pair $(\cL, \cL'_{\ve H})$.
         Furthermore, $\varphi_{\ve, a}$ is isolated.
    \end{lem}
    \begin{proof}
         Let $u$ be any $2$-homogeneous function on $\cL$.  
         We may express $u = \rho^2 f(\xi)$ for $\xi := \frac{X}{|X|} \in \bS^3$ and $\rho(X) = |X|$.
         For the Lorentz cone, we can parametrize all $\xi \in \cL \cap \bS^{3}$ as
         \begin{equation}\label{eqn:ScoordinatesSphere}
         S := \{(\sin(\theta)\omega, \cos(\theta)) : \omega \in \bS^2, \theta \in [0,\tfrac{\pi}4)\}
         \end{equation}
         whose boundary is $\frac{1}{\sqrt{2}}(\omega, 1)$ for all $\omega \in \bS^2$. 
         We compute $\nabla u = \rho (2 f\xi + \nabla_{\bS^3} f)$ for $\nabla_{\bS^3}$ denoting the spherical gradient, and the Hessian
         \[
         D^2 u = \begin{pmatrix}
             2f & (\nabla_{\bS^3} f)^T \\
             \nabla_{\bS^3} f & \nabla_{\bS^3}^2 f + 2f \,\mr{Id}
         \end{pmatrix}.
         \]
         The identity map $u_0 : \cL \to \cL'$ is generated by $f_0 := \frac{1}{2}$, so $D^2 u_0 = \mr{Id}$. 
         We consider a perturbation $f_s := \frac{1}{2} + s\psi$ and we compute 
        \[
        \frac{d}{ds}\bigg\vert_{s=0} \log \det D^2 u_{f_s} = 2\psi  + \mr{tr}(\nabla_{\bS^3}^2 \psi +2\psi\,\mr{Id}) =  \Delta_S \psi + 8\psi.
        \]
        
        We now examine the linearization of the boundary condition.
        For $H \in C^{2,\alpha}(\bS^2)$, we can define the perturbed boundary given by $\theta = \frac{\pi}{4} + H(\omega)$ using the coordinates for $S$ in equation~\eqref{eqn:ScoordinatesSphere}.
        Let $\rho_{H}$ be a homogeneous function defining this perturbation, so $\cL'_H := \tC(S'_H)$ as defined in equation~\eqref{eqn:PerturbedTarget}.
        We compute the gradient and its normalization
        \[
        \nabla u(\rho \xi) = \rho (2f(\xi)\xi + \nabla_{\bS^3} f(\xi)),  \qquad T_f(\xi) := \frac{2f(\xi)\xi +\nabla_{\bS^3} f(\xi)}{|2f(\xi)\xi +\nabla_{\bS^3} f(\xi)|} 
        \]
        and the total gradient image is generated by the image $T_f(S)$. 
        By convexity, we can examine this by considering the gradient image of $\p S$ where $\theta = \frac{\pi}{4}$.
        The gradient equality $\nabla u(\cL) = \cL'$ is given by $T_f(\p S) = \p S'_H$.
        We linearize this at $H_s := sH$ and $f_s := \frac{1}{2} + s \psi$, and we have
        \[
        \frac{d}{ds}\bigg\vert_{s=0}T_{f_s} = \nabla_{\bS^3} \psi ,
        \]
        so the linearized second boundary conditions is $\p_\nu \psi = H$ on $\p S$. 
        Let 
        \[
        X := C^{2,\alpha}(\ol{S}), \qquad P :=C^{2,\alpha}(\bS^2),\qquad \text{and}\qquad Y:= C^{0,\alpha}(\ol{S}) \times C^{1,\alpha}(\bS^2).
        \]
        We consider the operator $\cF : X  \times P\to Y$defined by
        \[
        \cF(f, H) = (\log \det D^2 (\rho^2 f), \theta(T_f)- \tfrac{\pi}{4} - H(\omega(T_f))),
        \]
        where $\theta, \omega$ are the angular coordinate functions from equation~\eqref{eqn:ScoordinatesSphere}.
        We know that $\cF(\tfrac{1}{2},0) = (0,0)$ is a solution.
        For general inhomogeneous data, the above computations show perturbing $u_0$ in the direction $(g,b) \in C^{0,\alpha}(S) \times C^{1,\alpha}(\bS^2)$ is given to first-order by 
        \begin{equation}\label{eqn:LinearizedHotLorentz}
        \begin{cases}
            (\Delta_{\bS^3} + 8)\psi &= g \text{ in }S,\\
            \p_\nu \psi &=b \text{ on }\p S
        \end{cases}
        \end{equation}
        for $g$ and $b$ defining the interior and boundary data perturbations.
        We have
        \[
        D_f \cF\left(\frac{1}{2},0\right)[\psi] = ((\Delta_{\bS^3} +8)\psi , \p_\nu \psi) \qquad \text{and}\qquad D_H \cF \left(\frac{1}{2},0\right)[H ] = (0,-H),
        \]
        showing that $D\cF\left(\tfrac{1}{2},0\right)[\psi, H] = ((\Delta_S + 8)\psi, \p_\nu \psi - H)$.
        
        Let $L$ be the operator corresponding to the homogeneous equation~\eqref{eqn:LinearizedHotLorentz}.
        The harmonics $\psi(\theta, \omega) = v(\theta)Y_m(\omega)$ where $-\Delta_{\bS^2}Y_m =m(m+1)Y_m$, so $v$ satisfies
        \[
        v'' + 2\cot(\theta)v' + (8 - \csc^2(\theta)m(m+1))v = 0\qquad \text{and}\qquad v'(\tfrac{\pi}{4}) = 0.
        \]
        For $m = 0$, $v = \cos^2(\theta) - \frac{1}{4}$, which fails the boundary condition.
        For $m = 1$, $v = \sin(\theta)\cos(\theta)$, yielding solutions
        \[
        K := \{\psi_c(\theta, \omega) := \sin(\theta)\cos(\theta)(c \cdot \omega),  c \in \bR^3\} .
        \]
        For $m \geq 2$, we consider the Rayleigh quotient formulation for $\psi$, which, with the Neumann condition, implies $8\int_S \psi^2 \,dV = \int_S |\nabla \psi|^2\,dV$.
        We compute $|\nabla \psi|^2 = (v')^2Y_m^2 + \csc^2(\theta) |\nabla_{\bS^2}Y_m|^2$ and, by definition, $\int_{\bS^2}|\nabla Y_m|^2 = m(m+1)\int_{\bS^2} Y_m^2$.
        Therefore, we have
        \begin{equation}\label{eqn:RayleighExpanded}
        8\int_0^\frac{\pi}{4}v^2 \sin^2(\theta)\, d\theta = \int_0^\frac{\pi}{4}(v')^2\sin^2(\theta)\,d\theta + m(m+1)\int_0^\frac{\pi}{4}v^2\, d\theta. 
        \end{equation}
        On this range of $\theta$, $\sin^2(\theta) \leq \frac{1}{2}$, so equation~\eqref{eqn:RayleighExpanded} implies the inequality
        \[
        (8 - 2m(m+1))\int_0^\frac{\pi}{4}v^2\sin^2(\theta)\,d\theta \geq \int_0^\frac{\pi}{4} (v')^2\sin^2(\theta)\,d\theta.
        \]
        However, for $m \geq 2$, the left-hand side is negative, which is a contradiction.
        Therefore, the kernel of $L$ is precisely the $m = 1$ modes, which is three dimensional.
        We can identify the kernel $K$ by differentiating the Lorentz boosts; define
        \[
        P_a := \mr{Id} + \frac{aa^T}{1 + s_a},\qquad A_a=\begin{pmatrix} P_a&a\\ a^T&s_a\end{pmatrix}, \qquad s_a:= \sqrt{1+|a|^2},
        \]
        which satisfy $\det A_a = 1$, $A_a \cL = \cL'$ and $u_a(X) := \frac{1}{2} |A_aX|^2$, which is a~\hyperref[eqn:HOT]{H.O.T.} map from $\cL$ to $\cL'$. 
        By construction, we see that $u_a(A_a^{-1}X) = u_0(X)$. 
        Therefore, differentiating this map at $a = 0$ identifies the kernel of $L$, so we can consider $a$ as a coordinate on the kernel $K$.

        The perturbed target cone $\cL'_{\ve H}$ is defined by 
        \[
        \p S'_{\ve H} := \{(\sin(\tfrac{\pi}{4} + \ve H(\omega))\omega, \cos(\tfrac{\pi}{4} + \ve H(\omega))): \omega \in \bS^2\},
        \]
        and we want to solve $\cF(f, \ve H) = 0$.
        We can decompose 
        \[
        X = K \oplus X_1 \qquad \text{and}\qquad  Y = \mr{im}\, L \oplus Y_1.
        \]
        The Neumann realization of $\Delta_{\bS^3} + 8$ is self-adjoint.
        The Schauder theory implies $L$ is Fredholm and therefore has index $0$, meaning  $\dim K = \dim Y_1 = 3$ from the above computation.
        Since $L$ has index $0$, outside the kernel, $L$ provides an isomorphism $L : X_1 \to \mr{im}\, L$.

        For each $a$ and $0 < \ve \ll 1$, let $h_{a,\ve} : \bS^2 \to \bR$ be such that its graph over $\p S$ is $\p (A_a^{-1}\cL'_{\ve H}) \cap \bS^3$. 
        Since $A_a^{-1}$ preserves $\cL'$, we have $h_{a, 0} = 0$.
        The following function $\cG$ corresponds to a perturbed solution of~\eqref{eqn:HOT} for $(\cL, \cL'_{\ve H})$ after composing with $A_a$:
        \[
        \cG(a, w, \ve) := \cF(\tfrac{1}{2} + w, h_{a,\ve}), \qquad \cG(a,0,0) = 0, \quad D_w \cG(a,0,0)=L\vert_{X_1}.
        \]
        We now do a Lyapunov–Schmidt reduction (cf.~\cite[Theorem 5.1]{Guo-WuLyapunov}) to find a \hyperref[eqn:HOT]{H.O.T.} solution ${u} \in \cH_{(\cL, \cL'_{\ve H})}$. 
        We can decompose $\cG$ into its projection onto $\mr{im}\, L$ and its $Y_1$ component. 
        The implicit function theorem for Banach spaces solves for the projection of $\cG$ onto $\mr{im}\,L$, thereby providing some $w(a,\ve)$ where $\Pi_{\mr{im}\,L}(\cG(a, w(a,\ve),\ve)) = 0$.
        We must also find when the finite dimensional piece in $Y_1$ vanishes.
        We label
        \[
        r(a, \ve) := \Pi_{Y_1}(\cG(a, w(a,\ve),\ve)) \in Y_1 \cong \bR^3. 
        \]
        In summary, we seek $w(a,\ve)$ and $r(a, \ve)$ such that
        \[
        \cG(a, w, \ve) = 0 \qquad \iff \qquad w = w(a,\ve) \text{ and } r(a,\ve) = 0,
        \]
        where $w$ is given by the infinite dimensional implicit function theorem on $L : X_1 \to \mr{im}\, L$ and $r$ is given by solving the finite dimensional piece.
       
        Linearizing at $u_0$, the Fredholm alternative shows that equation~\eqref{eqn:LinearizedHotLorentz} is solvable for $(g,b)$ if and only if $\int_S g\psi_c \,dV = \int_{\p S} b \psi_c\, dA$ for all $\psi_c$.
        If $(g,b) = (0,H)$, this is equivalent to
        \begin{equation}\label{eqn:FredholmAlt}
        \int_{\p S }H(\omega)\psi_c \,dA = 0\qquad \iff \qquad \int_{\bS^2}H(\omega)\omega \, d\sigma = 0.
        \end{equation}
        We now compute this obstruction at an arbitrary point $a \in \bR^3$, corresponding to acting on the target by $A_a^{-1}$, finding solutions near $u_a$ instead of near $u_0$. 
        The action on $\p S$ is given by $A_a^{-1}(\omega, 1) = (s_a - a\cdot \omega) \left(\frac{P_a \omega - a}{s_a - a\cdot \omega },1\right)$.
        Therefore, the infinitesimal boundary displacement of $H(\omega)$ is given by $(s_a - a\cdot \omega)^{-2}H(\omega)$ and the Jacobian is $d\sigma \mapsto (s_a - a\cdot \omega)^{-2}d\sigma$.
        Therefore, equation~\eqref{eqn:FredholmAlt} at $a$ is given as
        \begin{equation}\label{eqn:FredholmAlta}
        \int_{\bS^2}H(\omega)(s_a - a\cdot \omega)^{-5}(P_a \omega - a)\,d\sigma = 0,
        \end{equation}
        which is $\frac{1}{4}P_a\nabla \Phi_H(a) = 0$.
        We differentiate $r(a,\ve)$ at $\ve = 0$ yielding
        \[
        \p_\ve r(a,0) =\Pi_{Y_1}(D_w \cG(a,0,0)[\p_\ve w(a,0)] + \p_\ve\cG (a,0,0) ).
        \]
        Since $D_w \cG(a,0,0) = L$, the first term is $L(\p_\ve w(a,0))$, which has no $Y_1$ component.
        Therefore, $\p_\ve r(a,0) = \Pi_{Y_1}\p_\ve \cG(a,0,0)$, which is precisely the Fredholm alternative in equation~\eqref{eqn:FredholmAlta}.
        Therefore, we have
        \[
        r(a,\ve) = c_0 \ve P_a \nabla \Phi_H(a) + O(\ve^2)
        \]
        for some $c_0 \neq 0$.
        Furthermore, we have that $D_a(P_a \nabla \Phi_H(a))\vert_{a = a_0} = P_{a_0}D^2 \Phi_H(a_0)$.
        If $a_0$ is a non-degenerate critical point, this is invertible. 
        Since $r(a,0) \equiv 0$ since each $u_a$ is a solution, we have $\frac{r(a,\ve)}{\ve}$ extends to $0$ in a $C^1$ manner, so the finite dimensional implicit function theorem yields $a_\ve = a_0 + O(\ve)$ with $r(a_\ve,\ve) = 0$, and $D_a r(a_\ve, \ve) = c_0 \ve P_a D^2 \Phi_H(a_0) + O(\ve^2)$.
        Since both $P_a$ and $D^2 \Phi_H(a_0)$ are invertible, $a_\ve$ is the only zero of $r(\cdot, \ve)$ in a neighborhood of $a_\ve$. 
    \end{proof}

    We now exhibit an $H$ such that $\Phi_H$ has $2$ non-degenerate critical points and $H^{-1}(0)$ is a Jordan curve in $\bS^2$.
    The fact that $H^{-1}(0)$ is a Jordan curve means that $\Sigma_{\ve H} := \cL'_{\ve H} \cap \{t = 1\}$ in relation to $Q := \cL' \cap \{t = 1\}$ will be unlinked, so the previous lemma proves the result. 
    The following harmonic polynomial appears in Szulkin~\cite{SzulkinExample} and we use some of its properties proved therein.
    Let $P(x,y,z) := x^3 -3xy^2 + z^3  - \frac{3}{2}(x^2 + y^2)z$ which is a harmonic cubic. 
    The function $P$ has a unique critical point at $0$. 
    Furthermore, $P^{-1}(0)$ is homeomorphic to the plane and $P^{-1}(0)\vert_{\bS^2}$ is a simple closed curve.
    Let $Y$ be the associated spherical harmonic of $P\vert_{\bS^2}$.
    Since $P$ is a homogeneous cubic, $P(q) = -P(-q)$.
    For $a := (1,0,0)$, we define $\ell(x,y,z) := \nabla P(a) \cdot (x,y,z)$.
    Therefore, $\nabla (P-\ell)(\pm a) = 0$ and $D^2 P(a) = \begin{pmatrix}
       6 & 0 & -3\\
       0 & -6 & 0 \\
       -3 & 0 & 0
   \end{pmatrix}$ is non-degenerate.
   Therefore, $\pm a$ are both non-degenerate critical points of $P - \ell$. 
  We can expand using homogeneity
  \[
  \Phi_Y(q) = C_1 P(q) + O(|q|^5) \qquad \text{and}\qquad \Phi_{\ell}(q) = C_2 \ell(q) + O(|q|^3).
  \]
  For $H_\delta := C_1^{-1}Y - \delta C_2^{-1}\ell$, by homogeneity of $P$ and $\ell$, we have the expansion
  \[
  \delta^{-\frac{3}{2}}\Phi_{H_\delta}(\delta^\frac{1}{2}w) = P(w) - \ell(w) + O(\delta), 
  \]
  so for small $\delta$, we know that $\Phi_{H_\delta}$ has two non-degenerate critical points at $q_{\pm,\delta} = \pm \sqrt{\delta} a + O(\delta^{\frac{3}{2}})$.
  Therefore, Lemma~\ref{lem:LorentzPerturb} provides the perturbed cone $\widetilde{\tC}'$ and two distinct~\hyperref[eqn:HOT]{H.O.T.} potentials on $(\cL, \widetilde{\tC}')$ by choosing $\widetilde{\tC}' = \tC(\Sigma_{\ve H})$ for $\ve \leq \min\{\ve_{q_{+,\delta}}, \ve_{q_{-,\delta}}\}$. 
    
    Finally, since $H$ does not vanish at its critical points, Corollary~\ref{cor:Crossing=OTsep} shows that $(\cL, \cL'_{\ve H})$ is OT-separable.
\end{proof}
\begin{remark}
One can generalize the above construction using higher order harmonic polynomials to create $2k$ distinct solutions on similar perturbations.
\end{remark}

\section{Obstructions}\label{sec:obstructions}
In this section, we analyze and prove obstructions to solving~\eqref{eqn:HOT}.
The first obstruction comes from the splitting theorem below saying $\tC$ and $\tC'$ must split off lines that pair perfectly to admit a~\hyperref[eqn:HOT]{H.O.T.} map.
From this, we characterize a necessary condition on the global boundary $C^{1,1}$-regularity given by isomorphic skeletal decompositions of $\Omega$ and $\Omega'$.
We further examine another broad class of obstructions arising from a moment map condition.
These geometric obstructions for the tangent cone pair to $(x,T(x))$ inhibit the $C^{1,1}$-boundary regularity of optimal transport at $(x, T(x))$.

\subsection{Splitting theorem}
The first result states that cone pairs that split off lines and have a~\hyperref[eqn:HOT]{H.O.T.} solution must be products, justifying the previous reduction to strict cones. 
We recall the Brascamp-Lieb inequality and rigidity statement (cf.~\cite{BLsource}).
\begin{theorem}[Brascamp-Lieb]\label{thm:BL}
    Let $\Omega \subset \bR^n$ be an open convex set.
    Assume $V \in C^2(\Omega)$ and $D^2 V >0$.
    For any $f \in W^{1,2}(\Omega, d\mu)$, 
    \[
        \int_\Omega
       \left(f-\int_\Omega f\,d\mu\right)^2d\mu \leq \int_\Omega \la (D^2 V)^{-1}\nabla f, \nabla f\rg\,d\mu,\qquad d\mu := \frac{e^{-V}\,dx }{\int e^{-V}\,dx}.
    \]
    Furthermore, equality holds precisely when $f -\int_\Omega f\,d\mu = \la a, \nabla V\rg$ for some $a \in \bR^n$ such that $\la a, \nu(x)\rg = 0$ for all $x \in \p \Omega$. 
\end{theorem}

We now prove the following splitting theorem, which is a stronger version of Theorem~\ref{thm: introSplitting}.
\begin{theorem}\label{thm:splitting}
    Let $\varphi$ be a~\hyperref[eqn:HOT]{H.O.T.} map from $(\tC, \tC')$.
    Suppose $L := \ol{\tC} \cap (-\ol{\tC})$ and $L' := \ol{\tC'} \cap (-\ol{\tC'})$ are the translation subspaces of $\tC$ and $\tC'$.
    Let $(\tC_0, \tC'_0) := (\tC \cap W, \tC' \cap W')$ for $W := (L')^\perp$ and $W' = L^\perp$.
    Then, $L$ and $L'$ pair perfectly and the map $\varphi$ splits as $\varphi(w + \ell) = \varphi_0(w) + \frac{1}{2}\la \ell, B\ell\rg$ for $w \in \tC_0$, $\ell \in L$, and $B :L \to L'$ a linear isomorphism.
    The function $\varphi_0 : \tC_0 \to \bR_+$ satisfies $\det D^2\vert_W\varphi \equiv c$ and $\nabla \varphi_0(\tC_0) = \tC'_0$, which is a~\hyperref[eqn:HOT]{H.O.T.} map on $(\tC_0, \tC'_0) \subset W \times W'$ up to scaling.
\end{theorem}
\begin{proof}

    Let $H(X) = D^2 \varphi$.    Let $K := \{\varphi < 1\}$ and $K' := \{\varphi^* < 1/4\}$, which are bounded convex subsets of $\tC$ and $\tC'$, respectively, and contain relatively open neighborhoods of the origin. We have $c|X|^2 \leq \varphi(X) \leq C|X|^2$ and $c'|Y|^2 \leq \varphi^*(Y) \leq C'|Y|^2$ for some constants $c,C,c',C'$.
    
    Fix any $\ell' \in L'$ non-zero, so $\tC' + t\ell' = \tC'$. 
    Integrating by parts shows the following identities:
    \begin{equation}\label{eqn:pa}
    \int_{\tC'}(\partial_{\ell'}  \varphi^*(Y)) e^{-\varphi^*(Y)}\,dY = 0
    \end{equation}
    and
    \begin{equation}\label{eqn:paa}
    \int_{\tC'}(\p_{\ell'} \varphi^*(Y))^2 e^{-\varphi^*(Y)}\,dY = \int_{\tC'}(\p^2_{\ell'\ell'} \varphi^*(Y)) e^{-\varphi^*(Y)}\,dY. 
    \end{equation}
    From Legendre duality, for $Y = \nabla \varphi(X)$, we have $\varphi^*(\nabla \varphi(X)) = \la X, \nabla \varphi(X)\rg- \varphi(X) = \varphi(X)$ by Euler criterion.
    Using $\det D^2 \varphi = 1$, this change of variables preserves Lebesgue measure and
    \[
    \p_{\ell'} \varphi^*(Y) = \la \nabla \varphi^*(Y), \ell'\rg  = \la X, \ell'\rg =: f_{\ell'}(X)
    \qquad\text{and}\qquad
    \p_{\ell'\ell'} \varphi^*(Y) = \la H(X)^{-1}\ell', \ell'\rg,
    \]
    which means that equations~\eqref{eqn:pa} and~\eqref{eqn:paa} imply
    \[
        \int_{\tC}f_{\ell'} e^{-\varphi(X)}\,dX = 0\qquad\text{and}\qquad
    \int_{\tC}f_{\ell'}^2e^{-\varphi(X)}\,dX = \int_{\tC}\la H(X)^{-1}\ell', \ell'\rg e^{-\varphi(X)}\,dX.
    \]
    Dividing the above identities by $\int_\tC e^{-\varphi(X)}\,dX$ and using $\nabla f_{\ell'} = \ell'$, the equality case of the Brascamp-Lieb inequality Theorem~\ref{thm:BL} is attained for the potential $\varphi$.
    The rigidity states that $\la X, \ell'\rg = \la a_{\ell'}, \nabla \varphi\rg$ for some vector $a_{\ell'}$. 
    Taking the gradient of this shows that 
    \begin{equation}\label{eqn:Ha=alpha}
    H(X) a_{\ell'} = \ell' 
    \end{equation}
    for all $X \in \tC$.

    We must now prove that $a_{\ell'} \in L$. 
    For $X \in \tC$, consider $X_t := \nabla \varphi^*(\nabla \varphi(X) + t\ell')$, and differentiating this using~\eqref{eqn:Ha=alpha} yields
    \begin{equation}\label{eqn:xtderive}
    \frac{d}{dt}X_t = D^2 \varphi^*(\nabla \varphi(X) + t\ell')\ell' = H(X)^{-1}\ell' = a_{\ell'} ,
    \end{equation}
    which shows that $X_t = X + ta_{\ell'}$. 
    
    Since $\nabla \varphi^*(\tC') = \tC$, $X_t \in \tC$ for all $X$, which implies that $a_{\ell'} \in L$.
    It remains to show that $a_{\ell'} \neq 0$, which follows from $\la a_{\ell'}, \ell'\rg = \la a_{\ell'}, H(X)a_{\ell'}\rg > 0$ since $H$ is positive-definite.
    Therefore, no non-zero $\ell' \in L'$ annihilates all of $L$, so $a_{\ell'}:L' \to L$ is injective.
    Repeating this argument by duality shows that $L$ and $L'$ are isomorphic.
    From this, we define $B : L \to L'$ to be the map $B\ell := H(X)\ell$. 
    
    From equation~\eqref{eqn:xtderive} and the fact that $\ell' \mapsto a_{\ell'}$ is an isomorphism from $L' \to L$, we realize that $B = H(X_0)\vert_L$ for any $X_0 \in \tC$, so this is independent of $X$ from the Brascamp-Lieb rigidity.
    By the symmetry and positivity of $H$, we have $\la BX, Y\rg = \la X, BY\rg$ and $\la X, BX\rg > 0$ for all $X \in L$ non-zero.
    The fact that the inner product is a perfect pairing on $L$ and $L'$ shows that $V = L \oplus W$ and $V^* = L' \oplus W'$. 
    Let $X = \ell + w$ for $\ell \in L$ and $w \in W$.
    We then compute that $D^2\varphi(X)[\ell, w] = \la H(X)\ell, w\rg = \la B \ell, w\rg = 0$. 
    Similarly, for $\ell_1,\ell_2 \in L$, $D^2\varphi(X)[\ell_1, \ell_2] = \la H(X)\ell_1, \ell_2\rg = \la B \ell_1, \ell_2\rg$.
    Integrating these over $L$ yields
    \[
    \varphi(w +\ell) = \varphi_0(w) +\frac{1}{2}\la \ell, B\ell\rg + \lambda(\ell)
    \]
    for some $\lambda \in L^*$, but 2-homogeneity forces $\lambda = 0$. 
    Finally, taking the gradient shows $\nabla\vert_W \varphi_0(\tC_0) = \tC'_0$ and since $D^2\varphi$ is block diagonal, $\det D^2 \varphi_0$ is constant.
    
\end{proof}

The splitting theorem provides an immediate obstruction to the global $C^{0,1}$ regularity of $\nabla u$ solving~\eqref{eqn:OT} based on the stratifications of the boundaries of $\Omega, \Omega'$.  Recall that the tangent cone to a convex set $\ol{\Omega}$ at $x_0$ is 
\[
\mr{Tan}_{x_0}(\Omega) := \{t(x-x_0) : x \in \mr{int}(\Omega), t> 0  \}.
\]
We can then filter $\ol{\Omega}$ by
\[
\mr{Sk}_*(\Omega) : =\mr{Sk}_0(\Omega) \subset \mr{Sk}_1(\Omega) \subset \cdots \subset \mr{Sk}_{n-1}(\Omega) \subset \mr{Sk}_n(\Omega) = \Omega 
\]
where 
\[
\mr{Sk}_j(\Omega) := \{x \in \ol{\Omega} : \mr{Tan}_x(\Omega) \text{ does not split off more than }j \text{ lines} \}.
\]
We define $\mr{Sk}_*(\Omega') $ similarly for $\Omega'$.
Note that $\mr{Sk}_{n-1}(\Omega) = \partial \Omega$.
See for example~\cite{SchneiderSkeletons} for details on skeletal filtrations for convex bodies. The splitting theorem, combined with Proposition~\ref{prop:NotHotNotC11} implies the following;
\begin{cor}\label{cor:skeletalIso}
    Let $\nabla u : \Omega \to \Omega'$ solves~\eqref{eqn:OT}.
    If $u \in C^{1,1}(\ol{\Omega})$, then $\mr{Sk}_*(\Omega) \cong \mr{Sk}_*(\Omega')$ as skeleta, and $T = \nabla u$ is an isomorphism of skeletal filtrations. 
\end{cor}

\begin{remark}\label{rmk:skeletonGlobalOb}
    An isomorphism of skeleta necessarily implies that the combinatorial data of the number of vertices, edges, faces, etc.~are the same, but is in fact stronger.
    For example, a cube and a tetrahedron with two truncated vertices both have $8$ vertices, $12$ edges, and $6$ faces, but are not isomorphic as skeleta.
\end{remark}

The condition that $\nabla u:\Omega \rightarrow \Omega'$ defines an isomorphism of skeletal filtrations provides a much stronger obstruction to global $C^{1,1}$-regularity than just an isomorphism of skeleta. Even pairs $(\Omega, \Omega')$ which have isomorphic skeleta can be obstructed from global $C^{1,1}$-regularity if each sub-skeleton cannot be compatibly matched through the gradient map of a convex function; see Figure~\ref{fig:ovidiu} for an example.

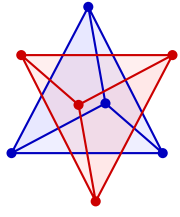
\begin{figure}
    \centering
    \begin{tikzpicture}[scale=2,
    x={(1.00cm,0.00cm)},
    y={(0.14cm,0.38cm)},
    z={(-0.04cm,1.05cm)}]

\pgfmathsetmacro{\zb}{-sqrt(2/3)/4}
\pgfmathsetmacro{\zt}{ 3*sqrt(2/3)/4}
\pgfmathsetmacro{\yy}{1/sqrt(3)}
\pgfmathsetmacro{\yh}{1/(2*sqrt(3))}

% Blue tetrahedron
\coordinate (B1) at ( 0,   0,   \zt);
\coordinate (B2) at (-0.5,-\yh, \zb);
\coordinate (B3) at ( 0.5,-\yh, \zb);
\coordinate (B4) at ( 0,   \yy, \zb);

% Red tetrahedron (dual / opposite)
\coordinate (R1) at ( 0,   0,  -\zt);
\coordinate (R2) at ( 0.5, \yh, -\zb);
\coordinate (R3) at (-0.5, \yh, -\zb);
\coordinate (R4) at ( 0,  -\yy, -\zb);

% ------------------------------------------------------
% Transparent face shading
% ------------------------------------------------------

% Blue faces
\fill[blue!45, opacity=0.11] (B1)--(B2)--(B3)--cycle;
\fill[blue!30, opacity=0.08] (B1)--(B3)--(B4)--cycle;
\fill[blue!60, opacity=0.07] (B1)--(B4)--(B2)--cycle;
\fill[blue!25, opacity=0.05] (B2)--(B3)--(B4)--cycle;

% Red faces
\fill[red!45, opacity=0.11] (R1)--(R2)--(R3)--cycle;
\fill[red!30, opacity=0.08] (R1)--(R3)--(R4)--cycle;
\fill[red!60, opacity=0.07] (R1)--(R4)--(R2)--cycle;
\fill[red!25, opacity=0.05] (R2)--(R3)--(R4)--cycle;

% ------------------------------------------------------
% Edges
% ------------------------------------------------------
\draw[blue!75!black, line width=0.8pt]
  (B2)--(B3)--(B4)--cycle
  (B1)--(B2) (B1)--(B3) (B1)--(B4);

\draw[red!80!black, line width=0.8pt]
  (R2)--(R3)--(R4)--cycle
  (R1)--(R2) (R1)--(R3) (R1)--(R4);

% Vertex bullets
\foreach \P in {B1,B2,B3,B4}
  \fill[blue!75!black] (\P) circle (0.95pt);
\foreach \P in {R1,R2,R3,R4}
  \fill[red!80!black] (\P) circle (0.95pt);

\end{tikzpicture}

    \caption{An optimal transport potential $u$ from a tetrahedron to its dual cannot be $C^{1,1}$ up to the boundary, despite having isomorphic skeleta.
    By the symmetries, it provides a local solution of $\bR \times \tC_0$ to $\tC_0' \times \bR$ for $\tC_0, \tC'_0$ both pointed $2$-dimensional cones.
    The local solution interchanges the line with the boundaries of $\tC_0$ and $\tC'_0$, so the blowup model is $(\tC_\infty, \tC_\infty') = (\bR^3_+, \bR^3_+)$, a half-space on both sides.}
    \label{fig:ovidiu}
\end{figure}

\subsection{Symplectic structure and moment map}\label{sec:ObMoment}

A second class of obstructions arises from an $\mathfrak{sl}(n)^*$-valued moment map defined on an infinite dimensional symplectic manifold.
While this construction is essentially formal, it provides genuine obstructions to the existence of \hyperref[eqn:HOT]{H.O.T.} maps, and serves as important motivation for the discussion in Section~\ref{sec:ModuliMoment} below.

Let $(\tC, \tC')\subset V\times V^*$ be positively aligned convex cones.
For simplicity, we assume that $(\tC, \tC')$ are smooth, strictly convex cones.  
We make the following definition.
\begin{defn}
    A \textbf{section} $\Gamma$ of $\del \tC$ is a codimension-$1$ submanifold of $\del \tC\setminus\{0\}$ meeting every ray of $\del\tC$ transversely and exactly once.
\end{defn}
Consider the Fr\'echet manifold
\[
\widehat{\cQ} := \{ (\tC, \Gamma; \tC',\Gamma'): (\tC,\tC')\subset V\times V^*, \text{ and } \Gamma \subset \del C, \Gamma'\subset \del C'\}
\]
where $\Gamma,\Gamma'$ are sections of $\del\tC$ and $\del \tC'$, respectively.
We remark that the $(\tC, \tC')$ data in the definition of $\widehat{\cQ}$ are ancillary; indeed, one can recover $\tC$ (resp.~$\tC'$) from the sections $\Gamma$ (resp.~$\Gamma'$) by taking the cones generated $\Gamma$ (resp.~$\Gamma'$) and the origin.
Let $E$ denote the Euler vector field of the rescaling action on $V$, and $E'$ the Euler vector field on $V^*$. 
For simplicity we also fix compatible inner products and let $\nu_{\tC}$ (resp.~$\nu_{\tC'}$) be a outward pointing normal vector to $\tC$ (resp.~$\tC'$). The sections $\Gamma$ and $\Gamma'$ are oriented by the volume forms
\[
d\sigma_{\Gamma} := \iota_E \iota_{\nu_{\tC}}\Omega_V \qquad \text{and}\qquad d\sigma_{\Gamma'} := \iota_{E'} \iota_{\nu_{\tC'}}\Omega_{V^*}.
\]
Define
\[
\Lambda := \{tX : 0 \leq t \leq 1, X \in \Gamma\} \qquad \text{and} \qquad \Lambda' := \{tY : 0 \leq t \leq 1, Y \in \Gamma'\} .
\]
We can identify the tangent space to $\widehat{\cQ}$ as
\[
T_{(\tC, \Gamma; \tC',\Gamma')}\widehat{\cQ}= C^{\infty}(\Gamma, V/T\Gamma) \oplus C^{\infty}(\Gamma', V^*/T\Gamma').
\]
Let $(\eta, \eta')\in C^{\infty}(\Gamma, V/T\Gamma)\oplus C^{\infty}(\Gamma', V^*/T\Gamma')$. We define a $1$-form on $\widehat{\cQ}$ by
\begin{equation}\label{eqn:ThetaOmegaDefs}
\widehat{\lambda}_{\Gamma,\Gamma'}(\eta,\eta') := \frac{1}{n}\int_{\Gamma} \iota_\eta\iota_E \Omega_V -\frac{1}{n}\int_{\Gamma'}\iota_{\eta'}\iota_{E'} \Omega_{V^*}.
\end{equation}
Then $\Omega = d\lambda$ gives rise to a (formal) symplectic structure on $\widehat{Q}$.
Indeed, explicitly we have
\[
\Omega( (\eta_1,\eta_1'),(\eta_2,\eta_2'))= \int_{\Gamma}\iota_{\eta_2}\iota_{\eta_1}\Omega_{V} - \int_{\Gamma'}\iota_{\eta_2'}\iota_{\eta_1'}\Omega_{V^*}.
\]
Clearly $\widehat{\Omega}$ is closed, and one can check directly that it is non-degenerate, and hence $(\widehat{\cQ}, \widehat{\Omega})$ is an infinite dimensional symplectic manifold.

An action of $\sln$ is given by 
\[
g \cdot (\tC, \Gamma; \tC', \Gamma') = (g\cdot \tC, g \cdot \Gamma, g^{-T}\cdot \tC', g^{-T} \cdot \Gamma').
\]
The symplectic form is preserved under pullback by the action of $\sln$ because, by definition, the action preserves $\Omega_{V},\Omega_{V^*}$. We can find a moment map for this action $\widehat{\mu} : \widehat{\cQ} \to \mathfrak{sl}(n)^*$ given by
\[
\langle \widehat{\mu}(\Gamma, \Gamma'), A \rangle = \lambda_{\Gamma,\Gamma'}(\eta_{A}, \eta'_{A}),
\]
where $\eta_{A}(X)=AX$ and $\eta'_{A}(Y)= -A^{T}Y$ are the infinitesimal generators of the $1$-parameter subgroups $e^{tA}$ and $e^{-tA^{T}}$.
Substituting the formula for $\lambda$ yields
\[
\la \widehat{\mu}, A\rg = \frac{1}{n} \int_\Gamma \la AX, \nu_{\tC}\rg\,d\sigma_\Gamma - \frac{1}{n} \int_{\Gamma'} \la \nu_{\tC'}, A^{T}Y\rg\,d\sigma_{\Gamma'}
\]
which, by the coarea formula, we can express as
\[
    \int_{\Lambda}\la AX, \nu_{\tC}\rg\,d\cH^{n-1}= \int_{\Gamma} \int_0^1t \la AX, \nu_{\tC}\rg t^{n-2}\,dt\,d\sigma_\gamma = \frac{1}{n} \int_\Gamma \la AX, \nu_{\tC}\rg \,d\sigma_\Gamma,
\]
and similarly on $V^*$ for $\Gamma'$ and $\Lambda'$. 
Therefore, we also have
\[
\widehat{\mu} = \int_\Lambda X \otimes \nu_{\tC}\,d\cH^{n-1} - \int_{\Lambda'} \nu_{\tC'}\otimes Y\,d\cH^{n-1}.
\]

Let $\varphi \in \cH_{(\tC, \tC')}$ satisfy $\nabla \varphi(\tC) = \tC'$.
Define a submanifold $\cQ \subset \widehat{\cQ}$ by
\[
\cQ := \{(\tC, \tC',\varphi) : \varphi \in \cH_{(\tC, \tC')}, \nabla \varphi(\tC) = \tC'\}
\]
defining $\Gamma, \Gamma'$ from $\varphi$ from
\[
\Gamma_{\varphi}= \{\varphi=1\}\cap \del\tC \qquad\text{and}\qquad  \Gamma'_{\varphi^*}= \{\varphi^*=1\}\cap \del\tC'.
\]
Let $\cX := \{ (\tC,\tC'): (\tC,\tC')\text{ are positively aligned convex cones}\}$ and denote by $p : \cQ \to \cX$ the projection map forgetting the section data. We have the maps $\cX \xleftarrow{p}\cQ \xhookrightarrow{\iota}\widehat{\cQ}$.

Evaluating the moment map at a point $(\Gamma_{\varphi},\Gamma_{\varphi^*}')$ yields the formula
\begin{equation}\label{eqn:muvarphidef}
\mu_\varphi:= \widehat{\mu}(\Gamma_{\varphi},\Gamma_{\varphi^*}')= 
 \int_{\Lambda_{\varphi}} X \otimes \nu_\tC\, d\cH^{n-1} - \int_{\Lambda_{\varphi^*}'} \nu_{\tC'} \otimes Y\, d\cH^{n-1},
\end{equation}
where $\Lambda_{\varphi} := \ol{\{\varphi < 1\}} \cap \p \tC$ and $\Lambda'_{\varphi^*} := \ol{\{\varphi^* < 1\}} \cap \p \tC'$.

\begin{lem}\label{lem:muAcontraction}
    Let $\varphi \in \cH_{(\tC,\tC')}$ satisfy $\nabla \varphi(\tC) = \tC'$. 
    For any $A \in \mathfrak{sl}(V)$, if $\nabla \varphi (\tC) = \tC'$, we have 
    \[
    \la \mu_\varphi, A \rg = \int_{\{\varphi = 1\}}(\det D^2 \varphi -1)\frac{\la AX, \nabla \varphi\rg }{|\nabla \varphi|} \,d\cH^{n-1}.
    \]
    In particular, if $\varphi$ solves~\eqref{eqn:HOT}, then $\mu_\varphi = 0$. 
\end{lem}

\begin{proof}
    
    By convexity, the sublevel sets of $\varphi, \varphi^*$ have finite perimeters.
    We can consider the outward pointing normal vectors $\nu$ and $\nu'$ of the reduced boundaries $\p^* \tC$ and $\p^*\tC'$ which are defined $\cH^{n-1}$-almost everywhere.
    Let $S := \{\varphi = 1\}$ and $S' := \{\varphi^* = 1\}$ be the interior boundaries and $\Gamma := \p \tC \cap \ol{K}$ and $\Gamma' := \p\tC' \cap \ol{2K'}$ be the cone boundary components, where we recall $K := \{\varphi < 1\}$ and $K' := \{\varphi^* < 1/4\} \cap \tC'$.
    
    Since $A$ is traceless, we know $\mr{div}(AX) = 0$, so the divergence theorem shows
    \[
    \int_S \la AX, \nu_S\rg\,d\cH^{n-1} + \int_\Lambda\la AX, \nu\rg\,d\cH^{n-1} = \int_{K}\mr{div}(AX)\, dX= 0,
    \]
    and likewise
    \[
    \int_{S'} \la  \nu_{S'}, -A^TY\rg\,d\cH^{n-1} + \int_{\Lambda'}\la \nu',-A^TY\rg\,d\cH^{n-1} = \int_{2K'} \mr{div}(-A^TY)\,dY= 0.
    \]
    Therefore, we have
    \[
    \la \mu_\varphi, A\rg =  -\int_S\la AX,\nu_S\rg\,d\cH^{n-1} +
    \int_{S'}\la \nu_{S'}, A^TY\rg\,d\cH^{n-1}.
    \]
    
    By Legendre duality, we may apply the coordinate change $Y = \nabla \varphi(X)$ whose Jacobian is $\det D^2\varphi$, so $d\cH^{n-1}(Y) =\frac{|X|}{|Y|} \det D^2\varphi(X)\, d\cH^{n-1}(X)$.
    Since $S, S'$ are level sets of $\varphi, \varphi^*$, the normals are expressible as $\nu_S = \frac{Y}{|Y|}, \nu_{S'} = \frac{X}{|X|}$.
    Therefore, we may compute
   \begin{align*}
    \int_{S'}\la \nu_{S'}, A^TY\rg\,d\cH^{n-1}(Y)  &=  \int_{S'} \frac{\la X, A^TY\rg}{|X|}  \,d\cH^{n-1}(Y)\\
    &= \int_S \det D^2\varphi(X) \frac{\la X, A^TY\rg}{|Y|} \,d\cH^{n-1}(X) \\
    &= \int_S \det D^2\varphi(X) \frac{\la AX,Y\rg}{|Y|} \,d\cH^{n-1}(X).
    \end{align*}
    Since $\la AX,\nu_S\rg = \frac{\la AX,Y\rg}{|Y|}$, we have
    \[
        \la\mu_\varphi,A\rg = \int_S (\det D^2\varphi-1) \frac{\la AX,\nabla\varphi(X)\rg} {|\nabla\varphi(X)|} \,d\cH^{n-1}
    \]
    as claimed. 
    If $\varphi$ solves~\eqref{eqn:HOT}, then $\det D^2 \varphi = 1$, so $\mu_\varphi \equiv 0$. 
   
\end{proof}

We make the following definition:

\begin{defn}\label{def:nesting}
A pair of convex cones $(\tC,\tC')$ is {\bf nesting} along a family $e^{tA}$ for some non-zero $A \in \mathfrak{sl}(n)$ if the pair satisfies 
\[
\tC \subset e^{tA} \cdot \tC \qquad \text{and}\qquad  \tC'\subset e^{-tA^{T}} \cdot \tC' \qquad \text{for all }t > 0.
\]
\end{defn}

By Lemma~\ref{lem:muAcontraction}, nesting obstructs existence of homogeneous optimal transport maps.

\begin{prop}\label{prop:nesting}
    Suppose that $(\tC, \tC')$ are nesting along $e^{tA}$.
    If there exists a solution of~\eqref{eqn:HOT}, then $(e^{tA} \tC, e^{-tA^T}\tC') = (\tC, \tC')$ so $e^{tA} \in \mr{Aut}(\tC, \tC')$.
    In particular, if $\{e^{tA}\} \not\subset \mr{Aut}(\tC, \tC')$, then there does not exist a solution of~\eqref{eqn:HOT}.
\end{prop}
\begin{proof}
    Suppose that some $\varphi$ solving~\eqref{eqn:HOT} exists.
    From Lemma~\ref{lem:muAcontraction}, we know that the moment map $\mu \in \mathfrak{sl}(V)^*$ must vanish identically. 
    The nesting property demonstrates that
    \[
    \la AX, \nu\rg \geq 0\qquad \text{and} \qquad \la \nu', -A^TY\rg \geq 0
    \]
    for $\cH^{n-1}$ generic $X \in \p^* \tC$ and $Y \in \p^* \tC'$.
    Lemma~\ref{lem:muAcontraction} shows that for all $A \in \mathfrak{sl}(V)$
    \[
    \int_\Lambda\la AX, \nu\rg\,d\cH^{n-1}  = \int_{\Lambda'}\la  \nu', A^TY\rg\,d\cH^{n-1}.
    \]
    From nesting, the left-hand side is non-negative and the right-hand side is non-positive, so both must vanish.
    Therefore, we must have $(\tC, \tC') = (e^{tA} \tC, e^{-tA^T}\tC')$ and $\{e^{tA}\} \subset \mr{Aut}(\tC, \tC')$.
\end{proof}

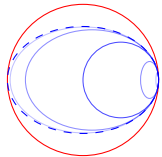
\begin{figure}
    \centering
    \begin{tikzpicture}

    % outer unit circle
    \draw[red, thin] (0,0) circle (1);

    % family of ellipses
    \foreach \tt/\op in {-2/0.25,-1/0.45,0/0.7,1/0.45,2/0.25} {
        \pgfmathsetmacro{\E}{exp(2*\tt)}
        \pgfmathsetmacro{\cx}{\E/(1+\E)}
        \pgfmathsetmacro{\a}{1/(1+\E)}
        \pgfmathsetmacro{\b}{1/sqrt(2*(1+\E))}
        \draw[blue, thin, opacity=\op] (\cx,0) ellipse [x radius=\a, y radius=\b];
    }

    % formal limit as t -> -infinity: x^2 + 2y^2 = 1
    \draw[blue, thin, dashed] (0,0) ellipse [x radius=1, y radius={1/sqrt(2)}];

\end{tikzpicture}
    \caption{The boundaries of the links $Q$ and $\Sigma$ are illustrated in red and blue, whose cones share a supporting hyperplane. The various blue ellipses are the image under the Lorentz boost preserving the red cone with positive weight at $(-1,0,1)$.
    This example fails OT-separability because the limiting link, displayed in dashed blue, admits a \hyperref[eqn:HOT]{H.O.T.} map given by a quadratic.
    Furthermore, this example is nesting (see Definition~\ref{def:nesting}) as seen by the equations of the blue ellipses $x^2 + 2y^2 + e^{2t}(1-x)^2 = 1$. 
    Therefore, there does not exist a \hyperref[eqn:HOT]{H.O.T.} map, (cf.~Proposition~\ref{prop:nesting}).}
    \label{fig:nestingEllipsesBoost}
\end{figure}

\begin{remark}
    In the previous section, we constructed pairs $(\cL, \cL'_{\ve H})$ of \hyperref[eqn:HOT]{H.O.T.} maps between the Lorentz cone and a perturbation of it.
    The obstruction we found there can be expressed as the lack of nesting of the perturbed pair along any boost symmetry of $\cL$.  
\end{remark}

Nesting is preserved under products, so $(\tC_0 \times \tC_1, \tC'_0 \times \tC'_1)$ is nesting if $(\tC_0, \tC'_0)$ is. This produces non-existence of~\hyperref[eqn:HOT]{H.O.T.} maps in every dimension, including examples with arbitrarily sophisticated linking topology.

\begin{example}
    Suppose both $\tC$ and $\tC'$ are unlinked simplicial cones.
    Up to duality, we may assume $\tC^\vee$ contains the ray spanned by a vertex $v_1 \in \Sigma$ and let $F := \mr{ConvexHull}(v_2,\ldots, v_n)$ be its opposing exposed face. 
    Suppose that $\tC'$ does not intersect the plane spanned by $F$.
    Consider the basis $e_i^* := (v_i,1) \in V^*$, then along the $1$-parameter subgroup given by $\mr{diag}(t^{n-1}, t^{-1},\ldots, t^{-1})$, the pair is nesting.
    More generally, only $\tC$ has to be simplicial.
    Other cones with symmetries with similar properties produce further nesting pairs. 

\end{example}

\section{Towards optimal transport and the moduli space of convex cones}\label{sec:ModuliMoment}% Are you HOT or NOT? I'm HOT. 

We now explain a conjectural picture relating optimal regularity of optimal transport to the existence of a moduli space of convex cone pairs admitting solutions of~\eqref{eqn:HOT}. While our discussion considers only the constant density setting, we expect this picture to generalize in a natural way to positive homogeneous densities.
Let
\[
(\cQ, \Omega)=\{(\tC, \tC', \varphi) : \varphi \in \cH_{(\tC, \tC')} \text{ such that } \nabla \varphi(\tC) = \tC'\}
\]
be the subspace of $(\widehat{\cQ},\widehat{\Omega})$ consisting of pairs of positively aligned cones equipped with a positive, convex, $2$-homogeneous function with $\nabla \varphi(\tC) = \tC'$.
The $2$-form $\Omega := \widehat{\Omega}\vert_\cQ$ is the restriction of the symplectic form, which may be degenerate.
In Section~\ref{sec:ObMoment}, we found that the group $\sln$ acts on $(\cQ,\Omega)$ by symplectomorphisms and this action provides a moment map $\mu: \cQ\rightarrow \mathfrak{sl}(V)^*$ given by equation~\eqref{eqn:muvarphidef}.
Let
\[
    \cX= \{ (\tC,\tC') : (\tC, \tC') \text{ are positively aligned} \}. 
\]
By Lemma~\ref{lem:posAlign}, there is a natural $\sln$-equivariant fibration $p:\cQ\rightarrow \cX$ given by forgetting $\varphi$. We make the following definition.

\begin{defn}\label{defn:sympReducStab}
    We say that a point $x\in \cX$ is:
    \begin{itemize} 
    \item {\bf M-semistable} if $\ol{\sln\cdot x} \cap p(\mu^{-1}(0)) \neq \emptyset$.
    We define the semistable locus by $\cX^{ss} :=\{ x\in \cX: x \text{ is M-semistable}\}$.

    \item {\bf M-polystable} if $x$ is M-semi-stable and $\sln \cdot x$ is closed in $\cX^{ss}$.  
    We denote the polystable locus by $\cX^{ps}$.
    
    \item {\bf M-stable} if $x \in \cX^{ps}$ and $\mr{Aut}(\tC, \tC')\subset \sln$ is compact.

    \item {\bf M-unstable} if $x \in \cX\setminus \cX^{ss}$.
    \end{itemize}
\end{defn}

Our conjectural picture relates the construction of a Hausdorff moduli space $\cX^{ps}/\!/\sln$ to the regularity/singularity dichotomy for optimal transport.
A minimal requirement to have any type of moduli theory is to satisfy the (LC) condition of Definition~\ref{defn:LC}.\footnote{The following elementary example was suggested by ChatGPT 5.6-Sol. Let $h_\alpha(\theta) := 1 + \ve \cos(\theta) + \ve \cos(\sqrt{2}\theta + \alpha)$ for some $\alpha \in \bR/2\pi \bZ$ and define $\tC_\alpha:=\{(u,v,z):u,v>0,\ |z|<\sqrt{uv}\,h_\alpha(\log(u/v))\}$
One can construct diverging sequence $M_j, N_j \in \sln$ such that $\tC_\pi = \lim M_j\cdot  \tC_0$ and $\tC_0 = \lim N_j \cdot \tC_\pi$, but $\tC_0$ and $\tC_\pi$ are not $\sln$-equivalent.
Let $Q := \frac{1}{2}|X|^2$, so the pairs $(\tC_0, \nabla Q(\tC_0))$ and $(\tC_\pi, \nabla Q(\tC_\pi))$ have \hyperref[eqn:HOT]{H.O.T.} maps given by $Q$.
One can further construct diverging maps $M_j$ such that $\lim_{j \to \infty} M_j \cdot (\tC_0, \nabla Q(\tC_0)) \to (\tC_0, \nabla Q(\tC_0))$, so the failure of Palais-Smale in Remark~\ref{rmk:C->CDegen} occurs while our method still finds a max-min critical point.\label{footnote:NonLCEx}}
Consider the following local optimal transport problem.

\begin{defn}\label{def:local}
Let $(\Omega, \Omega')\subset V\times V^*$ be a pair of open convex sets with $0\in \del \Omega$ and $0\in \del \Omega'$.
A solution of the {\bf local optimal transport problem} consists of the following data:
\begin{itemize}
\item open neighborhoods $N\subset V$ and $N^*\subset V^*$ with $0\in N$, and $0\in N^*$, and
\item a convex function $u:N\cap \Omega\rightarrow \bR$ solving the local optimal transport problem
\begin{equation}\label{eq: localOTmap}\tag{$\textup{LOT}$}
\det D^2u=1, \quad \nabla u(N\cap \Omega)=N^*\cap \Omega',\quad \nabla u(0)=0.
\end{equation}
\end{itemize}
\end{defn}

In the following discussion, we assume for simplicity that $\Omega$ and $\Omega'$ are locally $C^{1,\alpha}$-epigraphs over their tangent cones.
We state the following guiding conjecture.

\begin{conj}\label{conj:Master}
Let $(\Omega, \Omega')\subset V\times V^*$ be a pair of open convex sets with $0\in \del \Omega$ and $0\in \del \Omega'$.
If $(\tC,\tC')$ denotes the tangent cones of $(\Omega,\Omega')$ at $(0,0)$ and satisfies $(\mr{LC})$, then:
    
     \smallskip $(\mathbf{A}):$ There exists a solution of the local optimal transport problem~\eqref{eq: localOTmap} if and only if $(\tC,\tC')$ is M-semistable.

    \smallskip $(\mathbf{B}):$  If $(\tC,\tC')$ is M-semistable, then it admits a polystable degeneration.

    \smallskip $(\mathbf{C}):$ The pair $(\tC,\tC')$ admits a solution of~\eqref{eqn:HOT} if and only if $(\tC,\tC')$ is M-polystable.

    \smallskip $(\mathbf{D}):$ Assume there exists a solution of the local optimal transport problem~\eqref{eq: localOTmap}.  
    Then:
\begin{itemize}
    \item[$(i)$] $u$ is $C^{1,1}$ at $0\in \del \Omega$, and $u^*$ is $C^{1,1}$ at $0\in \del \Omega'$ if and only if the pair $(\tC,\tC')$  is M-polystable.
    \item[$(ii)$] Furthermore, if $u, u^*$ are $C^{1,1}$, then
    \begin{equation}\label{eq: convToBlowUp}
    u(x)= u_{\infty}(x) + o(|x|^2)
    \end{equation}
    where $u_{\infty}(x)$ is a solution of~\eqref{eqn:HOT} on $(\tC,\tC')$.
    \end{itemize}
    
\end{conj}
Furthermore, it is an interesting question if the automorphism group $\mr{Aut}(\tC, \tC')$ of a polystable pair is necessarily reductive; this is the case for pairs of of the form $(\tC,\tC)$ (with the identity map), and for polystable pairs in dimension 2 as explained below.
We now explain how the monotonicity formula implies some parts of Conjecture~\ref{conj:Master}.  

\begin{lem}\label{lem: localOTsemistable}
In the setting of Conjecture~\hyperref[conj:Master]{\ref{conj:Master} $(\mathbf{A})$}, suppose that there exists a solution of the local optimal transport problem~\eqref{eq: localOTmap}. If $(\tC, \tC')\in \cX$ denote the tangent cones to $(\Omega,\Omega')$ at $(0,0)$, then the pair $(\tC, \tC')$ is M-semistable.
\end{lem}
\begin{proof}
    This is a consequence of the monotonicity formula and Lemma~\ref{lem:muAcontraction}.
    Indeed, by \cite[Theorem 4.1]{TristanFreid} there exists a sequence of linear transformations $M_j \in {\rm SL}(V)$ such that $(M_j\tC, M_j^{-T}\tC')$ converge in the Hausdorff sense to a pair of convex cones $(\tC_{\infty}, \tC'_{\infty})$ admitting a homogeneous optimal transport map.
    By Lemma~\ref{lem:muAcontraction}, $(\tC_{\infty}, \tC'_{\infty})\in p(\mu^{-1}(0))$, and hence $(\tC,\tC')$ is M-semistable in the sense of Definition~\ref{defn:sympReducStab}.
\end{proof}

\begin{remark}
    For the purposes of our conjectural picture, the reader may safely assume that $(\Omega, \Omega')$ are precisely conical near $(0,0)$ in Conjecture~\hyperref[conj:Master]{\ref{conj:Master} $(\mathbf{A})$}.
\end{remark}

We note that the argument of Lemma~\ref{lem: localOTsemistable} yields the following direct corollary.

\begin{cor}
    Assuming Conjecture~\hyperref[conj:Master]{\ref{conj:Master} $(\mathbf{A})$} and \hyperref[conj:Master]{$(\mathbf{C})$}, if $x:=(\tC,\tC')\in \cX$ is M-semistable, then
    \[
    \overline{\sln\cdot x} \cap \cX^{ps}\ne \emptyset.
    \]
    That is, every semi-stable point $x\in \cX$ admits a polystable degeneration.
\end{cor}

From the moduli point of view, it is a natural question if the putative polystable degeneration in Conjecture~\hyperref[conj:Master]{\ref{conj:Master} $(\mathbf{B})$} is unique.
This conjecture is a type of uniqueness statement for blow-ups of optimal transport maps in the sense of \cite{TristanFreid}.
It amounts to the statement that there exists a pair of convex cones $(\tC_{\infty},\tC'_{\infty})$ admitting a solution of~\eqref{eqn:HOT} such that, for any solution $u$ of the local optimal transport problem~\eqref{eq: localOTmap}, the blow-up of $u$ is a homogeneous optimal transport map defined on $(\tC_{\infty},\tC'_{\infty})$.
Note that we do not claim that all solutions of~\eqref{eq: localOTmap} have the same blow-up, only that any blow-up is defined on the same underlying pair of convex cones.

The following lemma explains the relationship between M-polystability and the notion of OT-separability.

\begin{lem}\label{lem: OTsep-Mstab}
Let $(\tC, \tC') \in \cX$ satisfy condition $(\mr{LC})$.
\begin{itemize}
    \item[$(i)$] If $(\tC, \tC')$ is M-polystable, then $(\tC,\tC')$ is OT-separable.
    \item[$(ii)$] Assuming Conjecture~\hyperref[conj:Master]{\ref{conj:Master} $(\mathbf{A})$}, if $(\tC,\tC')$ is M-semistable and OT-separable, then $(\tC,\tC')$ is M-polystable.
\end{itemize}
\end{lem}
\begin{proof}
We first prove $(i)$.
Suppose $(\widehat{\tC}, \widehat{\tC}') \in \overline{\sln \cdot(\tC,\tC')}$ admits a solution of~\eqref{eqn:HOT}.
By Lemma~\ref{lem:muAcontraction}, $(\widehat{\tC}, \widehat{\tC}')$ is M-semistable.
Since we are assuming $(\tC,\tC')$ is  M-polystable, we conclude that $(\widehat{\tC}, \widehat{\tC}')\in \sln\cdot (\tC,\tC')$.
Thus, $(\tC,\tC')$ is OT-separable.

We now prove $(ii)$.
Suppose $(\widehat{\tC}, \widehat{\tC}') \in \overline{\sln \cdot(\tC,\tC')}$ is M-semistable.
Since we are assuming Conjecture~\hyperref[conj:Master]{\ref{conj:Master} $(\mathbf{A})$}, the pair $(\widehat{\tC},\widehat{\tC}')$ admits a solution of the local optimal transport problem~\eqref{eq: localOTmap}.
By the monotonicity formula, there is a pair of cones $(\tC_{\infty},\tC_{\infty}')\in \overline{\sln \cdot(\widehat{\tC}, \widehat{\tC}')}$ admitting a solution of~\eqref{eqn:HOT}.
By a diagonal argument we see that
\[
(\tC_{\infty},\tC_{\infty}')\in \overline{\sln \cdot(\tC, \tC')}.
\]
The OT-separability condition implies $(\tC_{\infty},\tC_{\infty}')\in \sln \cdot(\tC, \tC')$.
By condition $(\mr{LC})$, this implies $(\widehat{\tC},\widehat{\tC}')\in \sln \cdot(\tC, \tC')$ and therefore $(\tC,\tC')$ is M-polystable.
\end{proof}

We can prove one direction of Conjecture~\hyperref[conj:Master]{\ref{conj:Master} $(\mathbf{D})$} part $(i)$ using the monotonicity formula.  
Suppose that $u$ solves the local optimal transport problem~\eqref{eq: localOTmap}.  
If the tangent cones $(\tC,\tC')$ to $(\Omega,\Omega')$ at $(0,0)$ is M-stable, then $(\tC, \tC')$ is OT-separable.
Therefore, from Proposition~\ref{prop:OTsep->round}, $u$ is $C^{1,1}$ at $0\in \del \Omega$, and $u^*$ is $C^{1,1}$ at $0\in \del \Omega'$.

\begin{remark}
Conjecture~\hyperref[conj:Master]{\ref{conj:Master} $(\mathbf{D})$} part $(ii)$ is of a different nature.  Given a solution $u$ of~\eqref{eq: localOTmap}, once one knows the existence of some blow-up $u_{\infty}$ of $u$ (in the sense of \cite{TristanFreid}) defined on the same cones $(\tC,\tC')$, the asymptotics~\eqref{eq: convToBlowUp} amount to a uniqueness of blow-ups result.
\end{remark}

Conjecture~\hyperref[conj:Master]{\ref{conj:Master} $(\mathbf{A})$}, together with \cite{TristanFreid} easily implies one direction of Conjecture~\hyperref[conj:Master]{\ref{conj:Master} $(\mathbf{D})$}.

\begin{lem}
    Suppose $(\tC,\tC')$ is a pair of M-semistable convex cones.  Assuming Conjecture~\hyperref[conj:Master]{\ref{conj:Master} $(\mathbf{A})$}, if $(\tC,\tC')$ is M-polystable, then $(\tC,\tC')$ admits a solution of~\eqref{eqn:HOT}.
\end{lem}
\begin{proof}
    Since $(\tC,\tC')$ is a pair of M-semistable cones, and we are assuming Conjecture~\hyperref[conj:Master]{\ref{conj:Master} $(\mathbf{A})$}, we deduce the existence of a solution of the local optimal transport problem~\eqref{eq: localOTmap}.
    Applying the monotonicity formula, we find a pair of cones $(\tC_{\infty},\tC'_{\infty})\in \overline{\sln\cdot (\tC,\tC')}$ admitting a solution of~\eqref{eqn:HOT}.
    On the other hand, since $(\tC,\tC')$ is M-polystable, it is OT-separable by Lemma~\ref{lem: OTsep-Mstab} part $(i)$.
    Therefore $(\tC_{\infty},\tC'_{\infty})\in \sln\cdot (\tC,\tC')$, and the result follows. 
\end{proof}

\subsection{Moduli space in dimension two}\label{sec:modulin=2}
In this section we prove Conjecture~\hyperref[conj:Master]{\ref{conj:Master} $(\mathbf{A})$}, \hyperref[conj:Master]{$(\mathbf{B})$}, and \hyperref[conj:Master]{$(\mathbf{C})$} for cones in $\bR^2$.  

The classification of cones in $\bR^2$ admitting homogeneous optimal transport maps was carried out in \cite{TristanFreid}; we recall some aspects here and explain how the classification fits into the conjectural picture described above.
In $\bR^2$, every convex cone is ${\rm SL}(2)$ equivalent to either a half-space or an orthant. 
By the splitting theorem, Theorem~\ref{thm:splitting}, we know that for an optimal transport map to exist, either both cones are half-spaces or neither is a half-space.
In the mixed half-space/pointed cone case, the pair is nesting and hence obstructed by Proposition~\ref{prop:nesting}.
In the half-space to half-space case, a~\hyperref[eqn:HOT]{H.O.T.} solution exists whenever the pair is positively aligned. 
\begin{figure}
\centering
\begin{tikzpicture}[
    x=1cm,y=1cm,
    line cap=round,
    Cprimeedge/.style={blue!70!black, line width=.45pt},
    Cdualedge/.style={red!75!black, line width=.45pt},
    Cdualarrow/.style={red!75!black, line width=.9pt, -{Latex[length=2mm]}}
]

\def\W{1.8}
\def\H{1.35}
\def\R{1.25}
\def\RR{1.7}

% #1 = xshift
% #2 = yshift
% #3 = label
\newcommand{\halfspacebase}[3]{%
\begin{scope}[shift={(#1,#2)}]
    % C' = upper half-space y>0, shown in a truncated window
    \path[fill=blue, fill opacity=.22]
        (-\W,0) -- (\W,0) -- (\W,\H) -- (-\W,\H) -- cycle;

    % boundary line y=0
    \draw[Cprimeedge] (-\W,0) -- (\W,0);

    % origin
    \fill (0,0) circle (1pt);

    % subfigure label
    \node at (0,-0.32) {\scriptsize #3};
\end{scope}
}

% ---------------------------------------------------
% (a) C^\vee is a single ray
% ---------------------------------------------------
\begin{scope}[shift={(0,0)}]
    % blue upper half-space
    \path[fill=blue, fill opacity=.22]
        (-\W,0) -- (\W,0) -- (\W,\H) -- (-\W,\H) -- cycle;
    \draw[Cprimeedge] (-\W,0) -- (\W,0);
    \fill (0,0) circle (1pt);

    % single vector / ray
    \draw[Cdualarrow] (0,0) -- (68:\R);

    \node at (0,-0.32) {\scriptsize (a)};
\end{scope}

% ---------------------------------------------------
% (b) C^\vee is a strict cone inside the upper half-space
% ---------------------------------------------------
\begin{scope}[shift={(4.2,0)}]
    % blue upper half-space
    \path[fill=blue, fill opacity=.22]
        (-\W,0) -- (\W,0) -- (\W,\H) -- (-\W,\H) -- cycle;
    \draw[Cprimeedge] (-\W,0) -- (\W,0);
    \fill (0,0) circle (1pt);

    % strict cone whose closure lies in y>0
    \path[fill=red, fill opacity=.22]
        (0,0) -- (35:\RR) -- (145:\RR) -- cycle;

    \draw[Cdualedge] (0,0) -- (35:\RR);
    \draw[Cdualedge] (0,0) -- (145:\RR);

    \node at (0,-0.32) {\scriptsize (b)};
\end{scope}

\end{tikzpicture}
\caption{We show the two positively aligned cases where $\tC'$, illustrated in blue, is a half-space which we may assume is the upper half-space.
Case (a) shows the scenario in which $\tC$ is also a half-space, so its dual, in red, is a line in the interior given by $\{t Y_0 : t > 0\}$.
This case is M-polystable, preserved by the symmetry that scales $Y_0$ and the $x$-axis inversely to remain in $\sln$.
Case (b) shows where $\tC$ is a strict cone, which is strictly M-semistable, and for any $Y_0 \in \tC^\vee$, the same 1-parameter family above realizes the polystable degeneration as $\tC^\vee \leadsto \{tY_0 : t > 0\}$, which is the degeneration to (a), where a \hyperref[eqn:HOT]{H.O.T.} map exists.}
\label{fig:n=2halfspace}
\end{figure}
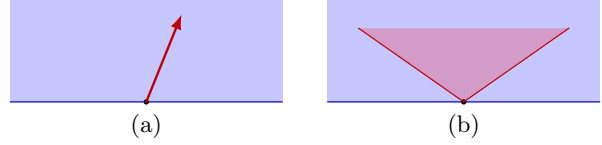

We first consider the case where at least one cone is a half-space; so let this be $\tC'$, which we may assume is the upper half-space, so the two cases of positively aligned cones are illustrated in Figure~\ref{fig:n=2halfspace}:
\begin{itemize}
    \item $\tC$ is a half-space: its normal lies in the interior of $\tC'$, Figure~\hyperref[fig:n=2halfspace]{\ref{fig:n=2halfspace}(a)}.
    \item $\tC$ is a pointed cone: its dual must fully lie in the interior, Figure~\hyperref[fig:n=2halfspace]{\ref{fig:n=2halfspace}(b)}.
    After a shear, it can be assumed to contain the point $(0,1)$, so this pair is nesting, for example along $\mr{diag}(e^t, e^{-t})$ in the standard coordinates, and degenerates to Figure~\hyperref[fig:n=2halfspace]{\ref{fig:n=2halfspace}(a)}.
\end{itemize}
Local solutions realizing Figure~\hyperref[fig:n=2halfspace]{\ref{fig:n=2halfspace}(b)} arise, for example, by taking optimal transport maps from a square to a triangle. 

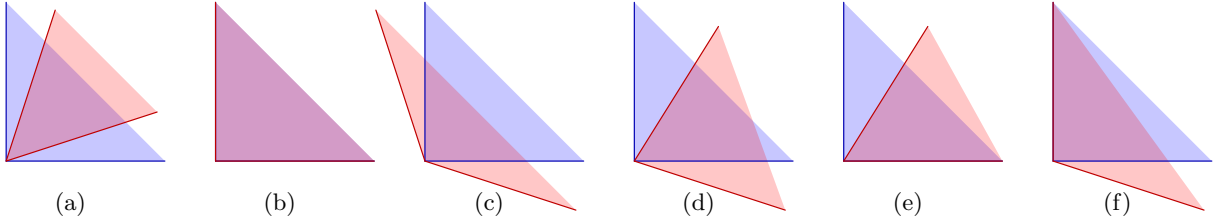
\begin{figure}
\centering
\begin{tikzpicture}[
    x=1cm,y=1cm,
    line cap=round,
    Cedge/.style={blue!70!black, line width=.45pt},
    Cpedge/.style={red!75!black, line width=.45pt}
]

\def\R{2.1}

% #1 = xshift
% #2 = yshift
% #3 = lower angle for C'
% #4 = upper angle for C'
% #5 = label
\newcommand{\conepair}[5]{%
\begin{scope}[shift={(#1,#2)}]

    % Fixed cone C = first quadrant
    \path[fill=blue, fill opacity=.22]
        (0,0) -- (\R,0) -- (0,\R) -- cycle;

    % Variable cone C'
    \path[fill=red, fill opacity=.22]
        (0,0) -- (#3:\R) -- (#4:\R) -- cycle;

    % Boundary rays of C
    \draw[Cedge] (0,0) -- (\R,0);
    \draw[Cedge] (0,0) -- (0,\R);

    % Boundary rays of C'
    \draw[Cpedge] (0,0) -- (#3:\R);
    \draw[Cpedge] (0,0) -- (#4:\R);

    % subfigure label
    \node at (.85,-0.57) {\scriptsize #5};

\end{scope}
}

% One-row layout, 10% bigger than previous version
\conepair{0.00}{0}{18}{72}{(a)}
\conepair{2.77}{0}{0}{90}{(b)}
\conepair{5.54}{0}{-18}{108}{(c)}
\conepair{8.31}{0}{-18}{58}{(d)}
\conepair{11.08}{0}{0}{58}{(e)}
\conepair{13.85}{0}{-18}{90}{(f)}

\end{tikzpicture}

\caption{We display all possible pairs of cones in $\bR^2$.
Up to $\sln$, we may assume that $\tC^\vee$, shown in blue, is the first quadrant and the target $\tC'$ is shown in red.
The first three cones admit~\hyperref[eqn:HOT]{H.O.T.} maps and are (a) acute, (b) M-polystable where $\tC^\vee = \tC'$, and (c) obtuse.
The pairs (a) and (c) are M-stable.
Cone (d) is M-unstable and (e) and (f) are the two strict M-semistable pairs.
We note that all of (d)--(f) are nesting along the 1-parameter subgroup $\mr{diag}(e^{-t}, e^t)$.
The pairs (e) and (f) degenerate to (b) along this family as $t \to \infty$ and $t \to -\infty$ respectively.}\label{fig:2dConePairs}
\end{figure}

We now assume both cones are pointed.  Using the action of ${\rm SL}(2)$, we may assume that $\tC^\vee$ is the first quadrant.
The target cone $\tC'$ is defined by its outer edges, which can lie inside, outside, or share an edge with the boundary of $\tC^\vee$, and these possibilities are represented schematically in Figure~\ref{fig:2dConePairs}.
Therefore, we let $\tC'$, be the cone generated by $(1,a)$ and $(b,1)$.  We will compute the moment map below, but let us summarize the conclusions of the analysis.
\bigskip
\begin{itemize}
    \item {\bf The M-polystable cases:} In Figure~\ref{fig:2dConePairs}, we have \hyperref[fig:2dConePairs]{(a)} where both $a,b > 0$, \hyperref[fig:2dConePairs]{(b)} where $a = b = 0$, and \hyperref[fig:2dConePairs]{(c)} where $a,b < 0$ as the three cases where a solution exists \cite{TristanFreid}. 

    \item {\bf The M-unstable cases:} In Figure~\ref{fig:2dConePairs} case \hyperref[fig:2dConePairs]{(d)}, $a < 0$ and $b > 0$, is M-unstable.
    This case is nesting and hence no homogeneous optimal transport map exists by Proposition~\ref{prop:nesting}.  Furthermore, there is no $\mr{SL}(2)$-degeneration of the pair $(\tC,\tC')$ to a pair of positively aligned half-spaces, or to a pair of M-polystable cones.  Hence there is no solution of~\eqref{eq: localOTmap}.  Similar analysis applies to the case $a>0$ and $b<0$.

    \item {\bf The M-semistable cases:} In Figure~\ref{fig:2dConePairs} when exactly one of $a$ or $b$ vanishes, we get the M-semistable cases of \hyperref[fig:2dConePairs]{(e)} and \hyperref[fig:2dConePairs]{(f)}.
These final three cases are nesting, so solutions are obstructed by Proposition~\ref{prop:nesting}.  
On the other hand, cases \hyperref[fig:2dConePairs]{(e)} and \hyperref[fig:2dConePairs]{(f)} admit unique $M$-polystable degenerations given by $M_t := \mr{diag}(e^t, e^{-t})$ acting on $V^*$. 
As $t \to -\infty$, case \hyperref[fig:2dConePairs]{(e)} degenerates to \hyperref[fig:2dConePairs]{(b)} and as $t \to +\infty$, case \hyperref[fig:2dConePairs]{(f)} degenerates to \hyperref[fig:2dConePairs]{(b)}. 
\end{itemize}

Let us compute the moment map. 
Let $\Lambda_{\varphi}$ be $\{te_1 : 0 \leq t \leq r_1\} \cup \{te _2 : 0 \leq t \leq r_2\}$ representing $\ol{\{\varphi < 1\}} \cap \tC$.
Let $X_1 := r_1 e_1$ and $X_2 := r_2 e_2$. 
By duality, we have $Y_1 := \frac{2}{r_1}(1,a)$ and $Y_2 := \frac{2}{r_2} (b,1)$ given by the boundary of $\ol{\{\varphi^* < 1\}} \cap \p \tC'$.
Using that the outward normals are $-e_2^*$ and $-e_1^*$ on the two components of $\Lambda$, we compute
\[
\int_\Lambda X \otimes \nu_{\tC}\, d\cH^1 = -\int_0^{r_1} te_1 \otimes e_2^*\,dt - \int_0^{r_2}te_2 \otimes e_1^*\,dt = -\frac{1}{2} \begin{pmatrix}
    0 & r_1^2\\ r_2^2 & 0
\end{pmatrix}.
\]
On the dual side, the normal vectors are $\frac{(a,-1)}{\sqrt{1 + a^2}}$ and $\frac{(-1,b)}{\sqrt{1 + b^2}}$.
Since $Y(t) = t(1,a)$, we have $|Y'(t)| = \sqrt{1 + a^2}dt$ and similarly for $Y(t) = t(b,1)$.
Therefore, along $Y(t)$ we have $\nu_{\tC'}\,d\cH^{n-1} = (a,-1)\,dt$ and similarly along the other boundary.
From this, we can compute
\[
\int_{\Lambda'} \nu_{\tC'}\otimes Y\,d\cH^1 =\int_0^\frac{2}{r_1}(a,-1)\otimes t(1,a)\, dt + \int_0^\frac{2}{r_2}(-1,b)\otimes t(b,1)\, dt= \frac{2}{r_1^2}\begin{pmatrix}
    a & a^2 \\ -1 & -a
\end{pmatrix} + \frac{2}{r_2^2}\begin{pmatrix}
    -b & -1\\
    b^2 & b
\end{pmatrix}.
\]
From equation~\eqref{eqn:muvarphidef}, we therefore have
\[
\mu_\varphi =\int_{\Lambda} X\otimes \nu_{\tC}\,d\cH^1 - \int_{\Lambda'} \nu_{\tC'}\otimes Y\,d\cH^1  = \begin{pmatrix}
    -\frac{2a}{r_1^2} + \frac{2b}{r_2^2} & -\frac{r_1^2}{2} - \frac{2a^2}{r_1^2} + \frac{2}{r_2^2}\\
    -\frac{r_2^2}{2} + \frac{2}{r_1^2} - \frac{2b^2}{r_2^2} & \frac{2a}{r_1^2} -\frac{2b}{r_2^2}
\end{pmatrix}
\]
which vanishes exactly when $ar_2^2 = br_1^2$ and $r_1^2r_2^2 = 4(1 - ab)$. 
Recalling that we have the flexibility to choose $r_1$ and $r_2$ to satisfy the above equations, we see that the moment map can only vanish if either $ab > 0$ (and then we choose positive $r_1$ and $r_2$ to solve the system) or $a = b = 0$ (which imposes $r_1r_2 = 2$).
These cases correspond exactly to the $M$-polystable cases in the above classification.  Alternatively, by \cite{TristanFreid}, these pairs of cones exactly correspond to the pairs of cones admitting homogeneous optimal transport maps.
This proves Conjecture~\hyperref[conj:Master]{\ref{conj:Master} $(\mathbf{C})$}.

\begin{figure}
    \centering
    \setlength{\tabcolsep}{0pt}
    \begin{tabular}{@{}c@{\hspace{1.6cm}}c@{}}
       \includegraphics[height=2cm]{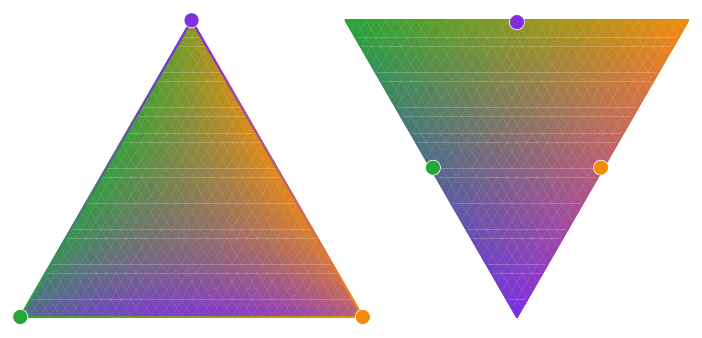}  &
       \includegraphics[height=2cm]{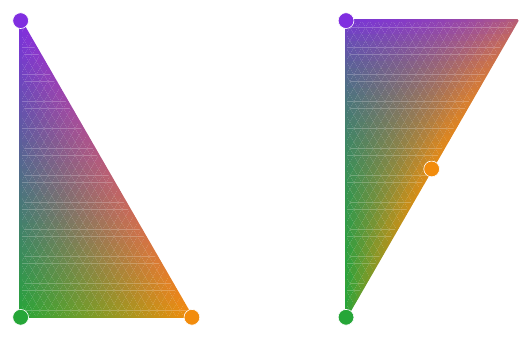} \\
       (a) & (b)
    \end{tabular}
    \caption{Figure (a) shows an optimal transport map from an equilateral triangle to a rotated copy, both with unit densities.
    The image of the optimal transport is illustrated by the matching color gradient. By the symmetries and uniqueness of~\eqref{eqn:OT}, the corners map to the midpoints as marked.
    This realizes the M-semistable case of a strict cone to a half-space.
    If we slice this in half, we get (b), which has the other M-semistable local solution between two pointed cones where $\tC'$ and $\tC^\vee$ share one boundary edge.}
    \label{fig:semistable}
\end{figure}
The completion of the conjecture follows from Lemma~\ref{lem:AcuteObtuseNodegen} and Proposition~\ref{prop:OTsep->round}.
Every M-semistable model case presented in Figure~\ref{fig:2dConePairs} can be realized as a local solution, in the sense of Definition~\ref{def:local}, by maps between triangles with unit density. 
For example, the M-semistable case is given by mapping a 30-60-90 triangle to a reflection of itself; see Figure~\ref{fig:semistable}.
The strictly semistable pairs \hyperref[fig:2dConePairs]{(e)} and \hyperref[fig:2dConePairs]{(f)} in Figure~\ref{fig:2dConePairs} admit unique polystable degenerations to case \hyperref[fig:2dConePairs]{(b)}, a pair $(\tC_\infty, \tC'_\infty)$ of pointed cones with $\tC_\infty^\vee = \tC'_\infty$. On the other hand, the semi-stable situation in which a pointed cone admits a local optimal transport map to a half-space, as depicted in \hyperref[fig:semistable]{\ref{fig:semistable}(a)}, admits a unique polystable degeneration to a pair of positively aligned half-spaces.

\end{document}